\documentclass{amsart}

\usepackage{hyperref}
\usepackage{tabularx} 

\usepackage{esint}
\usepackage{amssymb,centernot}
\usepackage{tikz}
\usepackage{graphicx}
\usepackage{amsrefs,enumitem,pgfkeys}
\usepackage{soul}
\usepackage{dirtytalk}

\usepackage{MnSymbol,wasysym}
\usepackage{fancyhdr}
\usepackage{time}
 
\newcommand\restr[2]{{
  \left.\kern-\nulldelimiterspace 
  #1 
  \vphantom{\big|} 
  \right|_{#2} 
  }}

\usepackage{comment}
\newtheorem{theorem}{Theorem}
\newtheorem{lemma}[theorem]{Lemma}
\newtheorem{corollary}[theorem]{Corollary}
\newtheorem{proposition}[theorem]{Proposition}
\theoremstyle{definition}

\newtheorem{remark}[theorem]{Remark}

\newtheorem{example}[theorem]{Example}
\newtheorem{definition}[theorem]{Definition}

\newcommand{\eref}[1]{(\ref{e.#1})}
\newcommand{\tref}[1]{Theorem \ref{t.#1}}
\newcommand{\lref}[1]{Lemma \ref{l.#1}}
\newcommand{\pref}[1]{Proposition \ref{p.#1}}
\newcommand{\cref}[1]{Corollary \ref{c.#1}}

\newcommand{\sref}[1]{Section \ref{s.#1}}

\newcommand{\rref}[1]{Remark \ref{r.#1}}

\numberwithin{theorem}{section}
\numberwithin{equation}{section}

\newcommand{\Na}{\mathcal{N}}
\newcommand{\N}{\mathbb{N}}
\newcommand{\Z}{\mathbb{Z}}
\newcommand{\R}{\mathbb{R}}

\newcommand{\grad}{\nabla}
\newcommand{\Dupo}{\partial\{u>0\}}

\usepackage{mathtools}

\newcommand{\pt}{\partial}
\newcommand{\bB}{\overline{B}}

\newcommand{\bca}{\begin{cases}}
\newcommand{\eca}{\end{cases}}

\newcommand{\lb}{\left(}

\newcommand{\rb}{\right)}
\newcommand{\lmb}{\left[}
\newcommand{\rmb}{\right]}
\newcommand{\lma}{\left\{}
\newcommand{\rma}{\right\}}

\newcommand{\Ca}{{\mathcal{C}}}

\newcommand{\Ta}{{\mathcal{T}}}

\newcommand{\csubset}{\subset\joinrel\subset}

\newcommand{\Tc}{\mathcal{T}}

\newcommand{\tr}{\mathrm{tr}}

\newcommand{\gd}{\nabla}
\newcommand{\rta}{\rightarrow}
\newcommand{\lw}{\left|}
\newcommand{\rw}{\right|}
\newcommand{\be}{\begin{equation}}
\newcommand{\ee}{\end{equation}}

\newcommand{\bt}{\begin{thm}}
\newcommand{\et}{\end{thm}}
\newcommand{\bc}{\begin{cor}}
\newcommand{\ec}{\end{cor}}
\newcommand{\bl}{\begin{lem}}
\newcommand{\el}{\end{lem}}
\newcommand{\norm}[1]{\left\lVert#1\right\rVert}
\newcommand{\normm}[1]{{\left\vert\kern-0.25ex\left\vert\kern-0.25ex\left\vert #1 
    \right\vert\kern-0.25ex\right\vert\kern-0.25ex\right\vert}}

\newcommand{\dist}{\operatorname{dist}}

\def\XXint#1#2#3{{\setbox0=\hbox{$#1{#2#3}{\int}$ }
\vcenter{\hbox{$#2#3$ }}\kern-.6\wd0}}

\DeclareFontShape{OT1}{cmr}{bx}{sc}{<-> cmbcsc10}{}

\usepackage{marvosym}

\def\XXint#1#2#3{{\setbox0=\hbox{$#1{#2#3}{\int}$ }
\vcenter{\hbox{$#2#3$ }}\kern-.6\wd0}}

\newcommand{\ep}{\varepsilon}
\newcommand{\osc}{\mathop{\textup{osc}}}

\begin{document}

\title[A gradient degenerate Neumann problem]{Regularity theory of a gradient degenerate Neumann problem}
\author[W. M. Feldman]{William M Feldman}
\address{Department of Mathematics, University of Utah, Salt Lake City, Utah, USA}
\email{feldman@math.utah.edu}
\author[Z. Huang]{Zhonggan Huang}
\address{Department of Mathematics, University of Utah, Salt Lake City, Utah, USA}
\email{zhonggan@math.utah.edu}

\begin{abstract}
    We study the regularity and comparison principle for a gradient degenerate Neumann problem. The problem is a generalization of the Signorini or thin obstacle problem which appears in the study of certain singular anisotropic free boundary problems arising from homogenization. In scaling terms, the problem is critical since the gradient degeneracy and the Neumann PDE operator are of the same order. We show the (optimal) $C^{1,\frac{1}{2}}$ regularity in dimension $d=2$ and we show the same regularity result in $d\geq 3$ conditional on the assumption that the degenerate values of the solution do not accumulate. We also prove a comparison principle characterizing minimal supersolutions, which we believe will have applications to homogenization and other related scaling limits.
\end{abstract}
\maketitle

\section{Introduction}

This paper considers the following critically degenerate Neumann problem
\begin{equation}\label{e.grad-degen-neumann}
\begin{cases}
\Delta u = 0 & \hbox{in } B_1^+\\
\min\{\partial_1 u,|\gd' u|\} = 0 &\hbox{on } B_1'.
\end{cases}
\end{equation}
Here we have denoted \(\pt_i=e_i\cdot \gd\) for \(i=1,\cdots, d\), \(\gd'= (\pt_2,\cdots, \pt_d)\), \(B_1^+=B_1\cap\{x_1>0\}\subset \R^d\) and \(B_1'=B_1\cap\{x_1=0\}\subset \R^{d-1}\). The \emph{contact set} of $u$, formally defined as 
\begin{equation}\label{e.contact-set-initial-def}
    \Ca_u := \{x \in B_1': \partial_1u >0\} \subset \{x \in B_1': |\grad'u| = 0\},
\end{equation}
is of central interest.

This problem appears in the study of certain singularly anisotropic Bernoulli free boundary problems arising from homogenization (see \sref{bernoulli-connection} below). In elliptic PDE terms, the problem \eref{grad-degen-neumann} is an example of a PDE with gradient degeneracy. The non-local PDE operator $\partial_1u$ and the gradient degeneracy $|\grad'u|$ are both first-order derivatives making the problem critical.

We will study the regularity and comparison principle for solutions of \eref{grad-degen-neumann}. First, we will show that solutions are Lipschitz continuous. Then we prove the optimal \(C^{1,1/2}\) regularity of solutions in \(d=2\). In higher dimensions $d \geq 3$, we prove optimal regularity under the condition that $u$ only takes finitely many distinct values on its contact set $\Ca_u$.

The gradient degenerate Neumann problem \eqref{e.grad-degen-neumann} is closely related to and, in fact, generalizes the well-known Signorini or thin obstacle problem \cites{athanasopoulos2006optimal,fernándezreal2016c1alpha,garofalo2009monotonicity,garofalo2014,milakis2008regularity} 
\begin{equation}\label{o.thin-obstacle}
\begin{cases}
\Delta w = 0 & \hbox{in } B_1^+\\
\min\{\partial_1 w,-w\} = 0 &\hbox{on } B_1'.
\end{cases}
\end{equation} 

 Notice that \(w\) in the thin obstacle problem \eqref{o.thin-obstacle} is also a viscosity solution to \eqref{e.grad-degen-neumann} with \(w\equiv 0\) on \(\Ca_w\) and \(w\le 0\) on the whole flat boundary \(B_1'\). Unlike the thin obstacle problem, the problem \eqref{e.grad-degen-neumann} does not involve any pre-defined obstacle. However, we will show that any viscosity solution \(u\) to \eqref{e.grad-degen-neumann} is constant on each component of the ``contact set" \(\Ca_u\) defined in \eref{contact-set-initial-def} (see \lref{contact-set-open} below). Thus our ``contact set" $\Ca_u$ generalizes the role of the contact set in the thin obstacle problem, and our problem falls under the general class of \emph{unconstrained free boundary problems} surveyed in \cite{figalli2015overview}. 
 
 Although Signorini solutions solve the degenerate Neumann problem \eqref{e.grad-degen-neumann}, the problem \eqref{o.thin-obstacle} allows additional solutions that do not arise from a Signorini problem. There is, in general, non-uniqueness of solutions to the problem \eqref{e.grad-degen-neumann} even with Dirichlet data posed on the outer boundary $\partial B_1\cap\{x_1>0\}$. Maximal subsolutions of \eref{grad-degen-neumann} just solve the Neumann problem. Minimal supersolutions, on the other hand, generally have nontrivial contact sets $\Ca_u$. In some cases, the minimal supersolution corresponds to a Signorini problem, but even when $\Ca_u$ has only finitely many components solutions may bend below \emph{or} above the \say{obstacle} (see Figure \ref{fig:comparison}). Our final main result of the paper is a comparison principle (see Theorem \ref{comparisonprinciple}) which characterizes minimal supersolution by one additional non-local viscosity solution property, the \emph{boundary maximum principle}.  We expect this comparison principle to allow for regularization arguments, and to have applications in homogenization.

 The generalization brings several new challenges in the analysis of regularity. For example, because of the absence of a thin obstacle, it seems unclear that we can obtain semi-convexity/-concavity of a solution as in the thin obstacle case \cites{athanasopoulos2006optimal,milakis2008regularity}. We solve this challenge by proving pointwise differentiability via a different approach that combines the nontangential convergence theories and the Almgren monotonicity formula. There are also possible piling-ups of infinitely many components of \(\Ca_u\) with \(u\) having infinitely many different values on them, which might ruin the differentiability of a solution (see Theorem \ref{t.introtheorem3}). However, this challenge seems unattainable in the current context so we will defer this issue to future work.

{The problem \eqref{e.grad-degen-neumann} can also be viewed as a critical case of a class of gradient degenerate elliptic problems. In the pioneering work \cite{imbert2013c1alpha}, Imbert and Silvestre studied the following type of degenerate elliptic equation
\[|\grad v|^\gamma F(D^2v) = f\]
with $F$ being uniformly elliptic. This research continued in the case of non-local operators of order $1 < \sigma < 2$ in several works \cites{araújo2023fractional,prazeres2021interior}
\be|\grad v|^\gamma \Delta^{\sigma/2} v = f.\label{inhomogeneousgraddege}\ee 
In this context our problem \eqref{e.grad-degen-neumann} falls at the critical order $\sigma=1$ where the gradient degeneracy and the non-local PDE operator are of the same order. Our work is the first to discuss finer properties in this challenging critical case for gradient degenerate PDEs of this type.}

\subsection{Main results} Matching the optimal \(C^{1,1/2}\) regularity of the thin obstacle problem, we will show the following main result on the regularity of \eqref{e.grad-degen-neumann} in dimension $2$.

\begin{theorem}\label{t.introtheorem1}
    Suppose that $u$ solves \eqref{e.grad-degen-neumann} and $d=2$. Then $u$ is in $ C^{1,\frac{1}{2}}_{\textup{loc}}(B_1^+ \cup B_1')$ and there is a universal $C \geq 1$ so that
    \be
    \norm{u}_{C^{1,\frac{1}{2}}\lb\overline{B_{1/2}^+}\rb} \le C \norm{u}_{L^\infty(B_1^+)}. 
    \ee
    \label{introtheorem1}
\end{theorem}

We can also prove similar regularity in dimension $d \geq 3$ under the condition that $u$ takes at most finitely many values on its facets.  

\begin{theorem}\label{t.introtheorem2}
    Suppose that $u$ solves \eqref{e.grad-degen-neumann}, $d \geq 2$, and  $u(\mathcal{C}_u) \subset \R$ is finite. Then $u$ is in $ C^{1,\frac{1}{2}}_{\textup{loc}}(B_1^+ \cup B_1')$ and
    \be
    \norm{u}_{C^{1,\frac{1}{2}}\lb\overline{B_{1/2}^+}\rb} \le C \norm{u}_{L^\infty(B_1^+)}, 
    \ee
    where $C$ is universal when $d=2$, and in $d \geq 3$, at most, $C$ depends on $d$ and the minimal gap of the degenerate values as defined below
    \be\label{minimalgapcondition}
    \textup{gap}(u):=\min\{|a-b|; a\ne b,\,a,b\in u(\Ca_u)\}.
    \ee
    \label{introtheorem2}
\end{theorem}
\begin{remark}
    The minimal gap is always positive under the assumption \(u(\Ca_u)<\infty\), and it is a useful quantitative parameter of the latter condition. In dimension \(d\ge 3\), we can slightly improve the bounding coefficient for a smaller regularity exponent \(1/2>\alpha=\alpha(d)>0\) by proving the estimate:
    \[
     \norm{u}_{C^{1,\alpha}\lb\overline{B_{1/2}^+}\rb} \le C \norm{u}_{L^\infty(B_1^+)},
    \]
    with \(C\) depending at most on \(d\) and \(\# u(\Ca_u)\). Unlike the positive minimal gap \eqref{minimalgapcondition}, the quantity \(\# u(\Ca_u)\) sets no restrictions on the distances between any two distinct degenerate values. See Remark \ref{replacegapbynumber} for the details.
\end{remark}

The proof of the conditional regularity also shows the following result, which is useful to interpret the remaining open issues about the regularity of \eref{grad-degen-neumann}.  If $u$ were to fail to be differentiable at the origin then $u$ would need to have infinitely many facets in any neighborhood of $0$.
\begin{theorem}\label{t.introtheorem3}
    If $u$ solves \eref{grad-degen-neumann} in $B_1^+$ and $u$ fails to be differentiable at $0$ then $\#u(\Ca_u \cap B'_r) = +\infty$ for all $0<r<1$.
\end{theorem}

Of course, we do not have any example of non-differentiability, so it may be possible to rule this scenario out using other methods. We will further interpret this conditional result below in Section \ref{section.C1alphacondition}.

\subsection{Ideas of the proof}

\subsubsection{Nontangential convergence and almost everywhere differentiability}
In Section \ref{section.lips}, we establish the Lipschitz estimate by using a doubling variable method and the Jensen-Ishii lemma \cites{IshiiLions1990,UsersGuide}. To go beyond the Lipschitz regularity, the typical approach in the Signorini problem goes via a semi-concavity/semi-convexity estimate \cites{athanasopoulos2006optimal,milakis2008regularity}.   This technique does not seem to work in our setting.  Instead, in Section \ref{section.almostdiff}, we utilize the classical theory of the non-tangential boundary behavior of bounded harmonic functions \cites{hunt1968boundary,garnett2005harmonic}.  Since the gradient \(\gd u\) is bounded, by the Lipschitz estimate, and harmonic we can apply this classical theory of harmonic analysis. We show surface measure almost everywhere differentiability (including nontangential directions) of solutions \(u\) to \eqref{e.grad-degen-neumann} on $B_1'$.

\subsubsection{On 2D regularity}
In \cite{athanasopoulos2008structure}*{Section 2}, an idea by Hans Lewy was introduced to observe the optimality of the \(C^{1,1/2}\) regularity of the thin obstacle problems. In Section \ref{section.2dopt} we show that this idea can also be applied to the gradient degenerate Neumann problem \eqref{e.grad-degen-neumann} in dimension \(d=2\). Let \(\gd u = (\pt_1u,\pt_2u)\) be the bounded gradient and then we can define
\[
F=\pt_2 u + i \pt_1 u
\]
as a complex analytic function on \(B_1^+\). Its square satisfies
\[
G:= F^2 = |\pt_2 u|^2 - |\pt_1 u|^2 + 2 i \pt_1u\pt_2 u =: U + iV.
\]
By the boundary condition of \eqref{e.grad-degen-neumann}, we know that \(V\equiv0\) on the flat boundary \(B_1'\), and hence \(V\) can be harmonically extended to the whole \(B_1\) via odd extension. By classical complex analysis, this means that \(G\) is a complex analytic function in the whole \(B_1\). Now \(F=\sqrt{G}\) will admit \(C^{1/2}\) regularity across \(B_1'\) and hence \(u\in C^{1,1/2}\).

\subsubsection{Conditional regularity in \texorpdfstring{$d \geq 3$}{d}}
In Sections \ref{section.C1alphacondition} and \ref{section.conditionopti} we establish the conditional regularity results of \tref{introtheorem2} and \tref{introtheorem3}. 


Our work introduces a distinct method for addressing pointwise differentiability, which avoids relying on the semi-convexity estimate typically used in thin obstacle problems \cites{milakis2008regularity,athanasopoulos2006optimal}. As previously noted, our specific equation \eqref{e.grad-degen-neumann} does not lend itself to semiconvexity-based analysis. Instead, our novel approach integrates the property of non-tangential almost everywhere differentiability with the Almgren monotonicity formula to establish pointwise differentiability. The Almgren monotonicity formula has been extensively applied in the study of thin obstacle problems, see for example \cites{garofalo2009monotonicity, fernándezreal2020obstacle,athanasopoulos2008structure} and other references therein.  In our case there is an additional error term in the derivative of the Almgren frequency functional which we have, so far, only been able to control using the condition \(\# u(\Ca_u\cap B_1)<\infty\). This is the only place where the condition is used in the proof in dimension \(d\ge3\).

Next, using the pointwise differentiability property, we establish a $C^{1,\alpha}$-type improvement of flatness iteration. In the \(C^{1,\alpha}\)-iteration, similar to the set-up in \cite{imbert2013c1alpha}, it is useful to consider the following tilted boundary condition
\be\label{tiltedboundaryconditionm}
\min\{\pt_1 u +m_1,|\gd' u + m'|\}=0,
\ee
with \(m=(m_1,m')\in \R^d\). Unlike the iterations in \cites{imbert2013c1alpha}, we show that there is a general constraint on the gradient \(m\): if \(\osc_{B_1^+} u \le 1\) satisfies \eqref{e.grad-degen-neumann} with the boundary condition replaced by \eqref{tiltedboundaryconditionm} then the vector \(m=(m_1,m')\in \R^d\) would satisfy
\[
|\min\{m_1,|m'|\}| \le K(d),
\]
for some positive constant \(K(d)\) depending only on the dimension. This dichotomy classifies the allowed gradients \say{\(m\)} in the iterations into two cases \(m=(m_1,0)\) with \(m_1>0\) or \(m=(0,m')\) with \(m'\in \R^{d-1}\). We emphasize here that this dichotomy idea and the improvement of flatness procedure essentially does not depend on the finiteness condition \(\# u(\Ca_u\cap B_1)<\infty\), and hence we would obtain a full \(C^{1,\alpha}\) estimate as long as we had pointwise differentiability.

\subsection{Motivations and literature}

\subsubsection{Unconstrained free boundary problems and gradient degenerate elliptic equations}

The gradient degenerate problem \eqref{inhomogeneousgraddege} has drawn much attention in recent years, and can be in general categorized into the class of \emph{regularity matching} problems, see Section 2.2 in the survey of Figalli and Shahgholian \cite{figalli2015overview} on unconstrained problems. In particular, the homogeneous version of the equation \eqref{inhomogeneousgraddege} can be viewed as a regularity matching problem: for a bounded domain \(\Omega\subset\R^n\)
\be
\bca
|\gd u|^\gamma \Delta^{\sigma/2} u = 0 & \hbox{ in } \Omega\\
u=g & \hbox{ outside }\Omega,
\eca\label{homogeneousgradeg}
\ee
where \(\Omega\) is a bounded open domain. In this problem \(u\) satisfies a non-local elliptic problem outside the free domain \(\{|\gd u|=0\}\), in the interior of which the gradient vanishes. Multiple regularity results for different choices of \(\gamma\) and \(\sigma\) have been discussed \cites{araújo2023fractional,prazeres2021interior}. In \cites{prazeres2021interior}, a \(C^{1,\alpha}\) regularity result is obtained for \eqref{inhomogeneousgraddege} for \(1<\sigma<2\) and \(\sigma\) close to 2. The proof relies on a perturbative method around the case \(\sigma=2\), which is included in the well-known work of Imbert and Silvestre \cites{imbert2013c1alpha}. Recently in \cite{araújo2023fractional}, under the condition that the exterior datum \(g\) admits only one solution to the homogeneous equation \eqref{homogeneousgradeg}, an optimal \(C^{1,\alpha}\) regularity result is obtained for the case \(1<\sigma<2\) with 
\be\label{optimalexponentnonloca}
\alpha(\gamma,\sigma)=\frac{\sigma-1}{1+\gamma}.
\ee
As \(\sigma\rta 2^-\), the estimate remains uniform and coincides with the result when \(\sigma=2\) \cites{imbert2013c1alpha}. The gradient degenerate Neumann problem \eqref{e.grad-degen-neumann} can be categorized into the nonlocal gradient degenerate problem \eqref{homogeneousgradeg} in the case that \(\sigma=1\). Indeed, if \(u\) is a global solution to \eqref{e.grad-degen-neumann} then we know that \cites{silvestre2007regularity,caffarelli2008regularity} 
\[
\pt_1 u=\Delta_{x'}^{1/2}u
\]
with \(\Delta_{x'}^{1/2}\) the fractional Laplacian on \(\R^{d-1}\). Now any global viscosity solution to the equation \eqref{e.grad-degen-neumann} satisfies
\[
|\gd' u| \Delta_{x'}^{1/2} u =0
\]
in the viscosity sense. However, by the optimal regularity exponent as described in \eqref{optimalexponentnonloca}, the \(C^{1,\alpha}\) regularity would reduce to a Lipschitz one when \(\sigma\rta 1^+\), which means that \eqref{e.grad-degen-neumann} lies exactly in the critical case of the gradient degenerate problem \eqref{homogeneousgradeg}.

The original strategy of Imbert and Silvestre \cite{imbert2013c1alpha} relies on the following property -- which generalizes to similar nonlinear PDE: 
\[
\hbox{any viscosity solution of } \ |\grad u|^\gamma\Delta u =0 \ \hbox{ in } B_1 \ \hbox{ actually solves } \Delta u = 0 \ \hbox{ in } B_1.
\]
Flat solutions of inhomogeneous problems then inherit regularity from the solutions of the homogeneous problem.  However, in the non-local case \(1<\sigma<2\), we don't have the same property:
\[
|\gd u|^\gamma \Delta^{\sigma/2} u = 0 \quad\textup{ does not imply }\quad  \Delta^{\sigma/2} u =0,
\]
which is due to the nonuniqueness of solutions to the homogeneous problem \eqref{homogeneousgradeg}. Previous results in the literature either require $\sigma$ near $2$ to inherit regularity from the second order case \cite{prazeres2021interior}, or the most recent results obtain regularity under the assumption that \eqref{homogeneousgradeg} has a unique solution to apply a similar improvement of flatness strategy again. In our problem, we also have a similar nonuniqueness issue, but we are specifically interested in general solutions of the homogeneous problem \eqref{e.grad-degen-neumann} in cases of non-uniqueness where the minimal supersolution is nontrivially distinct from the Neumann solution.  We also build on several ideas from \cite{imbert2013c1alpha}, including the Lipschitz estimate and the formulation of the $C^{1,\alpha}$ iteration, but the source of differentiability is distinct and is more related to the thin obstacle theory.  Thus, our techniques combine ideas from the gradient degenerate elliptic PDE theory and the thin obstacle problem.

\subsubsection{Singular Bernoulli free boundary problems}\label{s.bernoulli-connection}
 
Our original motivation to study \eqref{e.grad-degen-neumann} comes from a connection with a singularly anisotropic Bernoulli one-phase problem. Specifically, consider the Bernoulli-type one-phase problem set in the exterior of a compact region $K$
 
\be
\bca
\Delta u =0, &\text{ in }\{u>0\}\setminus K,\\
u\equiv1,&\text{ on }K,\\
|\gd u|= Q(\gd u), & \text{ on } \partial \{u>0\} \setminus K,
\eca\label{discontinuous}
\ee
with the anisotropy \(Q\) being a 0-homogeneous function of the form
\be
Q(e)=\bca
1,&e\ne e_1\\
2,&e=e_1,
\eca
\ee
where \(e_1,e_2,\dots,e_d\) form an orthonormal basis for \(\R^d\). 

This type of singular anisotropy \(Q\) arises from a natural homogenization problem for the classical Bernoulli one-phase problem \cites{caffarelli2007homogenization, kim2008,CaffarelliLeeMellet,feldman2019free,feldman2021limit}.  Specifically, consider the following one-phase problem with laminar oscillatory heterogeneity.
\be
\bca
\Delta u_\ep =0, &\text{ in }\{u_\ep>0\}\setminus K,\\
u_\ep\equiv1,&\text{ on }K,\\
|\gd u_\ep|= q(x_1/\ep), & \text{ on } \partial \{u_\ep>0\}.
\eca\label{laminarcoeff}
\ee
Here $q$ is a $1$-periodic function on \(\R\). While the energy minimizing solutions of \eqref{laminarcoeff} converge to solutions of a classical Bernoulli problem, it is known that the \emph{minimal supersolutions} $u_\ep$ instead converge to the minimal supersolution of the anisotropic problem \eqref{discontinuous}, see \cites{caffarelli2007homogenization,feldman2021limit}.  

There are some results on the regularity of solutions to \eqref{discontinuous} in the case when $K$ is convex \cites{feldman2019free,feldman2024convex,changlara2017boundary}, however little is known without convexity.

The connection between the anisotropic free boundary problem \eqref{discontinuous} and the gradient degenerate Neumann problem \eqref{e.grad-degen-neumann} comes from the formal asymptotic expansion of flat solutions.  Such formal asymptotic expansions can be leveraged, rigorously, to obtain regularity of flat solutions in many PDE problems, the general idea was introduced by Savin \cite{savin2007small} and first leveraged for free boundary problems in a very influential paper of De Silva \cite{desilva2011free}.

To be more specific: suppose that $u$ solves the homogeneous one-phase problem in $B_1$
\[\Delta u = 0 \ \hbox{ in } \ \{u>0\} \cap B_1, \ \hbox{ with } \ |\grad u| = 1 \ \hbox{ on } \ \partial \{u>0\} \cap B_1\]
and is $\ep$-flat, i.e.
\[(x_1 - \ep)_+ \leq u(x) \leq  (x_1 + \ep)_+ \ \hbox{ in } \ B_1\]
for some small enough $\ep >0$. Then one considers the formal asymptotic expansion
\[u(x) = (x_1 + \ep w(x) + o(\ep))_+.\]
Computing the boundary condition
\[1 = |\grad u|^2 = 1 + 2\ep \partial_1 w + o(\ep) \ \hbox{ on } \ \partial \{ u > 0\} \approx \{x_1 >0\}\]
one finds that $w$ should solve the Neumann problem
\[
\begin{cases}
\Delta w = 0 & \hbox{in } B_1^+\\
\pt_1 w = 0 &\hbox{on } B_1'.
\end{cases}
\]
De Silva's approach \cite{desilva2011free} shows the rigorous validity of this asymptotic expansion and uses this to establish $C^{1,\alpha}$ regularity of the free boundary of sufficiently flat (universal $\ep$) solutions.

Later in \cite{changlara2017boundary}, Chang-Lara and Savin studied the regularity of \(\pt\{u>0\}\) when \(u\) is constrained in the way that \(u=0\) outside a smooth obstacle domain \(W_{obs}\) that contains \(K\). They proved optimal \(C^{1,1/2}\) regularity of the free boundary that is near \(\pt W_{obs}\) under the assumption that \(\pt  W_{obs}\) is \(C^{1,1}\). The key observation in their paper is that when \(u\) is sufficiently flat in \(\{u>0\}\cap B_1(x)\), the free boundary can be well-approximated by the function graph of a solution to the thin obstacle problem \eqref{o.thin-obstacle}. The derivation from the asymptotic expansion of \eqref{discontinuous} to the equation \eqref{e.grad-degen-neumann} follows a similar logic to \cite{desilva2011free}.

An analogous formal asymptotic expansion of the singular anisotropic Bernoulli problem \eqref{discontinuous} leads to the gradient degenerate Neumann problem \eqref{e.grad-degen-neumann}. More specifically if $u$ solves \eqref{discontinuous} in $B_1$ and is $\ep$-flat 
  \[(x_1 - \ep)_+ \leq u(x) \leq  (x_1 + \ep)_+ \ \hbox{ in } \ B_1\]
then we can formally expand
\[u(x) = (x_1 + \ep w(x) + o(\ep))_+. \]
If we ignore the higher-order terms, we have, at a free boundary point 
  \[1\le Q(\grad u)^2 = |\grad u|^2 = 1 + 2 \ep \partial_1 w,\]
  which requires that \(\pt_1 w\ge 0\). If \(|\gd' w|>0\) then \(\gd u = e_1 +  \ep\gd w\) is not parallel to \(e_1\) and hence
   \[1= Q(\grad u)^2 = 1 + 2 \ep \partial_1 w.\]
  This formally leads to the boundary condition of the limiting problem: \(\pt_1 w \ge0\) and if \(|\gd' w|>0\) then \(\pt_1 w =0\), which can be simplified as \(\min\{\pt_1 w, |\gd' w|\}=0\) as illustrated in \eqref{e.grad-degen-neumann}.

  In Section \ref{section.flatasymsingularbernoulli} below we follow the approach of \cites{desilva2011free,changlara2017boundary} to show a rigorous flat asymptotic expansion for directionally monotone solutions of \eqref{discontinuous}.

  \begin{proposition}\label{flatasymptoticexpansion}
    For all $\eta>0$ there exists $\ep_0>0$ so that if $u$ is a minimal supersolution of \eqref{discontinuous} in $B_1$ and is $\ep$-flat \eqref{fltanessforp} with slope $p = e_1$ and $\ep \leq \ep_0$ then there is a solution $w$ of \eqref{e.grad-degen-neumann} so that
    \[(x_1 + \ep w(x) - \eta \ep)_+ \leq u(x) \leq (x_1 + \ep w(x) + \eta \ep)_+ \ \hbox{ in } \ B_{1/2}.\]
\end{proposition}

Actually, we can show that $w$ is a minimal supersolution of \eqref{e.grad-degen-neumann}, see Remark \ref{equivalenceofthreenotionsofsolutions}.

\subsection{Non-uniqueness and comparison principle for minimal solutions}\label{subsection.nonuniqueness}

As is known in gradient degenerate problems \cite{araújo2023fractional}, we don't in general have uniqueness for problems of the type \eqref{homogeneousgradeg} for \(0<\sigma<2\). The same phenomenon also occurs when we consider the problem \eqref{e.grad-degen-neumann} with a fixed boundary data on \(\pt B_1\cap\{x_1\ge0\}\). The Perron's method minimal supersolution plays an important extremal role.  It satisfies an additional viscosity solution property, the \emph{boundary maximum principle} (see Lemma \ref{boundarymaximumprinciple}). Our last main result of the paper is a comparison principle characterizing the minimal supersolution. 

\begin{theorem}\label{comparisonprinciple}
      Let \(v\) be a super-solution (see Definition \ref{supsolutiondef}) and \(u\) a sub-solution (see Definition \ref{subsolutiondef}) that satisfies an additional boundary maximum principle as described in Lemma \ref{boundarymaximumprinciple}. If \(v\ge u\) on the boundary \(\pt B_1\cap\{x_1\ge0\}\), then we have \(v \ge u\) on the whole \(\overline{B_1^+}\).
 \end{theorem}

A similar comparison principle for the Bernoulli-type problem can be found in \cite{feldman2019free}*{Theorem 5.3}.  The usefulness of this sort of theorem is that it gives a ``local" viscosity solution characterization of the minimal supersolution.  This uniqueness property can be used in the proof of homogenization or other regularization limits, for example as done for related free boundary problems in \cites{feldman2019free,feldman2024convex}.

In general, without the boundary maximum principle, the gradient degeneracy causes the comparison principle to fail. The touching point between a strict subsolution $v$ and supersolution $u$ may occur within the contact set \(\Ca_u\) of the supersolution \(u\), and positivity of $\partial_1u>0$ is no contradiction. The boundary maximum principle is enough to rule out this scenario. There are also technical challenges since we must work with general semi-continuous sub and supersolutions and the PDE is on a lower dimensional set. To regularize, we need to use tangential sub-/sup-convolutions with harmonic replacement.

In general, given a continuous boundary condition $g$ on $\partial B_1 \cap \{x_1 \geq 0\}$, we have at least three different methods to generate solutions of \eqref{e.grad-degen-neumann}.  We can simply solve the Neumann problem, this gives the maximal subsolution.  We can solve the thin obstacle problem with obstacle $\max_{\partial' B_1'} g$ from above. And we can find the Perron's method minimal supersolution. The Perron's method minimal supersolution always satisfies the boundary maximum principle, while the Signorini and Neumann solutions may not. In Figure \ref{fig:comparison} we show an example where all three of these solutions are distinct.

\begin{figure}[htbp]
    \centering
    \includegraphics[width=0.6\textwidth]{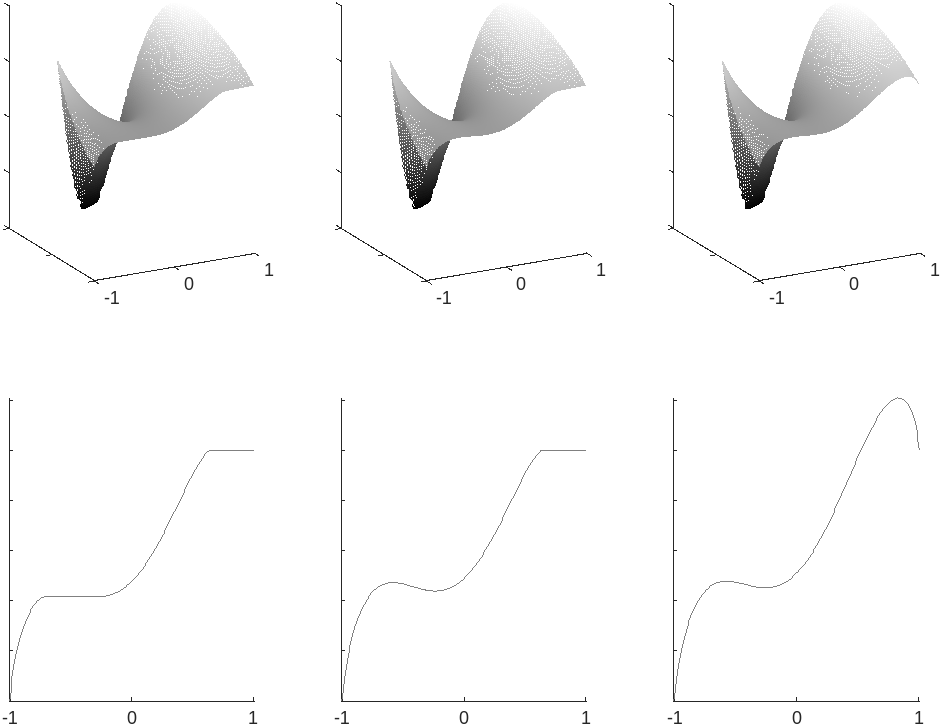}
    \caption{Different solutions of \eqref{e.grad-degen-neumann} with the same boundary data \(g(x_1,x_2)=-43x_1^{8} + 19x_1 + 5x_2 - 5\). Top: (Left) the minimal supersolution: $\Ca_u$ has two components; (Middle) the thin obstacle solution \eqref{o.thin-obstacle} with $\max_{\partial' B_1'} g=0$ as an obstacle from above: it has only one flat component; (Right) the Neumann solution / maximal subsolution. Bottom: the corresponding restrictions to \(B_1'\). Notice that only the minimal supersolution satisfies the boundary maximum principle in this case.}
    \label{fig:comparison}
\end{figure}

\subsection{Outline}
In Section \ref{section.prelim} we will discuss the viscosity solutions to the equation \eqref{e.grad-degen-neumann} and define the \emph{contact set} \(\Ca_u\) of a viscosity supersolution \(u\). In Section \ref{section.lips} we establish an interior Lipschitz estimate for all bounded viscosity solutions in any dimensions \(d\ge 2\) by applying the doubling variable technique in \cite{imbert2013c1alpha}.

In Section \ref{section.almostdiff} we review some results from the literature on the non-tangential boundary behavior of bounded harmonic functions and show the surface measure almost everywhere differentiability (including non-tangential directions) of a solution \(u\) to \eqref{e.grad-degen-neumann} up to the boundary \(B_1'\). In Section \ref{section.2dopt} we prove Theorem \ref{introtheorem1} by applying the almost everywhere differentiability up to \(B_1'\) and the complex analytic arguments. 

In Section \ref{section.C1alphacondition} we prove the Almgren monotonicity formula under the additional condition \(\# u(\Ca_u) < \infty\). In the same section, we establish the improvement of flatness and hence the \(C^{1,\alpha}\) regularity by using the monotonicity formula. In Section \ref{section.conditionopti} we finish the proof of Theorem \ref{introtheorem2} by using the Almgren monotonicity again. In Section \ref{section.flatasymsingularbernoulli} we show the flat asymptotic expansion of \eqref{discontinuous} gives rise to the problem \eqref{e.grad-degen-neumann}. 

\subsection{Acknowledgments} The authors were both supported under the NSF grant DMS-2009286.  The first author appreciated a helpful conversation with Mark Allen.
\section{Preliminaries\label{section.prelim}}
\subsection{Notations}

\begin{enumerate}[label = $\bullet$]
    \item \(d=n+1\ge 2\) are dimensions.
    \item \((x_1,x')=(x_1,x_2,\cdots,x_d)\in \R^d \) are the coordinate functions. \(e_1,\cdots,e_2\) form an orthonormal basis for \(\R^d\). \(\pt_i,\,i=1,\cdots,d\) are the partial derivatives with respect to the directions \(e_i\). \(\gd'= (\pt_2,\cdots,\pt_d)\) is the tangential gradient.
    \item \(B_r(x)\) is the open ball centered at \(x\in \R^d\) with radius \(r>0\). \(B_r =B_r(0)\).
    \item \(\pt \Omega\) is the boundary of an open domain \(\Omega\subset\R^d\).
    \item \(\Omega^+=\Omega\cap \{x_1>0\}\).
    \item \(\Omega'=\Omega\cap\{x_1=0\}\). \(\pt' \Omega'\) is the relative boundary of \(\Omega'\) in \(\{x_1=0\}\). \(B_r'=B_r'(0)\).
    \item \(\overline{\Omega}\) is the closure of \(\Omega\).
    \item The notation $A\sqcup B$ denotes disjoint union of sets $A$ and $B$.

\end{enumerate}

\subsection{Viscosity solutions} Let us discuss the definition of viscosity solutions to the equation \eqref{e.grad-degen-neumann}.
\begin{definition}  
A function $u \in \mathrm{USC}\lb \overline{B_1^+}\rb$ is a \emph{subsolution} of \eref{grad-degen-neumann} if $u$ is subharmonic in $B_1^+$ and whenever $\varphi$ smooth touches $u$ from above at $x_0 \in B_1'$ {with $\Delta \varphi(x_0) < 0$}
\[ \partial_1 \varphi(x_0) \geq 0.\]
\label{subsolutiondef}
\end{definition}

\begin{definition}  
A function $u\in \mathrm{LSC}\lb \overline{B_1^+}\rb$ is a \emph{supersolution} of \eref{grad-degen-neumann} if $u$ is superharmonic in $B_1^+$ and whenever $\varphi$ smooth touches $u$ from below at $x_0 \in B_1'$ with $\Delta \varphi(x_0) > 0$ then
\[ \min\{\partial_1 \varphi(x_0),|\gd' \varphi|(x_0)\} \leq 0.\]
In other words
\[\hbox{if } |\gd'\varphi|(x_0) >0 \ \hbox{ then } \ \partial_1 \varphi(x_0) \leq 0.\]
A continuous function is called a \emph{viscosity solution} if it is both sub- and supersolutions.\label{supsolutiondef}
\end{definition}

\begin{remark}
 There is no comparison principle and no uniqueness for the solutions as defined above. However, in Section \ref{section.comparisonprinciple}, we will discuss the comparison principle for a supersolution and a \emph{strong} subsolution. This comparison principle characterizes the minimal supersolutions to the problem \eqref{e.grad-degen-neumann}.

\end{remark}

We provide some special example solutions.
\begin{example}
\label{quasisolexamp1}
     Any solution to the Signorini problem
        \be
        \bca
        \Delta w = 0,&\hbox{in }B_1^+\\
        w\le c,&\hbox{on }B_1'\\
        \partial_1 w = 0,&\hbox{on } \{w<c\}\cap B_1'\\
        \partial_1 w \ge 0,&\hbox{on }B_1',
        \eca
        \ee
        where \(c\ge \sup_{\partial'B_1'} g\) is some constant. A simple example solution to this equation for \(c=0\) is \(w(x,y)=-\mathrm{Re}\lb(x+iy)^{3/2}\rb\);
\end{example}
\begin{example}
         The sign-reversed Signorini problem
         \be
        \bca
        \Delta w^- = 0,&\hbox{in }B_1^+\\
        w^-\ge \tilde{c},&\hbox{on }B_1'\\
        \partial_1 w^- = 0,&\hbox{on } \{w^->\tilde{c}\}\cap B_1'\\
        \partial_1 w^- \ge 0,&\hbox{on }B_1',
        \eca
        \ee
        where \(\tilde{c}\le \inf_{\partial'B_1'} g\) is some constant. An example solution is \(w^-(x,y)=\mathrm{Re}\lb(x+iy)^{5/2}\rb\).
    \label{quasisolexamp2}
\end{example}

\begin{lemma}
    Let \(u_k\) be a family of continuous viscosity solutions to \eqref{e.grad-degen-neumann} which converge uniformly in on compact subsets of $B_1^+$ to a limit \(u_\infty\), then \(u_\infty\) is also a viscosity solution.
\end{lemma}
We omit the proof since it follows the standard argument from the viscosity solution theory.

\subsection{Contact, non-contact set, and the thin free boundary} Let us now study the behavior of a supersolution \(u\) on the flat boundary \(B_1'\). First, we give a formal definition of the contact set. This is named in analogy to the thin obstacle problem, but there is, technically speaking, no obstacle to be contacted.

\begin{definition}\label{definitionofCu}
   Let \(u\) be a supersolution to \eqref{e.grad-degen-neumann}, then the \emph{contact set} $\Ca_u$ is defined by:
   \begin{equation*}
       \Ca_u := \{x \in B_1'\,:\, \ \exists \varphi \in C^\infty \hbox{ touching $u$ from below in $\overline{B_1^+}$ at $x$ with } \partial_1\varphi(x)>0\}.
   \end{equation*}
 
\end{definition}

Our first result says that $\Ca_u$ is open.
 
\begin{lemma}\label{l.contact-set-open}
    Let \(u\) be a supersolution, then \(\Ca_u\) is relatively open in \(B_1'\) and \(u\) is constant on each component of \(\Ca_u\).\label{definecu}
\end{lemma}

We will return to the proof in a moment. First, we give some additional definitions, also named in analogy to the thin obstacle problem.

\begin{definition}
       Define the \emph{non-contact set} $\Na_u$ to be the relative interior of $B_1' \setminus \Ca_u$ and the \emph{free boundary} $\Gamma_u:= B_1' \setminus (\Ca_u \cup \Na_u)$. 
\end{definition}

Given these definitions we have
\begin{equation}
    B_1' = \Ca_u \sqcup \Na_u\sqcup \Gamma_u.
\end{equation}
Also note that, from \lref{contact-set-open}, the free boundary $\Gamma_u$ is relatively closed in $B_1'$ and also 
\[\Gamma_u = \partial'\Ca_u \ \hbox{ and } \ \Gamma_u = \partial'\Na_u.\]

    \begin{remark}\label{generalizedCauNauGammau}
        We can extend the definitions of contact/non-contact/free-boundary sets to a larger class of problems. Suppose \(u\) is a viscosity solution to \eqref{e.grad-degen-neumann} with the boundary condition replaced by the following tilted version 
        \[
        \min\{\pt_1 u+p_1, |\gd' u +p'|\}=0,  \hbox{ on } B_1'
        \]
        for some \(p=(p_1,p')\in \R^d\), then that is equivalent to say that \(u+p\cdot x\) is a viscosity solution to the original equation \eqref{e.grad-degen-neumann}, and hence we may define 
        \[
        \Ca_u:=\Ca_{u+p\cdot x},\,\Na_u:=\Na_{u+p\cdot x},\,\hbox{ and }\Gamma_u:=\Gamma_{u+p\cdot x}
        \]
        correspondingly. 
    \end{remark}

\begin{proof}[Proof of \lref{contact-set-open}]
Let \(x_0\in\Ca_u\), and then we may assume, by translation and rescaling, that \(x_0=0\) and $u(0) = 0$, and \(u\) satisfies the following one-sided flatness condition:
\[
u(x)\ge \beta x_1 -\varepsilon,\,x\in \overline{B_1^+}:= \bB_1\cap\{x_1\ge0\},
\]
where \(\beta>0\) is the inward normal slope of the touching test function in the definition of $\mathcal{C}_u$, and \(\varepsilon>0\) can be made arbitrarily small (at the cost of rescaling to a smaller radius depending on $\beta>0$ and the $C^1$ modulus of the touching test function). 

It then suffices to show that \(u\) must be identically equal to $u(0) = 0$ in a small neighborhood of \(0\in B_1'\). To that end, let us consider the following family of harmonic parabola barriers
\[
v_{t,s}(x_1,x'):= -\delta |x'-s|^2 +\frac{\beta}{2} (x_1)_+ + (d-1)\delta x_1^2 -\varepsilon + \ep t,
\]
where \(\delta>0\) is a small number. It is not hard to observe that when \(\varepsilon\ll \delta < \min\{1,\beta\}/(100d)\) and \(s\in \{0\}\times\R^{d-1},\,|s|<1/4\), we have for \(x\in \pt B_1\cap\{x_1\ge0\}\)
\be\label{cauboundaryinequal}
\begin{split}
   \beta x_1 -\varepsilon-v_{1,s}(x)& = \frac{\beta}{2} (x_1)_+-\ep +\delta |x'-s|^2 - (d-1)\delta x_1^2 \\
  &=\frac{\beta}{2} (x_1)_+ +\delta |x'|^2 - 2\delta x'\cdot s + \delta|s|^2 - (d-1)\delta x_1^2 -\ep\\
  &\ge\frac{\beta}{2} (x_1)_+  - d\delta x_1^2+ \delta  - 2\delta x'\cdot s  -\ep\\
  &\ge \lb\frac{\beta}{2}-d\delta\rb (x_1)_+ + \delta - 2\delta |s| -\ep\\
  &>0.
\end{split}
\ee
Now we have 
\[
v_{0,s}(x) \le u(x),\,|s|<1/4\text{ and }x\in\overline{B_1^+},
\]
and then we may consider the largest \(t^\ast\ge0\) such that 
\[
v_{t^\ast,s}(x) \le u(x),\,|s|<1/4\text{ and }x\in\overline{B_1^+}.
\]
Because \( v_{t^\ast,s}\) are harmonic on \(B_1^+\), the touching point cannot be in the interior. Also because of \eqref{cauboundaryinequal}, the touching point cannot occur at \(\pt B_1\cap\{x_1\ge0\}\) either.

Let \(\tilde{x}\in B_1'\) be a touching point, and then we claim that \(\Tilde{x}=s\), which can be observed by computing
\[
\partial_1 v_{t^\ast,s}(\Tilde{x})=\frac{\beta}{2}>0,
\]
and so this would lead to a contradiction to the super-solution condition if \(\tilde{x}\ne s\).

Now, let \(h(s)=t^\ast(s)\) be defined for each specific \(|s|<1/4\). Because by definition \(h\) is \(C^{1,1}\) on the lower side and hence \(h\) is Lipschitz, and because all the lower touching parabolas have the touching points at the peaks, which means that the gradient of \(h\) must be 0 almost everywhere and hence equals some constant.
    
\end{proof}

\section{Lipschitz regularity\label{section.lips}}

 In the course of proving the Lipschitz regularity of solutions to \eref{grad-degen-neumann} it is convenient to consider a slightly more general class of boundary conditions that arise from renormalizations of the form $u(x) \to u(x) - p'\cdot x$.  Let $p' \cdot e_1 = 0$ be a fixed vector orthogonal to $e_1$ and consider the variant of \eref{grad-degen-neumann}
    \be \label{pprimeboundarycon}
    \begin{cases}
\Delta u = 0 & \hbox{in } B_1^+\\
\min\{\pt_1 u, |\gd' u +p'|\}=0  &\hbox{on } B_1'.
\end{cases}
    \ee
    We will prove a Lipschitz estimate on this general class of equations independent of the vector $p'$.

\begin{lemma}\label{l.Lipschitz-estimate}
Let $p' \cdot e_1 = 0$. There is a constant \(C(d) \geq 1\), independent of $p'$, such that, if \(u\) is a continuous viscosity solution of \eqref{pprimeboundarycon}, then
    \be
    \norm{\grad u}_{L^\infty\lb B_{1/2}^+\rb} \le C \norm{u}_{L^\infty(B_1^+)}.
    \ee\label{lipschitzregu}
    In particular, $u$ is locally Lipschitz continuous in $B_1^+ \cup B_1'$.
\end{lemma}
The idea of the proof is from \cite{imbert2013c1alpha}*{Lemma 4}. Basically, this is a version of the Bernstein method for proving Lipschitz regularity in nonlinear elliptic and parabolic equations which uses doubling of variables when differentiating in the PDE is not possible due to insufficient regularity and lack of a good smoothing procedure.  The origin of the idea goes back to \cite{IshiiLions1990}.
\begin{proof}

By homogeneity of the equation, we can assume that \(\norm{u}_{L^\infty(B_1^+)}\le 1/2\). 
It suffices to show that we can find \(L_1>0\) and \(L_2>0\) such that, for all \(x_0\in B_{1/2}^+\) (and hence \(x_0\cdot e_1>0\)),
\begin{equation}\label{e.M-def}
M=\sup_{x,y\in\overline{B_1^+}} u(x)-u(y) - L_1 \omega(|x-y|)-L_2|x-x_0|^2-L_2|y-x_0|^2\le 0,
\end{equation}
where \(\omega(x)=s-\frac{2}{3} s^{3/2}\) if \(s\le 1\) and \(\omega(s)=\omega(1)\) if \(s\ge 1\). If one proves such an inequality then the Lipschitz constant will be bounded from above by all \(L>L_1+L_2\). Indeed, by boundedness and continuity of \(u\) in \(\overline{B_1^+}\), it suffices to consider the case when \(|x-y|<1\) and \(x,y\in B_1^+\). In this case, we choose \(x_0=y\) and obtain
\[
u(x)-u(y) \le L_1 \omega(|x-y|) + L_2|x-y|^2 \le (L_1+L_2)|x-y|.
\]

Assume towards a contradiction that \(M>0\). Note that, since $u$ is continuous the maximum in \eref{M-def} is achieved. Suppose that \((x,y)\in \overline{B_1^+}\times \overline{B_1^+}\) is a pair that achieves the maximum. Then \(x\ne y\) since, otherwise, \(M\le0\) contradicting the assumption. Note that this is where we use the fact that $u$ is a continuous viscosity solution, for semi-continuous viscosity solutions $u^*(x) - u_*(x)$ can be strictly positive allowing the maximum to occur when $x=y$. 

Then we obtain, by the assumption \(M>0\),
\[
L_1 \omega(|x-y|)+L_2|x-x_0|^2+L_2|y-x_0|^2 < u(x) - u(y)\le |u(x)|+|u(y)|\le 1.
\]
By choosing \(L_2=(4/r)^2\) for some fixed small number \(1\gg r>0\) we obtain that \(|x-x_0|\le r/3\) and \(|y-x_0|\le r/3\).  Now we may assume that both \(x\ne y\) are contained in \(\bB_r(x_0)\cap \overline{B_1^+}\csubset B_{2/3}^+\cup B_{2/3}'\).

We now apply the Jensen-Ishii Lemma, see \cite{UsersGuide}*{Theorem 3.2}, to construct a limiting sub-jet \((q_x,X)\) of \(u\) at \(x\) and super-jet \((q_y,Y)\) of \(u\) at \(y\), where 
\be
q_x=q+2 L_2(x-x_0) \quad\text{and}\quad q_y=q-2L_2(y-x_0),
\ee
with \(q=L_1 \omega'(|x-y|)\frac{x-y}{|x-y|}\) and for all small \(\eta>0\) (dependent of the distance \(\dist(x,y)\))
\be
\lb\begin{matrix}
    X&0\\
    0&-Y
\end{matrix}\rb \preccurlyeq \lb \begin{matrix}
    Z&-Z\\
    -Z&Z
\end{matrix}\rb + (2L_2 + \eta) Id,\label{JensenIshiilemma}
\ee
with
\be
\begin{split}
    Z&=L_1\lmb \lb \frac{1}{|x-y|}-\frac{1}{|x-y|^{1/2}}  \rb Id + \lb \frac{1}{2|x-y|^{1/2}}-\frac{1}{|x-y|}\rb \frac{(x-y)\otimes(x-y)}{|x-y|^2}\rmb\\
&=: L_1\lmb \lb \frac{1}{|x-y|}-\frac{1}{|x-y|^{1/2}}  \rb Id + \lb \frac{1}{2|x-y|^{1/2}}-\frac{1}{|x-y|}\rb j\otimes j\rmb.
\end{split}
\ee
Notice that we have the identity
\[
j\cdot Zj= -\frac{L_1}{2}\frac{1}{|x-y|^{1/2}}. 
\]
For notational convenience, we will merely discuss the limit case in the rest of the proof, although it will be more accurate to discuss everything before taking a limit. Let us first discuss the case that both \(x,y\in B_1^+\). By harmonicity of \(u\) in \(B_1^+\) we know that
\[
\tr(X-Y)\ge 0.
\]
On the other hand, if we apply any vector of the form \((v,v)^T\) to \eqref{JensenIshiilemma} we obtain
\be
(X-Y)v\cdot v \le (4L_2+2\eta)|v|^2,\label{XminusYgeneralbound}
\ee
and if we apply \((j,-j)^T\) then we have 
\[
(X-Y)j\cdot j \le 4L_2+2\eta - 2 L_1 |x-y|^{-1/2}\le 4L_2+2\eta -  L_1,
\]
when \(r\ge |x-y|>0\) is chosen small. Suppose \(\{j,\tilde{e}_2,\cdots,\tilde{e}_d\}\) is an orthonormal basis for \(\R^d\) then we obtain 
\[
\tr(X-Y)= (X-Y)j\cdot j+\sum_{i=2}^d (X-Y)\tilde{e}_i\cdot \tilde{e}_i\le d(4L_2+2\eta)- L_1<0
\]
if one choose \(L_1\) large enough.

We also claim that \(x\not\in B_1'\) or otherwise because we assumed \(x_0\cdot e_1>0\)
\[
q_x\cdot e_1 = L_1\omega'(|x-y|)\frac{-y_1}{|x-y|} - 2 L_2 x_0\cdot e_1 <0,
\]
contradicting the Neumann subsolution condition of \(u\) at \(x\in B_1'\). 

It then suffices to consider the case that \(x\in B_1^+\) and \(y\in B_1'\). In this case, we apply the supersolution condition and because
\[
q_y\cdot e_1=L_1 \omega'(|x-y|)\frac{x_1}{|x-y|}+2L_2(x_0\cdot e_1)>0,
\]
we know that \(q_y=(q_y\cdot e_1) e_1-p'\) with \(q_y\cdot e_1>0\) according to the supersolution condition \ref{supsolutiondef}. Now we arrive at this last case that \(y=(0,y')\in \Ca_{u}\) according to Lemma \ref{definecu} and because \(u\) restricted to \(B_1'\) is linear on connected components of \(\Ca_{u}\) we have the following inequality
\[
e\cdot Ye \le0,\hbox{ for all }e\perp e_1.
\]
In particular, by combining this inequality with \eqref{XminusYgeneralbound} we obtain
\be
e\cdot Xe \le (4L_2+2\eta)|v|^2,\hbox{ for all }e\perp e_1.\label{boundXtangential}
\ee
Let \(j=\frac{x-y}{|x-y|}=:(j_1,j')\) as before with \(j'=\beta \tilde{e}\in B_1'\) for some \(1>\beta\ge0\), then if we apply \((j+\tilde{e},\tilde{e})^T\) to \eqref{JensenIshiilemma} we will obtain
\[
(j+\tilde{e})\cdot X (j+\tilde{e}) \le \tilde{e}\cdot Y\tilde{e}+ j\cdot Z j + 10L_2+10\eta \le 10L_2+10\eta -  L_1 ,
\]
and we may also apply \((j-\tilde{e},-\tilde{e})^T\) to obtain, similarly,
\[
(j-\tilde{e})\cdot X (j-\tilde{e}) \le  10L_2+10\eta -  L_1.
\]
For the case \(\beta>0\) we can take
\[
\left\{\frac{j+\tilde{e}}{\sqrt{|j_1|^2+|1+\beta|^2}},\frac{j-\tilde{e}}{\sqrt{|j_1|^2+|1-\beta|^2}},\overline{e}_3,\cdots,\overline{e}_d\right\}
\]
as an orthonormal basis of \(\R^d\) and combine all the above estimates to obtain
\[
\begin{split}
    \tr(X) &= \frac{1}{|j_1|^2+|1+\beta|^2}(j+\tilde{e})\cdot X (j+\tilde{e})\\
    &\quad+ \frac{1}{j_1^2+|1-\beta|^2}(j-\tilde{e})\cdot X (j-\tilde{e})+\sum_{i=3}^d \overline{e}_i\cdot X \overline{e}_i\\
&\le\frac{1}{2-2\beta}\lb10L_2+10\eta -  L_1\rb +\lb d-1 \rb\lb 10L_2+10\eta\rb -  L_1\\
&<0,
\end{split}
\]
where on \(\overline{e}_i\cdot X \overline{e}_i\) we have used the bound \eqref{boundXtangential}. In the case \(\beta=0\) we may choose \(\tilde{e}=e_2\) and \(\overline{e}_i=e_i\) for each \(i=3,\cdots,d\) and then obtain a similar result. The contradiction of the inequality to the harmonicity of \(u\) leads to the proof of the lemma.

\end{proof}

\section{Nontangential convergence\label{section.almostdiff}}
According to the estimate in the prior section, we know that the gradient \(\gd u\) of a viscosity solution \(u\) to the equation \eqref{e.grad-degen-neumann} is bounded and harmonic on $B_r^+$ for all $r < 1$. We will apply classical harmonic analysis results on the boundary behavior and Poisson integral formulae for bounded harmonic functions in Lipschitz domains.

The following result can be found in the paper of Hunt and Wheeden \cite{hunt1968boundary}*{Page 311}, the exact statement in dimension $d=2$ can also be found in the book of Garnett and Marshall \cite{garnett2005harmonic}*{Corollary 2.5}.

\begin{theorem}\label{t.bdd-harmonic}
Suppose \(h\) is a bounded harmonic function in a Lipschitz and starlike domain \(\Omega\) then there is a bounded function \(f\) on \(\pt\Omega\) such that \(h\) converges to \(f\) nontangentially almost everywhere and \(h\) can be recovered from the Poisson integral of \(f\) on \(\pt\Omega\).\label{nontangentialconver}
\end{theorem}

In our paper we will focus on the case that \(\Omega= B_1^+\), which satisfies the conditions as described in Theorem \ref{nontangentialconver}

Now let $u$ be a viscosity solution of \eref{grad-degen-neumann}.  By \tref{bdd-harmonic} there is a full measure set \(E=E_u\subset B_1'\) so that the nontangential limit of \(\gd u\) exists at each \(y\in E\). Furthermore, since, again by \lref{Lipschitz-estimate}, \(\restr{u}{B_1'}\) is a Lipschitz continuous function on $B_1'$, it is differentiable in the tangential variables almost everywhere on $B_1'$. Thus we may also, without loss, assume that \(\restr{u}{B_1'}\) is differentiable in the tangential directions at all \(y\in E\). We notice that at this stage we don't know whether the nontangential limit of \(\gd' u\) coincides with the tangential gradient of \(\restr{u}{B_1'}\).

Combining the above information, there exist bounded functions \(\sigma,\tau\) on \(B_1'\), 
\be\label{e.sigmataodef}
\partial_1u(x)\rta\sigma(y),\,|\gd' u|(x)\rta\tau(y),\,\forall x\rta y\in E\hbox{ nontangentially.}
\ee
We would also like to write
\[
\gd u(x)\rta P(y),\,\forall x\rta y\in E\hbox{ nontangentially.}
\]
In the following, we would like to show that \(u\) is differentiable in \(E\) and in particular, \(P'(y) = \gd' \restr{u}{B_1'}(y)\) for all \(y\in E\).

\begin{lemma}\label{l.weaksol}
    For any bounded viscosity solution \(u\) to the problem \eqref{e.grad-degen-neumann}, there is a full measure subset \(E_u\subset B_1'\), on which \(u\) is differentiable (including nontangential directions), and 
    \[
    \min\{\sigma(y),\tau(y)\}=0,\,\forall y\in E=E_u,
    \]
    with \(\sigma,\tau\) as defined in \eqref{e.sigmataodef}. Moreover, \(u\) satisfies the Neumann boundary condition \(\partial_1 u =\sigma\) on \(B_1'\) in the distributional weak sense:
    \[
 \int_{B_{1}^+} \gd u \cdot \gd\phi + \int_{B_{1}'} \sigma \phi = 0,\,\hbox{ for all }\phi\in C_{\textup{loc}}^\infty(B_1^+\sqcup B_1').
    \]
    \label{weaksol} 
\end{lemma}

\begin{proof}
Let \(r>0\) be small, and we would like to consider the following families of functions with \(|x|\le 1,x_1\ge0\)
\[
u_r(x)=\frac{u(y+r x)-u(y)}{r}.
\]
By Lipschitz estimate, we know that the above family of functions has convergent subsequences. Let \(r_k\rta0\) be a subsequence such that \(u_{r_k}\) converges uniformly to some other Lipschitz function \(u_\infty\) in \(\overline{B_1^+}\). By classical viscosity solution theory, we know that \(u_\infty\) also has to be a viscosity solution to \eqref{e.grad-degen-neumann}.

On the other hand, by the nontangential convergence of the \(\gd u\) to the boundary, we have for \(x_1>0\)
\[
\begin{split}
  u_\infty(x)&=\lim_{k\rta\infty} \frac{u(y+r_k x)-u(y)}{r_k}\\
  &=\lim_{k\rta\infty}\frac{\int_{0}^{r_k} \gd u(y+t x)\cdot x dt}{r_k}\\
  &=P(y)\cdot x. 
\end{split}
\]
This equality is also true for \(x_1=0\) because  $u_\infty$ is Lipschitz continuous up to \(B_1'\). Since \(P(y)\) is uniquely determined, we know that the above convergence of \(u_{r_k}\) holds for any convergent subsequences \(r_k\), and hence we obtain differentiability at \(y\in E\). In particular, we have \(\gd' \restr{u}{B_1'}(y)= P'(y)\) for all \(y\in E\). Now the lemma is proved by observing that any viscosity solution of the form \(P(y)\cdot x\) satisfies
\[
\min\{P_1(y),|P'|(y)\}=\min\{\sigma(y),\tau(y)\}=0.
\]

To show that \(u\) satisfies the Neumann boundary condition in the distributional weak sense, we first observe that, by interior regularity of harmonic functions, for any \(\phi\in C_{\textup{loc}}^\infty(B_1^+\sqcup B_1')\) 
\[
    \int_{B_{1}^+\cap\{x_1>1/k\}} \gd u \cdot \gd\phi + \int_{B_{1}^+\cap\{x_1=1/k\}} \pt_1 u \phi = 0.
\]
By applying the nontangential convergence of \(\pt_1u(1/k,x')\rta \sigma(x')\) as \(k\rta\infty\) and the Lipschitz estimate \ref{lipschitzregu}, we know that after sending \(k\rta\infty\),
\[
    \int_{B_{1}^+} \gd u \cdot \gd\phi + \int_{B_{1}'} \sigma \phi = 0.
\]
\end{proof}

\begin{corollary}
    \label{consequenofuniformlipindeppprime}
   For an arbitrary \(p'\in\R^d\) such that \(p'\cdot e_1=0\), a viscosity solution \(w_{p'}\) of \eqref{pprimeboundarycon} also satisfies
    \[
    \min\{\partial_1w_{p'}(x),|\gd' w_{p'}(x)+p'|\}=0,\hbox{ for almost all }x\in B_1',
    \]
    in the sense of nontangential convergence. In particular, there exists a constant \(L=L\lb d,\norm{w_{p'}}_{L^\infty(B_1^+)}\rb>0\) such that if \(|p'|>L\) then \(w_{p'}\) satisfies the zero Neumann boundary condition on \(B_{1/2}'\) in the classical sense.

\end{corollary}

\begin{proof}
    The nontangential convergence can be derived similarly to Lemma \ref{weaksol}. On the other hand, by the Lipschitz estimate, Lemma \ref{lipschitzregu}, the function \(w_{p'}\) is uniformly Lipschitz with the Lipschitz constant \(L=L(d,\norm{w_{p'}}_{L^\infty})>0\) independent of the choice of \(p'\). If one choose \(|p'|>L\), then we have \(|\gd' w_{p'}(x)+p'|>0\) almost everywhere on \(B_{1/2}'\), which implies that \(\pt_1 w_{p'}=0\) almost everywhere on \(B_{1/2}'\). By the second part of Lemma \ref{weaksol}, it implies that \(w_{p'}\) satisfies the zero Neumann boundary condition on \(B_{1/2}'\) in the distributional weak sense. By classical regularity theory for Neumann problems, this implies that \(w_{p'}\) satisfies the Neumann boundary condition in the classical sense on \(B_{1/2}'\).
\end{proof}

\begin{corollary}
    There is a full-measure set \(E_u\subset B_1'\) such that \(\sigma(y)=0\) for all \(y\in E_u\cap \lb\Na_u\sqcup\Gamma_u\rb\). \label{characterizesigma}
\end{corollary}

\begin{proof}
    We pick \(E_u\) to be the set of differentiability of \(u\) and \(y\in E_u\cap \lb\Na_u\sqcup\Gamma_u\rb\). Let us now consider the blow-up function \(v(x)=v_r(x) = \frac{u(r x+y) - u(y)}{r}\) satisfies the following \(\ep=\ep_r\)-flatness
    \[
    \begin{split}
         v_r(x) &= \frac{u(r x+y) - u(y)}{r}\\
         &= p(y)\cdot x + O(\ep_r)\\
   &= \sigma(y) (x_1)_++ \gd' u(y)\cdot x + O(\ep_r).
    \end{split}
    \]
    where \(x\in \overline{B_1^+},\,r>0\) and as \(r\rta0^+\), the flatness \(\ep_r\rta0\). By Lemma \ref{weaksol} we know that \(\sigma(y)\ge 0\) and if \(\sigma(y)>0\) then \(\tau=|\gd'u(y)|=0\), and hence we may without loss write 
    \[
    v_r(x)=\sigma(y) (x_1)_++O(\ep_r).
    \]
  We argue similarly to Lemma \ref{definecu} by contradiction: if $\sigma(y)>0$, then we construct a function of the form
\[
\sigma(y) x_1/2 -\delta (x_2)^2 + 2\delta (x_1)^2 - \nu.
\]
Choose \(\delta>0\) and \(\nu\) properly, so that the function is below \(\sigma(y)x_1 -C\ep_r\) on \(\pt B_1\cap\{x_1\ge0\}\) for small $r>0$, and it will touch $v_r$ only on $B_1'$. If the touching point is 0, then because 0 is in $\Na_v\sqcup\Gamma_v$, $\sigma(y)\le 0$. If the touching point is not 0 then the touching point has to be in $\Na_v\sqcup\Gamma_v$ too by the supersolution condition of \(v\).
\end{proof}

\section{Optimal regularity in dimension \texorpdfstring{\(d=2\)}{d}\label{section.2dopt}}

In this section, we present the proof of the optimal \(C^{1,1/2}\) regularity of any viscosity solutions \(u\) to the equation \eqref{e.grad-degen-neumann} in dimension \(d=2\). 

The idea starts with the classic use of complex variables, originally due to Hans Lewy, and with well-known application in thin obstacle problems, see \cite{athanasopoulos2006optimal}*{Section 2}. Consider the complex analytic function

\[
F=\partial_2 u + i \partial_1 u,
\]
and its square
\[
G:=F^2=\lw\partial_2 u\rw^2-\lw\partial_1 u\rw^2 + 2 i \partial_1 u \partial_2 u=: U + iV.
\]
We focus first on the imaginary part
\[
V=2\partial_1 u \partial_2 u.
\]
Since it is the imaginary part of an analytic function \(V\) is harmonic, it is also bounded in \(B_1^+\) due to \lref{Lipschitz-estimate}. Furthermore, by Lemma \ref{weaksol}, we have that, for almost every $y \in B_1'$,
\be\label{e.V-nt-0}
V(x) \to 0 \ \hbox{ as } \ x \in B_1^+ \to y.
\ee
In other words, \(V = 0\) on \(B_1'\) in the sense of nontangential convergence. By applying Theorem \ref{nontangentialconver} then \(V\) satisfies zero Dirichlet boundary condition in the classical sense on \(B_1'\) and so \(V\) can be odd extended to a harmonic function in the whole disc \(B_1\), which we still denote by $V$. 

Then, by classical complex analysis, $V$ admits a unique harmonic conjugate in the entire $B_1$.  It must agree with $U$ in the upper half ball and so we denote it as $U$, a harmonic extension of $U$ to $B_1$.  Notice that the odd symmetry of $V$ implies that $U$ is even symmetric with respect to $x_1 \mapsto - x_1$.

Thus we proved the following lemma.
\begin{lemma}
    The function \(F\) is analytic in \(B_1^+\) and its square \(G=F^2\) has a unique analytic continuation to the whole disc \(B_1\).\label{analyticsquare}
\end{lemma}

With this lemma, we are now able to prove Theorem \ref{introtheorem1}.

\begin{proof}[Proof of Theorem \ref{introtheorem1}]
According to Lemma \ref{analyticsquare} we know that \(U\) has a harmonic extension to the whole disc \(B_1\). Also we have the formula $U = |\partial_1u|^2 - |\partial_2u|^2$ and so, given that the supports of $|\partial_1u|$ and $|\partial_2u|$ are disjoint on $B_1'$, we claim that
\be\label{e.sigma-U_+-id}
\sigma=|\partial_1u|=\sqrt{U_-} \  \hbox{ a.e. on $B_1'$},
\ee
where \(U=U_+-U_-\) is the standard decomposition into positive and negative parts. Since $U$ is harmonic in the entire $B_1$ and therefore $U_-$ is locally Lipschitz in $B_1$, the identity \eref{sigma-U_+-id} would imply  \(\sigma\in C_{\text{loc}}^{1/2}(B_1')\).

To prove the claim we first observe that by the nontangential limits of $\partial_ju$, \tref{bdd-harmonic} and \lref{weaksol}, for almost every $y \in B_1'$
\[
G(x)=\lw\partial_2 u\rw^2-\lw\partial_1 u\rw^2 + 2 i \partial_1 u \partial_2 u\rta \tau(y)^2-\sigma(y)^2  \ \hbox{ as } \ x \in B_1^+ \to y\hbox{ non-tangentially.}
\]
 On the other hand, we know that $G$ is defined and holomorphic in $B_1$ so the non-tangential limits must agree with the value of the function
 \[U(y) = G(y) = \tau(y)^2 - \sigma(y)^2 \ \hbox{for a.e. $y\in B_1'$}.\]
 Then, using that $\sigma(y)\tau(y) = 0$ almost everywhere on $B_1'$,
 \[U_-(y) = \sigma(y)^2 \ \hbox{ and } \ U_+ (y)= \tau(y)^2 \ \hbox{ on } \ B_1'.\]
This justifies the claim \eref{sigma-U_+-id}.

Then, since $u$ solves, in the distributional weak sense,
\[-\Delta u = 0 \ \hbox{ in } \ B_1^+ \ \hbox{ with } \ \partial_1 u = \sigma =  \sqrt{U_-} \ \hbox{ on } \ B_1',\]
by standard $C^{1,\alpha}$ estimates for the Neumann problem with a $C^{0,\alpha}$ boundary condition we obtain that $u \in C^{1,\frac{1}{2}}_{\textup{loc}}(B_1^+ \cup B_1')$.

\end{proof}

\section{Conditional regularity in dimension \texorpdfstring{\(d\ge3\)}{s}\label{section.C1alphacondition}}

In this section, we prove \(C^{1,\alpha}\) regularity of a solution \(u\) to the problem \eqref{e.grad-degen-neumann} under the following additional condition. Instead of assuming that \(u(\Ca_u)\) is finite, as stated in Theorem \ref{introtheorem2}, we make an equivalent (see \rref{separation}) assumption with the relevant parameter more clearly quantified:
\begin{equation}\tag{\(\textbf{A}_\delta\)}\label{e.adelta}
    \parbox{.9\textwidth}{\center \(u(\Ca_u)\) is a finite set so that for any connected components $I$ and $J$ of $\Ca_u$\\ with $u(I) \neq u(J)$ the separation condition $\dist(I,J) \geq \delta$ holds.}
\end{equation}

In terms of this separation hypothesis, we aim to prove the $C^{1,\alpha}$ regularity result.

 \begin{theorem}\label{t.C1alphaestimate}
     Let \(u\) be a viscosity solution to \eqref{e.grad-degen-neumann} that satisfies condition \textup{\eref{adelta}}, then there is a small \(\alpha=\alpha(d)\in (0,1)\) such that \(u\in C_{\emph{loc}}^{1,\alpha}(B_1^+\sqcup B_1')\) and there is a constant \(C=C(d)\delta^{-\alpha}>0\) such that
     \[
     \norm{u}_{C^{1,\alpha}\lb\overline{B_{1/2}^+}\rb} \le C  \norm{u}_{L^\infty(B_1^+)}.
          \]
    \label{C1aplhaestimate}
 \end{theorem}
\begin{remark}
    The controlling constant \say{\(C=C(d)\delta^{-\alpha}\)} can actually be replaced by \say{\(C(d,N)\)} with \(N=\# u(\Ca_u \cap B_1)\) (see Remark \ref{replacegapbynumber}). We retain the current exposition for the convenience of the proof.
\end{remark}

The proof will make use of the well-known Almgren frequency formula, which has seen frequent use in the thin obstacle problem \cites{fernándezreal2020obstacle,petrosyan2011monotonicity,garofalo2009monotonicity,athanasopoulos2008structure}. The main reason for our conditioning on hypothesis \eqref{e.adelta} is to guarantee the monotonicity of the frequency function.  It will be made clear in the computations in \sref{almgren} that the possible occurrence of infinitely many connected components of $\Ca_u$ piling up on a single point seems to ruin the monotonicity property.

\begin{remark}\label{r.separation}
   Even though the condition \eqref{e.adelta} is somewhat artificial because we cannot verify it in many interesting cases, it is indeed satisfied in the case of the classical Signorini problem and it demonstrates the central difficulty of our problem \eqref{e.grad-degen-neumann}. The Signorini problem corresponds to the case that \(u(\Ca_u)=\{0\}\) is a singleton and \emph{also} $u \leq 0$ on $B_1'$. The singleton case also includes the cases of the sign-reversed Signorini problem as introduced in Example \ref{quasisolexamp2}. However, we are not able to make any general guarantee on when a particular boundary condition may admit a solution to this sign-reversed Signorini problem.\label{remarkonA}
        \end{remark}
        \begin{remark}\label{r.separation2}
        As mentioned above the hypothesis that $u(\Ca_u)$ is finite, and hypothesis \eqref{e.adelta} are in fact equivalent, the latter just quantifying a useful parameter.  Let \(u\) be a viscosity solution with \(\norm{u}_{L^\infty(B_1^+)} =1\) so that $u$ is Lipschitz with Lipschitz constant at most \(L=L(d)>0\). Suppose additionally that $u(\Ca_u)$ finite. Then call $\delta = L^{-1}\min\{|u(z)-u(w)|: \ z, w \in \Ca_u\textup{ and }u(z)\ne u(w)\}$, which is positive due to the set \(u(\Ca_u)\) being finite. Let \(I,J\) be a pair of components of \(\Ca_u\) such that \(u(I)\ne u(J)\). By Lipschitz continuity of \(u\), we have
\[
\dist(I,J) \ge \frac{|u(x)-u(y)|}{L} \ge \delta,
\]
where \(x\in I,y\in J\). 

\end{remark}

\subsection{Almgren monotonicity formula} \label{s.almgren}
In this section, we will study the monotonicity of the Almgren frequency function for the gradient degenerate Neumann problem \eqref{e.grad-degen-neumann}. {Due to the above remarks it suffices to consider the case that 
\be\label{e.single-component}
u(\Ca_u\cap B_1) = \{0\}.
\ee
When we prove \tref{C1alphaestimate} below we will just make an initial re-scaling to a ball of radius $\delta$ to achieve this hypothesis. The initial scaling determines the dependence on $\delta$ in the theorem.}

The following computations, if not particularly mentioned, are obtained after a mollification procedure and an appropriate use of Lemma \ref{weaksol}. Let \(u\) be the even extension of a viscosity solution to \eqref{e.grad-degen-neumann} such that \(u(0)=0\). Consider the frequency functional
\begin{equation}\label{e.frequency-defn}
    N(r)=\frac{r \int_{B_r} |\gd u|^2}{\int_{\pt B_r} u^2}=\frac{r D(r)}{H(r)}.
\end{equation}
Differentiating the denominator gives
\be
H'(r)=\frac{d-1}{r}H(r) + 2 \int_{\pt B_r} u\pt_\nu u,\label{derivativeofHr}
\ee
where \(\pt_\nu\) is the unit outer normal derivative on \(\pt B_r\). Now we aim to integrate by parts in the second term. Recall that $\sigma$, as defined in \eref{sigmataodef}, is the nontangential limit of $\partial_1 u$ on $B_1'$ from $B_1^+$.  We can justify, using the distributional weak formulation as discussed in Lemma \ref{weaksol}, that the distributional Laplacian of $u$ is given by
\[\Delta u = 2 \sigma d\mathcal{H}^{d-1}|_{B_1'} \ \hbox{ in } \ B_1.\]
Using this identity we find
\be\label{Hderivative}
\begin{split}
    H'(r)&=\frac{d-1}{r}H(r) + 2 \int_{B_r} |\gd u|^2 + 4 \int_{B_r'} u \sigma\\
&=\frac{d-1}{r}H(r) + 2 D(r) + 4 c(r).
\end{split}
\ee
\begin{remark}
    This final term $c(r):= \int_{B_r'} u \sigma$ is a major difficulty that we are currently only able to deal with via conditioning on the hypothesis \eqref{e.adelta}, which has allowed us to reduce to the case $u(\Ca_u \cap B_1) = \{0\}$. In this case, we have \(u\equiv0\) on \(\Ca_u\cap B_1\), and on the other hand, we observe by Corollary \ref{characterizesigma} that \(\sigma=0\) a.e. on \(\lb\Na_u\sqcup\Gamma_u\rb\cap B_1\), which shows that \(\sigma u = 0\) a.e. on the whole \(B_1'\) and hence \(c(r)=0,\,0<r<1\).
\end{remark}

On the other hand, we also have after a mollification procedure, Rellich's formula
\[
\begin{split}
  \int_{\pt B_r} |\gd u|^2 &= \frac{d-2}{r} \int_{B_r} |\gd u|^2  + 2 \int_{\pt B_r} \lb\pt_\nu u\rb^2 - \frac{2}{r} \int_{B_r} (x\cdot \gd u) \Delta u\\
  &=\frac{d-2}{r} \int_{B_r} |\gd u|^2  + 2 \int_{\pt B_r} \lb\pt_\nu u\rb^2, 
\end{split}
\]
where \(\Delta u= 2 \sigma \,d\restr{\mathcal{H}^{d-1}}{\Ca_u}\) and \((x\cdot \gd u)\sigma=(x'\cdot \gd' u)\sigma = 0\) a.e. on \(B_1'\) due to \lref{weaksol}. Collecting these computations we have proved the monotonicity formula for \(u\)
\[
\begin{split}
    \frac{N'(r)}{N(r)}&=\frac{1}{r} + \frac{D'(r)}{D(r)} - \frac{H'(r)}{H(r)}\\
    &=2\lb\frac{\int_{\pt B_r} \lb\pt_\nu u\rb^2}{\int_{\pt B_r} u\pt_\nu u} - \frac{\int_{\pt B_r} u\pt_\nu u}{\int_{\pt B_r} u^2}\rb\\
    &\ge 0.
\end{split}
\]
Notice that if \(N(r)=\kappa\) for \(0<r<1\) then \(N'(r)=0\) and by the above Cauchy-Schwartz inequality we know that there is \(g(r)\) for each \(0<r<1\) such that
\[
\pt_\nu u = g(r) u.
\]
To determine \(g\) we observe on the other hand, 
\[
r\frac{d}{dr}\log H(r) = d-1 + 2N(r)=d-1+2\kappa,
\]
which implies that \(H(r)=H(1) r^{2\kappa+d-1}\) and by \eqref{derivativeofHr}
\[
g(r) \int_{\pt B_r} u^2 = \int_{\pt B_r} u\pt_\nu u=\frac{1}{2} \lb H'-\frac{d-1}{r}H\rb = \frac{\kappa}{r} H(r),
\]
and thus \(g(r)\equiv\kappa/r\). This implies that \(u\) is a \(\kappa\)-homogeneous function. We summarize the above computations in a theorem.
\begin{theorem}[Almgren Monotonicity Formula]
    Let $d\geq 2$ and \(u\) be a viscosity solution to \eqref{e.grad-degen-neumann} in $B_1^+$, evenly extended to $B_1$, which has $u(\mathcal{C}_u \cap B_1) = \{0\}$, \(u(0)=0\) and $0 \in \Gamma_u$. Then the quantity
    \[
    N(r)=N(r,u) = \frac{r \int_{B_r} |\gd u|^2}{\int_{\pt B_r} u^2}
    \]
    is monotone increasing in \(0<r<1\). Moreover, if \(N(r)\equiv\kappa\) for all \(0<r<1\) then \(u\) is a \(\kappa\)-homogeneous function in \(B_{1}\).
    \label{t.almgrenmonotone}
\end{theorem}

\subsection{Pointwise differentiability}

Let \(u\) be a viscosity solution evenly extended to the whole ball \(B_1\) that satisfies \eqref{e.adelta}. We would like to consider the following blow-up sequence at a fixed point \(0\in B_1'\), \(tx\in B_1\) and we also assume \(u(0)=0\)
\[
u_t(x) = \frac{u(tx)}{t} \ \hbox{ for } \ 0 < t < 1.
\]
Notice that \(\norm{u_t}_{C^{0,1}(\bB_1)}\le2 \norm{u}_{C^{0,1}(\bB_{1/2})}\) as \(t\rta0^+\), and when \(t\) is sufficiently small by condition \eqref{e.adelta}, \(u_t(\Ca_{u_t})=\{0\}\). In particular, we have after passage to a subsequence \(t_k\rta0\), there is an \(u_0\in C^{0,1}(\bB_1)\) such that \(u_{t_k}\rta u_0\) uniformly in \(\bB_1\).

\begin{lemma}[Blow-up limit at the free boundary]\label{l.blow-up-at-FB}
    If \(0\in \Gamma_u\) and \(u(0)=0\) then
    \[\frac{u(tx)}{t} \to 0 \ \hbox{ as $t\to 0$ uniformly on $\overline{B_1^+}$.}\]
\end{lemma}

\begin{remark}
   The differentiability of a solution \(w\) to \eqref{e.grad-degen-neumann} up to \(B_1'\) can be obtained partially by using interior regularity of Neumann or Dirichlet problems near \(\Na_w\) and \(\Ca_w\) respectively. This lemma completes the proof of pointwise differentiability by establishing the differentiability on \(\Gamma_w\).
\end{remark}

\begin{proof}
   We prove by contradiction and assume there is a blow-up limit \(u_0\ne0\) and \(u_{t_k}\rta u_0\) uniformly on $\overline{B_1^+}$ as \(k\rta\infty\). By the argument in \eqref{Hderivative}, we know that \(H_{u_0}(r)>0\) for all \(0<r<1\) and hence for each \(r\) we can find \(k_0(r)>0\) such that \(H_{u_{t_k}}(r)>\frac{1}{2}H_{u_0}(r)>0\) for all \(k>k_0(r)\).
    
To prove the lemma, we notice that by the Almgren monotonicity formula (see \tref{almgrenmonotone}),
\[
N(r,u_{t_k})=N(rt_k,u)\rta N(0^+,u)=\kappa\ge0, 
\]
as \(k\rta\infty\), and on the other hand, we know that
\[
v_k=\frac{u_{t_k}}{\lb\int_{\pt B_r} u_{t_k}^2\rb^{1/2}}=\frac{u_{t_k}}{H_{u_{t_k}}^{1/2}(r)}
\]
satisfies for sufficiently large \(k\) and a constant \(C>0\) independent of \(k\)
\be\label{importantLinftyboundaindifferentiability}
\norm{v_k}_{L^2(\pt B_r)}=1,\,\norm{v_k}_{L^\infty(B_r)}\le C,\,\hbox{ and } \int_{B_r} |\gd v_k|^2 \le N(rt_k,u) \le N(1, u) \le C.
\ee
By interior estimates of harmonic functions, we have local uniform convergence of \(\gd v_k\) to \(\gd u_0/\norm{u_0}_{L^2(\pt B_r)}\) in \(B_r\setminus B_r'\), and because of boundedness of their \(L^\infty\) norms by the Lipschitz estimate \ref{lipschitzregu}, we have the strong convergence of \(v_k\) in \(H^1(B_r)\). This shows that
\[
N(r,u_{t_k})= \frac{r \int_{B_r} |\gd u_{t_k}|^2}{\int_{\pt B_r} u_{t_k}^2}\rta N(r,u_0),\,\hbox{ as }k\rta\infty.
\]
From this we obtain that \(N(r,u_0)=\kappa\) for all \(r>0\), and therefore \(u_0\) is \(\kappa\)-homogeneous on \(\R^d\) according to \tref{almgrenmonotone}. If \(u_0\ne0\) then it can only be a 1-homogeneous function by Lipschitz regularity.

Thus $u_0$ has the form
\[
u_0(r,\theta)=rh(\theta),\,r\ge0,\,\theta\in \pt B_1\cap\{x_1\ge0\}.
\]
Since \(u_0\) is also a viscosity solution to \eqref{e.grad-degen-neumann}, the function \(h\) must satisfy
\be\label{equationforthesphericalfunctionh}
\begin{cases}
    \Delta_\theta h(\theta) + (d-1)h(\theta)=0,&\theta\in \pt B_1\cap\{x_1>0\},\\
    \min\left\{-\partial_{\Vec{n}}h(\theta), \sqrt{|\gd_\tau h(\theta)|^2 + h^2(\theta)}\right\}=0,&\theta\in \partial' B_1',
\end{cases}
\ee
where \(\Delta_\theta\) is the Laplace-Beltrami operator on \(\pt B_1\cap\{x_1>0\}\), \(\partial_{\Vec{n}}\) is the outer normal derivative of \(\pt B_1\cap\{x_1\ge0\}\) on the boundary \(\pt'B_1'\) and \(\gd_\tau\) the tangential gradient on \(\pt'B_1'\). 

We first claim that 
\[
\int_{\pt' B_1'} h =0.
\]
This can be obtained by the following Green's formula with the linear function $\ell(x) = p \cdot x$, which all satisfy $\Delta_\theta \ell + (d-1)\ell = 0$ on the sphere,
\[0 = \int_{\partial B_1 \cap \{x_1>0\}} \ell (\Delta_\theta h +(d-1) h) - h (\Delta_\theta \ell + (d-1)\ell)  = \int_{\partial'B_1'} \ell \pt_{\Vec{n}} h - h \pt_{\Vec{n}}  \ell =  \int_{\partial'B_1'} \ell\pt_{\Vec{n}} h + hp_1\]
The claim is proved by taking \(p = e_1\).

On the other hand, we claim that we can apply \(h\) and do integration by parts to obtain
\be
\int_{\pt B_1\cap\{x_1>0\}}(d-1) h^2 - |\gd_\theta h|^2 + \int_{\pt' B_1'} h\partial_{\Vec{n}} h=\int_{\pt B_1\cap\{x_1>0\}}(d-1) h^2 - |\gd_\theta h|^2=0.\label{homogeneous1comfromhere}
\ee
For the above, we used that
\be\label{thecancelingofhdh}
\int_{\pt' B_1'} h\partial_{\Vec{n}} h =0.
\ee
This is because, by Lemma \ref{weaksol}, \(\min\{\pt_1 u_0(x), |\gd' u_0|(x)\}=0\) for almost every \(x\in B_1'\) in the nontangential convergence sense, which by homogeneity implies
\[
\min\left\{-\partial_{\Vec{n}}h(\theta), \sqrt{|\gd_\tau h(\theta)|^2 + h^2(\theta)}\right\}=0,
\]
for almost all \(\theta\in \pt'B_1'\) also in the sense of nontangential convergence. By a similar proof of the second part of Lemma \ref{weaksol}, we can justify the validity of integration by parts and the fact that \( h \pt_{\Vec{n}} h =0\) almost everywhere on \(\pt' B_1'\).

Now we can subtract off a linear function by considering \(\tilde{h}=h-\gamma x_1\) so that \(\tilde{h}\) also satisfies \eqref{homogeneous1comfromhere} and \(\int_{\pt B_1\cap\{x_1>0\}} \tilde{h}=0\). For the mean zero condition just choose
\[
\gamma=\frac{\int_{\pt B_1\cap\{x_1>0\}} h}{\int_{\pt B_1\cap\{x_1>0\}} (x_1)_+}.
\]
Notice that \(\gamma\le0\) because
\[
(d-1)\int_{\pt B_1\cap\{x_1>0\}} h=\int_{\pt B_1\cap\{x_1>0\}} -\Delta_\theta h
= \int_{\pt' B_1'} \partial_{\Vec{n}} {h}
\le 0.
\]
 As for \eqref{homogeneous1comfromhere}, it suffices to discuss \eqref{thecancelingofhdh}, which is because \(\Delta_\theta \tilde{h}+\tilde{h}=0\) in $\partial B_1 \cap \{x_1>0\}$ and so
\[
\begin{split}
   \int_{\pt' B_1'} \tilde{h}\partial_{\Vec{n}} \tilde{h}&=\int_{\pt' B_1'} h\partial_{\Vec{n}} \tilde{h} \\
 &=\int_{\pt' B_1'} h\partial_{\Vec{n}} h-\gamma \int_{\pt' B_1'} h\\
 &=0.
\end{split}
\]
If \(\tilde{h}\equiv0\) then \(u_0\equiv\gamma |x_1|\) and because \(u_0\) is also a solution to \eqref{e.grad-degen-neumann} and by prior discussions \(\gamma\le0\), we obtain \(\gamma =0\). 

If \(\tilde{h}\not\equiv0\), then 
\be
\frac{\int_{\pt B_1\cap\{x_1>0\}}|\gd_\theta \tilde{h}|^2}{\int_{\pt B_1\cap\{x_1>0\}}\tilde{h}^2}  =d-1.\label{theeigenh}
\ee
On the other hand, we know that the second Neumann eigenvalue
\[
\lambda:=\lambda_{N,2}(\pt B_1\cap\{x_1>0\})=  \inf_{\begin{subarray}{c}0\ne g \in H^1(\pt B_1\cap\{x_1>0\})\\\int_{\pt B_1\cap\{x_1>0\}} g=0\end{subarray}} \frac{\int_{\pt B_1\cap\{x_1>0\}}|\gd_\theta g|^2}{\int_{\pt B_1\cap\{x_1>0\}}g^2}>0
\]
is equal to \(d-1\) and the minimizing functions \(g\) must be restrictions of linear functions of the form \(p'\cdot x\) with \(p'\in \R^d\) and \(p'\cdot e_1=0\) \cite{Stein1970}*{Chapter 3}. By \eqref{theeigenh}, we know that \(\tilde{h}\) must be equal to a Neumann second eigenfunction and so
\[
u_0(x)=\gamma |x_1| + p'\cdot x
\]
is a smooth solution to \eqref{e.grad-degen-neumann} for some \(p'\).  Because, again, \(u_0\) is a viscosity solution to \eqref{e.grad-degen-neumann} and \(\gamma\le0\), we know that \(\gamma=0\) and \(u_0(x) = p'\cdot x\) for some \(p'\cdot e_1=0\).

To show that \(p'=0\), we argue by contradiction and assume that there is a sequence \(s_j\rta0^+\) such that \(f_j=u_{s_j}\) converges locally uniformly to \(p'\cdot x\) for some \(p'\ne0\) as \(j\rta\infty\). Now, we obtain a sequence \(\ep_j\rta0^+\) such that
\[
\overline{f}_j=\frac{f_j-p'\cdot x}{\ep_j}
\]
is uniformly bounded on \(B_{1}^+\) and satisfies boundary condition
\[
\min\{\pt_1 \overline{f}_j, |\gd' \overline{f}_j + p'/\ep_j|\}=0,\,\hbox{ on }B_{1}'
\]
in the viscosity sense. According to Corollary \ref{consequenofuniformlipindeppprime}, we know that when \(j\) is large \(\overline{f}_j\) must satisfy the zero Neumann boundary condition on \(B_{1/2}'\) and hence \(u_{s_j}=f_j\) must also satisfy this condition, which contradicts the assumption that \(0\in \Gamma_{u_{t}}\) for all \(t>0\).
\end{proof}

\subsection{Improvement of flatness} Combining Lemma \ref{l.blow-up-at-FB}, the interior regularity of zero Neumann and Dirichlet problems, we obtain everywhere differentiability (including nontangential directions) of \(u\) on \(B_1'\). To obtain the improvement of flatness results, we plan to first consider the following three different cases respectively
\begin{enumerate}[label = (\Roman*)]
    \item The first case deals with the small gradients, and it essentially corresponds to the original boundary condition
    \[
    \min\{\pt_1 u , |\gd' u|\} = 0.
    \]
    \item The second case deals with large tangential gradients, and it essentially corresponds to the tangentially modified boundary condition
    \[
    \min\{\pt_1 u , |\gd' u + q'|\} = 0,
    \]
    with \(q'\cdot e_1=0\) and \(|q'|\) large.
     \item The last case deals with large inner normal derivatives, and it essentially corresponds to the normal modified boundary condition
    \[
    \min\{\pt_1 u + q_1, |\gd' u |\} = 0,
    \]
    with \(q_1>0\) large.
\end{enumerate}
In fact, it turns out that these three cases are enough to derive the full improvement of flatness results. This can be obtained by Lemma \ref{dichotomyofgradients} that shows the dichotomy of gradients: a solution \(u\), with \(\osc u \le 1\), to 
\[
\min\{\pt_1 u + q_1, |\gd' u + q'|\} = 0
\]
would require \(|\min\{q_1,|q'|\}|\) to be bounded by a constant independent of \(u\). Now, with a bounded cost, we can modify \(u\) so that it is contained in one of the three categories we introduced above.

\begin{lemma}[Improvement of Flatness I]
       Let \(u\) be a viscosity solution to \eqref{e.grad-degen-neumann} in $B_1$ with either $u(\Ca_u \cap B_1) =\{ 0\}$ or \(\Ca_u \cap B_1=\emptyset\), and \(\osc_{B_1^+}{u}\le T_1\) for some fixed \(T_1>0\). There is a \(1/2>\mu=\mu(d,T_1)>0\) such that for each \(u\) as described there is \(1/2>\nu=\nu(u)\ge \mu\) such that
       \be
      \inf_{p\in \R^d} \osc_{B_{\nu}^+}\{u-p\cdot x\}\le \frac{1}{2}\nu.\label{improoffltainequalityatscalenu}
       \ee
 
       \label{improvofflatnessatboundary}
\end{lemma}

\begin{remark}
Notice that here \(T_1\) is absorbed in \(\nu(u)\) and \(\mu(d,T_1)\), instead of writing \(\nu T_1\) and \(\mu T_1\). We will use similar notations in the following improvement of flatness lemmas, where \(T_1, T_2\) and \(T_3\) are constants to be determined.\label{abouthteabsorption}
\end{remark}
\begin{remark}\label{replacegapbynumber}
    This lemma can actually be improved to the case that \(\#u(\Ca_u \cap B_1)\le N\) for some fixed positive integer \(N\) with a slight modification to the proof. In this way we can replace the constant \say{\(C=C(d)\delta^{-\alpha}\)} in Theorem \ref{C1aplhaestimate} by \say{\(C(d,N)\)}. That is, the controlling constant can be made independent of the minimal gap in \(u(\Ca_u \cap B_1)\), but solely depending on the dimension and the number of elements in it.

\end{remark}
 
\begin{proof}
According to \lref{blow-up-at-FB}, for blow-ups on $\Gamma_u$, and the interior regularity of the solution on \(\Ca_u\sqcup\Na_u\), we have the following convergence
\[
\inf_{p\in\R^d}\frac{ \osc_{B_r^+}\{u-p\cdot x\}}{r} \rta 0,\,\hbox{ as }r\rta0^+.
\]
Thus for each \(u\) there is a \(\nu'=\nu'(u)\in (0,1)\) such that for every \(0<r<\nu'\), the following inequality holds
\be
\inf_{p\in \R^d} \osc_{B_{r}^+}\{u-p\cdot x\}\le \frac{1}{4}r.\label{intermediateimproveflatness}
\ee
We now define 
\[
\eta(u)=\max\{0<s\le1\,;\, \eqref{intermediateimproveflatness} \hbox{ holds for }r = s\}>0.
\]
It then suffices to show a uniform positive lower bound for \(\eta\)'s since we will obtain \eqref{improoffltainequalityatscalenu} immediately by taking \(\nu=\eta/2\). We argue by contradiction and assume that there exists a sequence of functions \(u_j\) that satisfy the conditions described in the statement while
\be
\eta(u_j) \rta 0 \hbox{ as }j\rta\infty.\label{improvflatarguecontradasumm}
\ee
Because \(\osc_{B_1^+}(u_j) \le T_1\), we know by the Lipschitz estimate, \lref{Lipschitz-estimate}, that (up to a subsequence) \(u_j\) converges locally uniformly to some \(u_\infty\) in \(B_1^+\sqcup B_1'\). By classical viscosity solution theories, \(u_\infty\) also satisfies the conditions as described in the statements. Indeed, it suffices to check that either \(u_\infty\lb\Ca_{u_\infty}\cap B_1\rb=\{0\}\) or \(\Ca_{u_\infty}\cap B_1=\emptyset\). Suppose there is a number \(s\ne0\) such that \(s\in u_\infty\lb\Ca_{u_\infty}\cap B_1\rb\), then we can find a relatively open component \(I_s\subset \Ca_{u_\infty}\cap B_1\) such that \(u_\infty(I_s)=s\). By the local uniform convergence of \(u_j\) to \(u_\infty\) on \(B_1^+\sqcup B_1'\) we know that for some small \(\delta>0\), \(u_j\) are uniformly close to \(s\) on \(I_s\cap B_{1-\delta}\) for large \(j\). By the assumption \(u_j(\Ca_{u_j}\cap B_1)=\{0\}\), we know that \(u_j\) satisfy the zero Neumann boundary condition on \(I_s\cap B_{1-\delta}\), which would imply that \(u_\infty\) also satisfies the zero Neumann boundary condition on \(I_s\cap B_{1-\delta}\), contradicting the assumption that \(I_s\subset \Ca_{u_\infty}\). A similar proof can address the issues as discussed in Remark \ref{replacegapbynumber}.

On the other hand, by the everywhere differentiability of \(u_\infty\), we know that for some small \(\tilde{\eta}>0\) there must be some \(\tilde{p}\in \R^d\) such that
\[
\osc_{B_{\tilde{\eta}}^+}\{u_\infty-\tilde{p}\cdot x\}\le \frac{1}{8}\tilde{\eta},
\]
which contradicts the assumption \eqref{improvflatarguecontradasumm}.

\end{proof}
\begin{lemma}[Improvement of Flatness II]\label{improvofflatnessatboundary2}
    Let \(u\) be a viscosity solution to \eqref{e.grad-degen-neumann} with the boundary condition replaced by
    \[
    \min\{\partial_1 u, |\nabla' u + q'|\} = 0,
    \]
    for some \(q' \in \{0\} \times \mathbb{R}^{d-1}\), and \(\osc_{B_1^+}{u} \leq T_2\) for some fixed \(T_2 > 0\). There exists \(J = J(d,T_2) > 0\) such that if \(|q'| > J\), then there exists \(\iota = \iota(d, T_2) > 0\) with \(1/2 > \iota\) satisfying
    \[
    \inf_{p \in \mathbb{R}^d} \osc_{B_{\iota}^+}\{u - p \cdot x\} \leq \frac{1}{2}\iota. \label{improoffltainequalityatscalenu2}
    \]
\end{lemma}
\begin{proof}
    According to Lemma \ref{lipschitzregu} and Corollary \ref{consequenofuniformlipindeppprime}, we know that the Lipschitz constant \(L=L\lb d,\norm{u}_{L^\infty(B_1^+)}\rb>0\) of \(u\) in \(B_{1/2}^+\) is independent of the choice of \(q'\). If one chooses \(J=2 L \lb d, T_2\rb> L\), then \(u\) satisfies the zero Neumann boundary condition on \(B_{1/2}'\) and \(\osc_{B_{1/2}^+}{u} \leq T_2\), and then the improvement of flatness comes naturally from the smoothness of Neumann solutions.
  
\end{proof}

\begin{lemma}[Improvement of Flatness III]\label{diminishofosci}
    Let \(u\) be a viscosity solution to \eqref{e.grad-degen-neumann} with the boundary condition replaced by
    \[
    \min\{\partial_1 u+q_1, |\nabla' u |\} = 0,
    \]
    for some \(q_1\in \R\), and \(\osc_{B_1^+}{u} \leq T_3\) for some fixed \(T_3 > 0\). There exists \(I = I(d,T_3) > 0\) such that if \(q_1 > I\), then there exists \(\gamma = \gamma(d, T_3) > 0\) with \(1/2 > \gamma\) satisfying
    \be
    \inf_{p \in \mathbb{R}^d} \osc_{B_{\gamma}^+}\{u - p \cdot x\} \leq \frac{1}{2}\gamma. \label{dimiosciineq}
    \ee
\end{lemma}
\begin{proof}
Let us study the family of functions
\[
w=x_1 + \frac{u}{q_1},
\]
where \(q_1>I\) to be chosen. Let \(\ep=T_3/I\), we know that \(w\) are bounded solutions to \eqref{e.grad-degen-neumann} that are uniformly flat in the sense that
\[
x_1-\ep\le w\le x_1 + \ep.
\]
By a similar argument to Lemma \ref{definecu}, we may choose \(I=400\max\{1,T_3\}d\) so that \(\ep \le 1/(400d)\). We then obtain \(B_{1/4}'\subset \Ca_{w}\). This also implies that \(u\equiv C\) for some constant \(C\) on \(B_{1/4}'\). By an odd reflection, \(u-C\) can be extended to a harmonic function in \(B_{1/4}\) with \(\osc_{B_{1/4}}\{u-C\} = \osc_{B_{1/4}}\{u\} \le T_3\). By applying the interior regularity of harmonic function we can determine the constant \(0<\gamma< 1/4\) that satisfies \eqref{dimiosciineq}.

\end{proof}

\begin{lemma}[Dichotomy of Gradients]
    Suppose \(\osc_{B_1^+}(u) \le 1\) solves
    \be\label{e.m1-m'-PDE}
    \bca
   \Delta u = 0,&\hbox{ in }B_1^+\\
   \min\{\pt_1 u + m_1, |\gd' u +  m'|\}=0&\hbox{ on }B_1'
    \eca
    \ee
    in the viscosity sense for some \(m=(m_1,m')\in \R^d\). Then there is a constant \(K=K(d)>0\) such that
    \[
    |\min\{m_1,|m'|\}| \le K.
    \]\label{dichotomyofgradients}
\end{lemma}

\begin{proof}
    Suppose that there exist a sequence \(u_j\) satisfying \(\osc_{B_1^+}(u_j) \le 1\) and \eref{m1-m'-PDE}, but \(|\min\{m_{j,1},|m_j'|\}|\rta\infty\) as \(j\rta\infty\). 
    
    We can assume that either (i) \(\min\{m_{j,1},|m_j'|\}=m_{j,1}\) for all $j$ or (ii) \(\min\{m_{j,1},|m_j'|\}=|m_j'|\) for all $j$.  We consider the two cases separately.
    
    First suppose that \(\min\{m_{j,1},|m_j'|\}=m_{j,1}\). We claim that \(\displaystyle\liminf_{j\rta\infty}|m_j'|/|m_{j,1}|>0\). It suffices to consider the case that \(m_{j,1}<0\). To show this we consider 
    \[
    w_j^1=\frac{u_j}{|m_{j,1}|} + \frac{m_{j,1}}{|m_{j,1}|} x_1,
    \]
    which by assumption on \(u_j\) is a bounded sequence of functions. Observe that \(w_j^1\) satisfy the boundary condition
    \[
    \min\{\pt_1 w_j^1,|\gd' w_j^1 + m_{j}'/|m_{j,1}||\}=0.
    \]
    If a subsequence \(|m_j'|/|m_{j,1}|\rta 0^+\), then by the compactness of \(w_j^1\) (this is because of the Lipschitz estimate, Lemma \ref{lipschitzregu}), after passage to a subsequence, \(w_j^1\) would converge locally uniformly to a viscosity solution to \eqref{e.grad-degen-neumann}. On the other hand, this sequence converges uniformly to \(-x_1\), which is not a viscosity solution to \eqref{e.grad-degen-neumann} and shows the claim. 
    
    We further claim that it's impossible that \(|m_{j,1}|\rta\infty\). We divide into cases depending on whether $m_{j}'/|m_{j,1}|$ stays bounded or not. If $m_{j}'/|m_{j,1}|$ stays bounded then, after passage to a subsequence, we can suppose that \(m_{j}'/|m_{j,1}|\rta m_\infty'\in \{0\}\times\R^{d-1}\) with \(|m_\infty'|>0\) (by the first claim). This implies that the limit function \(w_\infty^1=\pm x_1\) is a viscosity solution of \(\min\{\pt_1 w_\infty^1,|\gd' w_\infty^1 +m_\infty'|\}=0\) on $B_1'$, which is not true. If \(m_{j}'/|m_{j,1}|\rta\infty\) then by Corollary \ref{consequenofuniformlipindeppprime}, for sufficiently large \(j\), \(w_j^1\) satisfies zero Neumann boundary condition on \(B_{1/2}'\), which also contradicts the form of the limit function \(w_\infty^1=\pm x_1\) because of compactness of zero Neumann solutions with bounded oscillation on $B_{1/2}^+$.

 In the case \(\min\{m_{j,1},|m_j'|\}=|m_j'|\) we define
    \[
    w_j^2=\frac{u_j}{|m_j'|} + \frac{m_j'}{|m_j'|}\cdot x + \frac{m_{j,1}}{|m_{j}'|} x_1. 
    \]
    This implies that \(w_j^2\) satisfies the original boundary condition \eqref{e.grad-degen-neumann}. In this case, we observe that when \(j\) is large, \(w_j^2\) (after passage to a subsequence) has one-sided flatness in the way that
   \[
   w_j^2 (x) \ge m_\infty'\cdot x + x_1 -o_j(1),\,\hbox{ for all }x\in \overline{B_1^+}
   \]
    with \(\displaystyle \lim_{j\rta\infty} \frac{m_j'}{|m_j'|}=m_\infty'\in \{0\}\times \R^{d-1}\) in a proper subsequence, \(|m_\infty'|=1\), and \(|w_j^2(0)|\le 1/|m_j'|=o_j(1)\). This contradicts the supersolution condition of \(w_j^2\) for large \(j\) according to a similar argument to the proof of Lemma \ref{characterizesigma}.
    \end{proof}

\subsection{\texorpdfstring{$C^{1,\alpha}$}{s}-iteration}
Before proving the theorem, we write a lemma that summarizes the improvement of flatness results in the previous subsection. Define the set of allowed gradients
\[
\Ta=\{(q_1,q')\in \R^d\,;\,\min\{q_1,|q'|\}=0\},
\]
and for \(R>0\) we define the fattening
\[
\Ta_R=\{(q_1,q')\in \R^d\,;\,|\min\{q_1,|q'|\}|\le R\}.
\]
\begin{lemma}\label{simpleiteration}
    Suppose that \(u\) is harmonic in $B_1^+$ with
    \[
    \min\{\pt_1 u +q_1,|\gd' u +q'|\}=0 \ \hbox{ on } \ B_1',
    \]
    for some \(q=(q_1,q')\in \Ta_R\) and \(\osc_{B_1^+}(u)\le 1\). If \(v(\Ca_{v} \cap B_1)\) has at most one element with \(v:= q\cdot x +u\) (which is a viscosity solution to \eqref{e.grad-degen-neumann}; see Remark \ref{generalizedCauNauGammau}), then there is a \(1/2>\nu=\nu(u)>0\) and \(\kappa=\kappa(d,R)>0\) such that \(\nu\ge \kappa\) and 
    \[
   \inf_{p\in\R^d} \osc_{B_\nu^+} (u-p\cdot x) \le \frac{1}{2}\nu.
    \]
\end{lemma}

\begin{proof}
We define \(\Bar{q}\in \Tc\) as follows
\[
{\bar{q}}:=\bca 
q',&\hbox{ if } \min\{q_1,|q'|\}=q_1\\
q_1 e_1,&\hbox{ if }\min\{q_1,|q'|\}=|q'|,
\eca \ \hbox{ and } \ {\tilde{q}:= q - \bar{q}}.
\]
Notice that \(|\tilde{q}|\le R\) because \(q\in \Tc_R\) and
\[
{\tilde{q}}:=\bca 
q_1 e_1,&\hbox{ if } \min\{q_1,|q'|\}=q_1\\
q',&\hbox{ if }\min\{q_1,|q'|\}=|q'|.
\eca 
\]
Now the function \(w:= \tilde{q}\cdot x + u\) satisfies \eqref{e.grad-degen-neumann} with the boundary condition replaced by 
\[
\min\{\pt_1 w +\overline{q}_1,|\gd' w +\overline{q}'|\}=0,
\]
where \(\overline{q}=(\overline{q}_1,\overline{q}')\in \Ta\), and \(\osc_{B_1^+}(w) \le 1+2 R\). 

If \(|\overline{q}| \ge 2\max\{I(d,1+2R),J(d,1+2R)\} =: \Lambda(d,R)\), then the improvement of flatness follows from Lemma \ref{improvofflatnessatboundary2} and \ref{diminishofosci}, where \(J(d,T_2)\) and \(I(d,T_3)\) are defined. Here we choose \(T_2=T_3=1+2R\).

If \(|\overline{q}| \le \Lambda(d,R)\), then we know that \(v=u+q\cdot x\) satisfies the original \eqref{e.grad-degen-neumann} with \(\osc_{B_1^+}(v) \le 1+ 2 R + 2\Lambda(d,R)\). Since \(v(\Ca_v)\) has at most one element, we can apply Lemma \ref{improvofflatnessatboundary} to obtain the improvement of flatness with \(T_1=1+ 2 R + 2\Lambda(d, R)\).

We may define \(\kappa(d,R)=\min\{\mu(d,T_1),\iota(d,T_2),\gamma(d,T_3)\}>0\).
\end{proof}

\begin{proof}[Proof of Theorem \ref{C1aplhaestimate}]
    First, we re-scale to reduce to the case that \(u(\Ca_u \cap B_1)=\{0\}\). If \(u\) satisfies condition \eqref{e.adelta}, then we can consider \(w(x):=\frac{u(\delta x + x_0)-u(x_0)}{\delta}\) and observe that
 \(w(\Ca_w  \cap B_1)\) has at most one element. Furthermore $w$ will be bounded independent of \(\delta\) due to the Lipschitz estimate (Lemma \ref{lipschitzregu}) and
    \[[u]_{C^{1,\alpha}(B_{\delta/2}(x_0))} \leq \delta^{-\alpha}[w]_{C^{1,\alpha}(B_{1/2})}.\]

    To prove that \(u\) is \(C_{\text{loc}}^{1,\alpha}(B_1^+\sqcup B_1')\) it suffices to show that \(u\) is \(C^{1,\alpha}\) at \(0\) in the sense that there is \(C=C(d)>0\) and \(p\in \R^d\) such that
    \be
    \osc_{B_r^+} (u-p\cdot x) \le C r^{1+\alpha},\,r\in (0,1).\label{contiC1alpha}
    \ee
    Indeed, by classical arguments this implies \(C_{\textup{loc}}^{1,\alpha}\) of \(u\) when restricted to \(B_1'\), which then implies the \(C_{\textup{loc}}^{1,\alpha}\) regularity of \(u\) in the whole \(B_1^+\sqcup B_1'\) by classical estimates for Dirichlet problems.
    
    To show \eqref{contiC1alpha}, it suffices to find a sequence \((q_k,r_k)\) such that \(q_k\in \R^d\), and 
    \be
    \osc_{B_{r_k}^+} (u-q_k\cdot x) \le r_k^{1+\alpha} \ \hbox{ for all $k \in \N$}\label{discreteC1alpha}
    \ee
     where \(r_k\rta0\) as \(k\rta\infty\) and \(\frac{1}{2}\ge \frac{r_{k+1}}{r_k} \ge \kappa(d)>0\) for all \(k\) and some $\kappa(d)>0$. If this is done then the constant \(C\) in \eqref{contiC1alpha} would take the form \(\kappa^{-(1+\alpha)}\).

    We start with \(u_0=u\) such that \(\osc_{B_{1}^+}(u_0)\le 1\) and \(q_0=0\). We fix \(R= K(d)\) as in Lemma \ref{dichotomyofgradients}, \(T_1=1+2K(d)+2\Lambda(d,K(d))\) and \(T_2=T_3=1+2K(d)\) as in the proof of Lemma \ref{simpleiteration}. By applying Lemma \ref{simpleiteration} to \(u_0\), we obtain \(1/2>\nu_1\ge\kappa(d,K(d))=:\kappa(d)>0\) and \(p_1\in \R^d\) such that
    \[
    \osc_{B_{\nu_1}^+} (u-p_1\cdot x) \le \frac{1}{2} \nu_1.
    \]
    We now choose \(\alpha>0\) small so that \(\kappa^\alpha>1/2\). Suppose for \(k\ge1\) we have already constructed \(q_k\in \R^d\) (notice that we already have \(q_1=p_1\) and \(r_1=\nu_1\)) such that \eqref{discreteC1alpha} holds true. We then consider for \(x\in B_1^+\sqcup B_1'\)
    \[
    u_k(x) = r_k^{-1-\alpha} \lb u(r_k x) - q_k \cdot (r_k x)\rb.
    \]
    Notice that \(\osc_{B_1^+}(u_k) \le 1\) and \(u_k\) satisfies \eqref{e.grad-degen-neumann} with boundary condition replaced by
    \[
    \min\{\pt_1 u_k + r_k^{-\alpha} q_{k,1},|\gd' u_k + r_k^{-\alpha} q_k'|\}=0.
    \]
    By Lemma \ref{dichotomyofgradients} we obtain that \(r_k^{-\alpha} q_k \in \Ta_{K}\), and then we may apply Lemma \ref{simpleiteration} with \(R=K(d)\) to \(u_k\) and obtain \(1/2>\nu_{k+1}\ge \kappa\), \(p_{k+1}\in \R^d\) such that
    \[
    \osc_{B_{\nu_{k+1}}} (u_k - p_{k+1}\cdot x) \le \frac{1}{2} \nu_{k+1}.
    \]
    Setting \(r_{k+1}=r_k\nu_{k+1}\) and \(q_{k+1}=q_k+r_{k}^{\alpha}p_{k+1}\), we will obtain
    \[
    \osc_{B_{r_{k+1}}^+} (u - q_{k+1}\cdot x) \le r_k^{1+\alpha} \frac{1}{2} \nu_{k+1} \le r_{k+1}^{1+\alpha}.
    \]
   
\end{proof}

\section{Conditional optimal regularity in \texorpdfstring{$d \geq 3$}{x}}\label{section.conditionopti}

In this section we discuss the optimal \(C^{1,1/2}\) regularity of a viscosity solution \(u\) to \eqref{e.grad-degen-neumann} satisfying the condition $\#u(\Ca_u \cap B_1)<+\infty$. The proof uses the Almgren monotonicity formula \tref{almgrenmonotone} again in a similar way to results for the thin obstacle problem. Let us start with a more detailed version of Theorem \ref{introtheorem2}.

 \begin{theorem}
     Let \(u\) be a viscosity solution to \eqref{e.grad-degen-neumann} that satisfies condition \emph{\eqref{e.adelta}}, then there is a constant \(C(d,\delta)=C(d)\delta^{-1/2}>0\) such that
     \[
     \norm{u}_{C^{1,1/2}(\overline{B_{1/2}^+})} \le C \norm{u}_{L^\infty(B_1^+)}.
          \]
    \label{optimalestimate}
 \end{theorem}

To obtain optimal regularity we would like to consider functions of the form
\[
    w_t(x)=\frac{u(tx)}{\lb\frac{1}{t^{d-1}}\int_{\pt B_t} u^2\rb^{1/2}},
\]
where \(u\) is evenly extended to the whole ball \(B_1\), \(tx\in B_1\). We would like to consider the blow-up limit of \(w_t\) at the base point \(0\in \Gamma_u\) and \(u(0)=0\). Notice that \(w_t\) is controlled in the sense that 
\be\label{boundaryL2ofw2is1}
\norm{w_t}_{L^2(\pt B_1)} =1.
\ee
On the other hand, by the Almgren monotonicity, we have
\[
\int_{B_1} |\gd w_t|^2 = N(1,w_t)=N(t,u) \le N(T,u),\,0<t\le T.
\]
 Unlike \eqref{importantLinftyboundaindifferentiability}, we don't immediately have a uniform \(L^\infty\) bound for \(w_t\), and thus these \(L^2\) estimates are not enough for working with the blow-up limits, and we need an additional \(L^2\) to \(L^\infty\) estimate to proceed.

As discussed in \rref{separation2} it suffices to consider the case that $u(\Ca_u \cap B_1) = \{0\}$, since if $u$ satisfies \emph{\eqref{e.adelta}} then it satisfies $\#u(\Ca_u \cap B_\delta) = 1$ in all balls of radius $\delta$.

\subsection{An \texorpdfstring{\(L^2\)}{s} to \texorpdfstring{\(L^\infty\)}{s} estimate}

When $u(\Ca_u \cap B_1) = \{0\}$ then $u$ also solves the following \emph{no-sign Signorini problem}
\be
\bca
\Delta w=0,&\hbox{ in }B_1^+\\
\min\{\pt_1 w,|w|\} = 0,&\hbox{ on }B_1'\\
w=g,&\hbox{ on }\pt B_1\cap\{x_1\ge0\},
\eca\label{localproblemtoe}
\ee
in the viscosity sense.  Note that solutions to this problem may only be $C^{1/2}$ regular, for example, $w(x) = \hbox{Re} (x_2+i|x_1|)^{1/2}$ solves, but we are just using this as a convenient setting to prove the $L^2 \to L^\infty$ estimate.
\begin{remark} 
    Given a viscosity solution \(w\) to \eqref{localproblemtoe}, we will obtain a partition
    \[
    B_1'=\Ca_w\sqcup\Gamma_w\sqcup\Na_w,
    \]
    where \(\Na_w=\{|w|>0\}\cap B_1'\) is open and \(\Gamma_w=\partial' \Na_w\) is called the \emph{free boundary} of \(w\). Notice that the definitions of these sets are essentially different from those for solutions to \eqref{e.grad-degen-neumann} in Remark \ref{generalizedCauNauGammau}. An example that shows this difference is $w(x_1,x_2) = \hbox{Re} (x_2+i|x_1|)^{1/2}$. This example is a solution to \eqref{localproblemtoe} but not \eqref{e.grad-degen-neumann}. Notice that for any \(N>0\) there exists a smooth function \(\phi_N\) touching \(w\) from below at \(0\), while \(\pt_1\phi_N(0)>N>0\), which means that \(0\in \Ca_w\) if in the sense of Definition \ref{definitionofCu}, but it is in fact contained in \(\Gamma_w\) by the definitions of \(\Na_w\) and \(\Gamma_w\) as described above. However, the definitions will \emph{coincide} if a solution solves both \eqref{e.grad-degen-neumann} and \eqref{localproblemtoe}. 
\end{remark}

\begin{lemma}
    Let \(w\) be a continuous viscosity solution to the equation \eqref{localproblemtoe}, then there is a constant \(C=C(d)>0\) such that
    \be
    \norm{w}_{L^\infty(B_{1/2}^+)} \le C \norm{g}_{L^2(\pt B_1\cap \{x_1\ge0\})}.\label{L2toLinfty}
    \ee
\end{lemma}
\begin{remark}
 Using the same proof we know that for some constant \(C>0\) and all \(r>0\) 
    \[
    \norm{w}_{L^\infty(B_{r/2}^+)} \le C r^{-d/2}\norm{w}_{L^2(B_r^+)},
    \]
    where \(C\) is independent of \(r\).

\label{remarkonrescalinofL2toLinfty}
\end{remark}

\begin{proof}
Let \(g_+=\max\{g,0\}\), \(g_-=\min\{g,0\}\), and denote \(v_+,\,v_-\) respectively the Neumann solution with boundary data \(g_+,\,g_-\). By classical theories \(v_{\pm}\) are smooth in \(B_1^+\sqcup B_1'\) up to the flat boundary. We claim that any continuous viscosity solution \(w\) to \eqref{localproblemtoe} has to satisfy
\[
v_-\le w \le v_+,\,\hbox{ in }\overline{B_1^+}.
\]
Once we prove the claim the estimate \eqref{L2toLinfty} will follow from the classical theories for Neumann solutions. The upper bound \(w\le v_+\) can be immediately obtained by observing that \(w\) is also a Neumann subsolution and \(g\le g_+\) on \(\pt B_1\cap\{x_1\ge0\}\). To obtain the lower bound we consider the following maximization problem for small \(\beta>0\)
\[
\max_{x\in \overline{B_1^+}} v_-(x)-w(x)+\beta (x_1)_+-\beta.
\]
By comparison principle of harmonic functions and \(g_-\le g\) on \(\pt B_1\cap\{x_1\ge0\}\), the maximum point \(x_{\ast}\) must occur in \(B_1'\) if the maximum value is positive. Let us first consider the case that \(g_-\ne0\). In this case \(v_-<0\) in \(B_1^+\sqcup B_1'\) by strong maximum principle, and hence 
\[
w(x_\ast)< v_-(x_\ast)-\beta<0.
\]
Moreover \(v_-+\delta(x_1)_++C\) touches \(w\) from below at \(x_\ast\) for some constant \(C\), which contradicts the supersolution condition of \(w\) at \(x_\ast\).  

In the case \(g_-=0\) we would like to show that \(w\ge0\). Let us similarly consider the following maximization problem
\[
\max_{x\in \overline{B_1^+}} \beta (x_1)_+-\beta-w(x).
\]
Also by maximum principle the maximum point \(x_\ast\) can only occur on \(B_1'\) if the maximum value is positive. This shows that 
\[
w(x_\ast)<-\beta<0,
\]
and then \(\beta(x_1)_++C\) touches \(w\) from below at \(x_\ast\), which also contradicts the supersolution condition of \(w\).
\end{proof}

\subsection{Blow-up profiles}

In this section, we discuss the possible blow-up profiles of the function sequence \(w_t\) as discussed after Theorem \ref{optimalestimate}.

According to \eqref{boundaryL2ofw2is1}, we know that the blow-up sequence \(w_t\) have bounded \(L^2\)-norm on the boundary portion \(\pt B_1\cap \{x_1\ge0\}\). By applying Lemma \ref{L2toLinfty}, we obtain boundedness of \(w_t\) in \(L^\infty(B_{1/2})\) (when extended to the whole ball by even reflection). Now using the \(C^{1,\alpha}\) estimate, Theorem \ref{C1aplhaestimate}, the sequence of functions \(w_t\) is bounded in \(C^{1,\alpha}\lb\overline{B_{1/4}}\rb\), which shows the following lemma.

\begin{lemma}
    The sequence of functions \(w_t\) is compact in both \(H^1(B_{1/4})\) and \(C^1\lb \overline{B_{1/4}}\rb\).
\end{lemma}

By applying this lemma, we can find \(t_j\rta0^+\) such that \(w_{t_j}\rta w_0\) in both \(H^1(B_{1/4})\) and \(C^1\lb \overline{B_{1/4}}\rb\) as \(j\rta\infty\). On the other hand, we have for \(0<r<1/4\)
\[
N(r,w_{0})=\lim_{j\rta\infty} N(r,w_{t_j})=\lim_{j\rta\infty} N(rt_j, u) = N(0^+,u)=:\kappa.
\]
Applying the Almgren's monotonicity formula, Theorem \ref{t.almgrenmonotone}, we obtain the following characterization of all the blow-up limits.

\begin{proposition}
    Let \(u\) be a viscosity solution to \eqref{e.grad-degen-neumann} that satisfies the condition \emph{\eqref{e.adelta}}, then the blow-up limit \(w_0\) as defined above is a nonzero global solution to \eqref{localproblemtoe}, and is homogeneous of degree \(\kappa=N(0^+,u)>1\).
\end{proposition}

Now, we would like to classify all the \(\kappa\)-homogeneous solutions \(w_\kappa\) to \eqref{localproblemtoe} in dimension \(d=2\). It can be checked (see Appendix \ref{Appendix.classification}) that after reflection and normalization, any nonzero homogeneous viscosity solution \(w_\kappa\) of degree \(\kappa\ge0\) has to take one of the forms for \(k\in \Z_+\) in Table \ref{tab:tableofclassificaiton}.

\begin{table}[htbp]
    \centering
    \renewcommand{\arraystretch}{1.5}
\begin{tabular}{|c|c|c|c|c|}
\hline
\(\kappa\) & \(w_\kappa(x_1,x_2)\) & \(\Ca_{w_\kappa}\cap B_1\) & \(\Gamma_{w_\kappa}\cap B_1\) & \( \Na_{w_\kappa}\cap B_1\)\\
\hline
$1$   & $|x_1|$  & \( B_1'\) &  \(\emptyset\) & \(\emptyset\)\\
\hline
$2k+1$   & $\text{Im}\left((x_2+i|x_1|)^\kappa\right)$  & \(B_1'\setminus\{0\}\)& \(\{0\}\) & \(\emptyset\)\\
\hline
 $\frac{2k-1}{2}$& $ \text{Im}\left((x_2+i|x_1|)^\kappa\right)$ & \(\{0\}\times(0,1)\)& \(\{0\}\)& \(\{0\}\times(-1,0)\) \\
\hline
 $k$& $\pm\text{Re}\left((x_2+i|x_1|)^\kappa\right)$ & \( \emptyset\)& \(\emptyset\)& \(B_1'\)\\
\hline
\end{tabular}

\vspace{0.2cm}

    \caption{Classification of \texorpdfstring{\(\kappa\)}{d}-homogeneous solutions to \eqref{localproblemtoe} in dimension \(d=2\). See the proof in Appendix \ref{Appendix.classification}.}
\label{tab:tableofclassificaiton}
\end{table}
\begin{remark}
    In the case \(\kappa=\frac{4k-1}{2}\) we have
    \[
    \text{Im}\left((x_2+i|x_1|)^\kappa\right)=-\text{Re}\left((-x_2+i|x_1|)^\kappa\right)
    \]
    correspond to the \emph{nontrivial} homogeneous solutions to the Signorini problem in Example \ref{quasisolexamp1}. In the case \(\kappa=\frac{4k-3}{2}\) we have
    \[
\text{Im}\left((x_2+i|x_1|)^\kappa\right)=\text{Re}\left((-x_2+i|x_1|)^\kappa\right)
    \]
    correspond to the \emph{nontrivial} homogeneous solutions to the sign-reversed Signorini problem in Example \ref{quasisolexamp2}.
\end{remark}
In particular, if \(\kappa>1\) then we know that \(\kappa\ge 3/2\). In higher dimensions, we can also obtain this property by using the ACF monotonicity formula, see \cites{petrosyan2011monotonicity,garofalo2009monotonicity}.

\begin{theorem}\label{theoremfor3over2inhigherdimension}
    Let \(w\) be a homogeneous viscosity solution to \eqref{localproblemtoe} of degree \(2>\kappa>1\). Then \(\kappa=3/2\), and 
    \[
    w(x)=\emph{Im}(x_2+i|x_1|)^{3/2}=- \emph{Re}(-x_2+i|x_1|)^{3/2},
    \]
    after a possible rotation in \(\R^{d-1}\) and normalization.
\end{theorem}

\begin{proof}[Sketch of Proof]
    We outline the idea of the proof here. See \cites{petrosyan2011monotonicity,garofalo2009monotonicity} for detailed proofs. Given \(w\), extended evenly to the whole ball \(B_1\), we would like to consider the following two functions with \(e\in \{0\}\times\R^{d-1}\)
    \[
    v_+=\max\{\pt_e w,0\},\,v_-=\max\{\pt_{-e} w, 0\}=\max\{-\pt_{e} w, 0\}.
    \]
    By \(C^{1,\alpha}\) regularity, Theorem \ref{C1aplhaestimate}, we can discuss everything in classical setting. Therefore, by the boundary condition \eqref{e.grad-degen-neumann}, we know that \(v_\pm\) are harmonic wherever they are positive, which shows that both of them are subharmonic. On the other hand, we have \(v_-\cdot v_+=0\). By the ACF monotonicity formula,
    \[
    \phi_e(r)=\frac{1}{r^{4}}\int_{B_r} \frac{|\gd v_+|^2}{|x|^{d-2}}\int_{B_r} \frac{|\gd v_-|^2}{|x|^{d-2}}= r^{4(\kappa-2)}\phi_e(1)
    \]
    is monotone in \(r>0\). When \(1<\kappa<2\) then the monotonicity would imply \(\phi_e(1)=0\), which means that one of \(v_\pm\) is identically zero and hence \(\pt_e w\) is either nonnegative or nonpositive on the entire \(\R^{d}\). We denote 
    \[
    S_+=\{e\in \mathbb{S}^{d-2}:=\pt'B_1' : \pt_e u \ge 0\},
    \]
    and
    \[
    S_-=\{e\in  \mathbb{S}^{d-2}:=\pt'B_1' : \pt_e u \le 0\}.
    \]
    Notice that \(S_+=-S_-\ne \emptyset\), \(\mathbb{S}^{d-2}= S_+\cup S_-\) and both are closed subsets. Since whenever \(d>2\), \(\mathbb{S}^{d-2}\) is connected and hence \(S_+\cap S_-\ne \emptyset\). One can choose \(e^{(1)}\in S_+\cap S_-\) to reduce to the orthogonal subspace of \(\{0\}\times \R^{d-1}\) with respect to \(e^{(1)}\). Proceed with the same procedure we can obtain an orthogonal sequence \(e^{(1)},\cdots, e^{(d-2)}\) such that we can reduce to the subspace spanned by \(e^{\ast}\in \{0\}\times\R^{d-1}\) with \(e^\ast\cdot e^{(j)}=0\) for all \(j=1,\cdots, d-2\), which is equivalently solving the problem in the case that \(d=2\). In this case we know that the only homogeneous solution with \(1<\kappa<2\) is of the form
    \[
    w(x)=\textup{Im}(x_2+i|x_1|)^{3/2}=- \textup{Re}(-x_2+i|x_1|)^{3/2}.
    \]
\end{proof}

\subsection{Optimal regularity in dimension \texorpdfstring{\(d\ge 3\)}{a} under condition \texorpdfstring{\eqref{e.adelta}}{d}}

To prove Theorem \ref{optimalestimate}, we follow the framework in \cite{garofalo2014} and start with the estimation near \(0\in\Gamma_u\).
    
\begin{lemma}
    Let \(u\) be a viscosity solution to \eqref{e.grad-degen-neumann} that satisfies condition \emph{\eqref{e.adelta}} with \(\osc_{B_1^+} u \le 1\), \(u(0)=0\) and \(0\in \Gamma_u\). Then \(|\gd u(0)|=0\) and there is \(r_0=C(d)\delta>0\) such that
    \[
    |v(x)|\le \overline{C} |x|^{3/2},\,|x|\le r_0,
    \]
    where \(\overline{C}>0\) is universal.
\end{lemma}

\begin{proof}
    For the proof recall the definition of the frequency function $N(u,r) = \frac{rD(r)}{H(r)}$ in \eref{frequency-defn} with $D(r) = \int_{B_r} |\grad u|^2$ and $H(r) = \int_{\partial B_r} u^2$. By Theorem \ref{theoremfor3over2inhigherdimension} we know that for \(0<r<r_0\)
    \[
    \frac{3}{2}\le N(0^+,u) \le N(r,u).
    \]
This implies that 
\[
r\frac{d}{dr}\log H(r) \ge d+2 \ \hbox{ for } \ 0<r<r_0, 
\]
and then 
\[
H(r)\le C r^{d+2}.
\]
After integrating with respect to \(r\) we know that
\[
\int_{B_r} u^2 \le C r^{d+3}.
\]
The proof is now completed by applying Remark \ref{remarkonrescalinofL2toLinfty}.    
\end{proof}

Now we can prove the optimal regularity estimate for all the viscosity solutions to \eqref{e.grad-degen-neumann}.

\begin{proof}[Proof of Theorem \ref{optimalestimate}]
    We follow the reflection arguments in Theorem 6.7 in \cite{garofalo2014}. Similar to the proof of Theorem \ref{C1aplhaestimate}, it suffices to consider the case that \(u(\Ca_u)\) has at most one element. For each \(x_0\in B_{1/2}^+\sqcup B_{1/2}'\) we define \(d(x_0)=\dist(x_0,\Gamma_u)\) with \(\Gamma_u\) the free boundary of \(u\). We claim that either
     \be
    B(x_0,d(x_0))\cap B_1'\subset \Ca_u\hbox{, or }B(x_0,d(x_0))\cap B_1'\subset \Na_u.\label{dichotomyofballwithdistancetogammau}
    \ee
    Indeed, by the partition \(B_1'=\Ca_u\sqcup\Na_u\sqcup \Gamma_u\) we know that \(A:=B(x_0,d(x_0))\cap B_1'\)  can be partitioned into
\[
A=(A\cap \Ca_u) \cup (A\cap \Na_u) \cup (A\cap\Gamma_u).
\]
We assume without loss that \(A\ne \emptyset\). By definition of \(d(x_0)\), we know that \(A\cap \Gamma_u=\emptyset\) and so if neither of \ref{dichotomyofballwithdistancetogammau} is satisfied then \(A\) would be a disconnected set, which contradicts of the fact that it is an open subball of \(B_1'\). If \(A_{x_0}\subset \Ca_u\) then we extend \(u\) to the whole \(B(x_0,d(x_0))\) by odd extension, and if \(A_{x_0}\subset \Na_u\) then we extend \(u\) to the whole \(B(x_0,d(x_0))\) by even extension. By Schwarz reflection, these are smooth extensions.

To show \(C^{1,1/2}\) it suffices to show that for \(|\xi-\eta|\le \frac{1}{32},\, \xi,\eta\in B_{1/2}^+\sqcup B_{1/2}'\),
\[
|\gd u(\xi) - \gd u(\eta)| \le C |\xi-\eta|^{1/2},
\]
for some constant \(C>0\). In the following \(C>0\) means a universal constant that may change from line to line. If \(d(\xi) \ge \frac{1}{16}\) (or symmetrically \(d(\eta)\ge \frac{1}{16}\)), then we can use the smoothness of \(u\) in \(B(\xi,\frac{1}{32})\) in either the case of odd or even extension as described above. If \(d(\eta)\le d(\xi) \le \frac{1}{16}\) and \(|\xi-\eta|\ge d(\xi)/2\), then by the gradient estimate we have
\[
\begin{split}
   |\gd u(\xi)|&\le \frac{C}{d(\xi)} \sup_{B(\xi,d(\xi))} |u|\\
 &\le\frac{C}{d(\xi)} \sup_{B(\xi_0,2d(\xi))} |u|\\
 &\le C d(\xi)^{1/2}\\
 &\le C |\xi-\eta|^{1/2},
\end{split}
\]
where \(\xi_0\in \Gamma_u\) is a point that \(|\xi-\xi_0|=d(\xi)\). Similarly \(|\gd u(\eta)| \le C |\xi-\eta|^{1/2}\). In the last case that \(d(\eta)\le d(\xi) \le \frac{1}{16}\) and \(|\xi-\eta|< d(\xi)/2\), we use the interior estimate for second order derivatives of harmonic functions and obtain
\[
\begin{split}
   |\gd u(\xi)-\gd u(\eta)|&\le \frac{C|\xi-\eta|}{d(\xi)^2} \sup_{B(\xi,d(\xi))} |u|\\
 &\le \frac{C|\xi-\eta|}{d(\xi)^2} \sup_{B(\xi_0,2d(\xi))} |u|\\
 &\le C |\xi-\eta|d(\xi)^{-1/2}\\
 &\le C |\xi-\eta|^{1/2}.
\end{split}
\]

\end{proof}

\section{Minimal supersolution and comparison principle\label{section.comparisonprinciple}}
In this section, we study the minimal supersolutions to \eqref{e.grad-degen-neumann} by using Perron's method and prove the characterizing comparison principle. Given a fixed continuous boundary data \(g\) on the boundary portion \(\pt B_1\cap \{x_1\ge0\}\), a minimal supersolution with respect to the boundary data \(g\) is defined as
\[
v_{g}(x):= \inf\{v(x)\,;\,v\hbox{ is a supersolution to \eqref{e.grad-degen-neumann} and } v\ge g\hbox{ on }\pt B_1\cap \{x_1\ge0\}\}.
\]
Unlike the general viscosity solutions to \eqref{e.grad-degen-neumann}, a minimal supersolution would satisfy an additional \emph{strong subsolution condition}.

\begin{definition}[Strong subsolution]
  \label{definitionofstrongsubsolution}  An upper semicontinuous function \(u\) is called a strong subsolution to \eqref{e.grad-degen-neumann} if it is a subsolution and there are no \(C^1\) up-to-boundary function of the form \(\varphi(x_1,x')\equiv\psi(x_1)\) that touches \(u\) from above in \(\Omega_h\cap\overline{B_1^+}\) at some \(x_0\in B_1'\) and \(\varphi>u\) in \(\overline{\Omega_h\setminus\Omega} \cap \overline{B_1^+}\) where \(\Omega\) is an arbitrary open domain of \(\R^d\) containing \(x_0\) and \(\Omega_h=\bigcup_{y\in\Omega} B_h(y)\) for some small \(h>0\) so that \(\overline{\Omega_h}\cap \overline{B_1^+} \csubset B_1^+\cup B_1'\).
\end{definition}
We will show in the following subsections that this strong subsolution condition is equivalent to the boundary maximum principle, and is indeed a necessary condition for a minimal supersolution.

\begin{remark}
    In Section \ref{section.flatasymsingularbernoulli}, we will discuss how the flat asymptotic expansion of the minimal supersolutions to \eqref{discontinuous} gives rise to the strong subsolution property of the asymptotic limit. This, by the comparison principle, Theorem \ref{comparisonprinciple} (which we will prove in this section) will then lead to the equivalence of the three notions of solutions: the minimal supersolutions, the solutions that satisfy the strong subsolution condition (or the asymptotic expansion limit arising in \eqref{discontinuous}) and the solutions that satisfy the boundary maximum principle.
\end{remark}

 To prove the comparison principle (see Theorem \ref{comparisonprinciple}) for supersolutions and strong subsolutions we face several difficulties due to the degeneracy of the problem \eqref{e.grad-degen-neumann}. The degenerate Neumann boundary condition in \eqref{e.grad-degen-neumann} is incompatible with the classical doubling variable arguments. We can not use the classical sub-/sup-convolution either because otherwise the sub-/super-solutions would not be preserved under the mollification procedure. We overcome this issue by introducing the \say{tangential} sub-/sup-convolution technique along with a harmonic lift (see Section \ref{s.tnagentialsubsup} and \ref{s.harmonic-lift}). 

 \begin{remark}\label{equivalenceofthreenotionsofsolutions}
     Combining the discussions in Section \ref{section.comparisonprinciple} and \ref{section.flatasymsingularbernoulli}, we know that the following three functions are equal to each other if they share the same boundary data on the boundary portion \(\pt B_1\cap \{x_1\ge0\}\)
     \begin{enumerate}
         \item the asymptotic expansion of the Singular Bernoulli problem as discussed in Section \ref{section.flatasymsingularbernoulli}, which satisfies an additional strong subsolution condition as defined in Definition \ref{definitionofstrongsubsolution};
         \item the viscosity solution to \eqref{e.grad-degen-neumann} that satisfies an additional boundary maximum principle as described in Lemma \ref{boundarymaximumprinciple};
         \item the minimal supersolution to \eqref{e.grad-degen-neumann} as discussed in Section \ref{section.comparisonprinciple}.
     \end{enumerate}
 \end{remark}

\subsection{Boundary maximum principle}

Let us now show that the strong subsolution condition is equivalent to the boundary maximum principle. To that end, let us recall the concept of sub/sup-convolutions.

\begin{definition}
    Let \(U\subset \R^d\) be a domain and \(u,v: U\rta \R\). The sup-convolution of \(u\in \mathrm{USC}\lb \overline{U}\rb\) is defined for \(\varepsilon>0\)
    \[
    u^\ep(x)=\sup_{y\in \overline{U}} \lma u(y)-\frac{1}{2\ep}|x-y|^2\rma,\,x\in\R^n.
    \]
    The inf-convolution of \(v\in\mathrm{LSC}\lb \overline{U}\rb\) is defined as
    \[
    v_\ep(x)=\inf_{y\in \overline{U}} \lma v(y)+\frac{1}{2\ep}|x-y|^2\rma,\,x\in\R^n.
    \]

\end{definition}

For the convenience of discussing the limit of inf/sup-convolutions, let us also introduce the half-relaxed limits of Barles \cite{Barles1987}.

\begin{definition}
    Let \(u^k\) be a family of functions that is bounded from above, then the upper half relaxed limit of \(u^k\) is defined as
    \[
    {\limsup_{k\rta\infty}}^\ast u^k(z) := \lim_{k\rta\infty} \sup_{n>k,\,|z-x|\le1/k} u^n(x).
    \]
    If \(v_k\) is a family of functions that is bounded from below, then the lower half relaxed limit is defined as
    \[
    \underset{k\rta\infty}{{\liminf}_\ast} v_k(z) := \lim_{k\rta\infty} \inf_{n>k,\,|z-x|\le1/k} v_n(x)
    \]
\end{definition}

\begin{remark}
Let us recall some of the basic properties of sub/sup-convolutions (see \cites{UsersGuide,Silvestre2015} and the references therein).
\begin{enumerate}
    \item  It is known that \(v_\ep\) and \(u^\ep\) are, respectively, semiconcave and semiconvex, specifically $v_\ep(x) - \frac{1}{2\ep}|x|^2$ is concave and $u^\ep(x) + \frac{1}{2\ep}|x|^2$ is convex.  In particular, both $u^\ep$ and $v_\ep$ are also Lipschitz continuous. \label{part.sompropertyofhalfrelaxed.1}
    \item If \(u\) is subharmonic in a domain \(\Omega\) then \(u^\ep\) is also subharmonic in a slightly smaller domain \(\Omega_\ep\csubset\Omega\). 
    \item The lower and upper half relaxed limits defined above are always, respectively, lower and upper semi-continuous.
    \item  When \(u\) is upper semi-continuous, then we have the following half-relaxed convergence
    \[
   u = {\limsup_{\ep\rta0^+}}^\ast u^\ep.
    \]
A similar convergence holds for lower semi-continuous \(v\):
\[
v = \underset{\ep\rta0^+}{{\liminf}_\ast} v_\ep .
\]
\item If \(u={\limsup}^\ast u^k\) on a compact set \(K\subset \R^d\), then we have 
\[
\limsup_{k\rta\infty} \max_K\lb u^k\rb \le \max_K(u).
\] 
A similar statement holds true for \(v={\liminf}_\ast v_k\), then 
\[
\liminf_{k\rta\infty} \min_K\lb v_k\rb \ge \min_K(u).
\]
\item Let \(u\) be an upper semi-continuous function and call its sup-convolutions \(u^\ep\), then because \(u^\ep\ge u\), on a compact set \(K\subset \R^d\) we have
\[
\lim_{\ep\rta0} \max_K\lb u^\ep\rb = \max_K(u).
\]
A similar result also holds for lower ones.
\end{enumerate}
\label{sompropertyofhalfrelaxed}
\end{remark}

Next, we show that {strong subsolutions} satisfy a \emph{boundary maximum principle}.

\begin{lemma}[Boundary Maximum Principle]
    Let \(u\) be a strong subsolution as defined in Definition \ref{definitionofstrongsubsolution}, then we have for any subdomain \(\Omega\csubset B_1'\)
    \be
    \max_{x\in \overline{\Omega}} u(x)=\max_{x\in \partial' \Omega} u(x),\label{boundarymaximumprinciple}
    \ee
    where \(\partial' \Omega\) is defined as the relative boundary of \(B_1'\) in \(\{x_1=0\}\).
\end{lemma}

\begin{remark}\label{equivalencebetweenboundamxandstrongsubsolution}
    We notice that if a subsolution \(u\) satisfies \eqref{boundarymaximumprinciple}, then it will also satisfy the strong subsolution condition.
\end{remark}

\begin{proof}
We extend \(u\) to \(\overline{B_1}\) evenly and consider it as an upper semi-continuous function defined on the whole \(\R^d\) by setting it to be \(-\infty\) outside \(\overline{B_1}\). It can be observed that \(u\) is subharmonic in \(B_1\). Let \(u^\ep\) be a family of sup-convolutions of \(u\), then according to Remark \ref{sompropertyofhalfrelaxed} we know that \(u^\ep\) are Lipschitz, subharmonic in \(B_{1-\gamma(\ep)}\), and \(u^\ep\) converges in \(\overline{B_1^+}\) to \(u\) as \(\ep\rta0\) in the sense
\[
u={\limsup_{\ep\rta0^+}}^\ast u^\ep.
\] 
Suppose there is a subdomain \(\Omega\csubset B_1'\) and positive numbers \(\delta,\rho>0\) such that 
\[
\max_{\overline{\Omega}} u \ge \max_{\overline{\Omega_\rho\setminus\Omega}} u +\delta,
\]
with \(\Omega_\rho:= \bigcup_{x\in\Omega} B_\rho'(x)\csubset B_1'\). According to Remark \ref{sompropertyofhalfrelaxed} (5) and (6), for sufficiently small \(\ep>0\) we have
\be
u^\ep(y) \le \max_{\overline{\Omega}} u -2\delta/3 \ \hbox{ for } \ y\in\overline{\Omega_\rho\setminus\Omega}.\label{boundarystricsmallofboundaryaxprixn}
\ee
On the other hand, we pick for convenience a subsequence \(\ep_j\rta0^+\) as \(j\rta\infty\), and using Remark \ref{sompropertyofhalfrelaxed} (6) again on the set \(K=\overline{\Omega}\), we obtain a sequence \(u_j:= u^{\ep_j}\) such that 
\be
\lim_{j\rta\infty} \max_{\overline{\Omega}} u_j = \max_{\overline{\Omega}} u,\label{maxlimitforsupconvol}
\ee
and the following property is satisfied by combining \eqref{boundarystricsmallofboundaryaxprixn} and \eqref{maxlimitforsupconvol}
\be
\max_{\overline{\Omega_{\rho}\setminus\Omega}} u_j \le \max_{\overline{\Omega}} u_j -\delta/2 ,\label{boundarystriimaximumprinciforuj}
\ee
for sufficiently large \(j\).

To make a contradiction, we assume \(L_j>0\) to be the Lipschitz constants of \(u_j\) (we may, without loss, assume that \(L_j\rta\infty\) as \(j\rta\infty\)), then we construct 
\[
w_j (x)=\max_{\overline{\Omega}} u_j + 100L_j (x_1)_+ ,
\]
that touches \(u_j\) from above in \(\Omega_\rho\times[0,r]\) at some \(x_j\in\Omega\) with \(r>0\) small so that \(\Omega_\rho\times[0,r]\csubset B_1^+\sqcup B_1'\). We claim that \(w_j-u_j\ge \delta/2>0\) on the fattened boundary \(\overline{\Omega_\rho\setminus\Omega}\times[0,r] \bigcup \overline{\Omega_\rho}\times[r(1-\rho),r]\). Indeed, on the set \(\overline{\Omega_\rho}\times[\delta/L_j,r]\), we always have
\[
w_j-u_j\ge 98\delta>\delta/2.
\]
For \(z\in \overline{\Omega_\rho\setminus\Omega}\times[0,\delta/L_j]\), we have by \eqref{boundarystriimaximumprinciforuj}, for sufficiently large \(j\) 
\[
\begin{split}
    w_j(z) &= \max_{\overline{\Omega}} u_j + 100L_j (x_1)_+ \\
    &\ge \max_{\overline{\Omega_{\rho}\setminus\Omega}} u_j  + \delta/2 + 100L_j (x_1)_+ \\
    &\ge u_j(z) + \delta/2.
\end{split}
\]
This completes the claim.

Now there is a sequence \(c_j\to 0\) (which is because \(w_j(0)=\max_{\overline{\Omega}} u_j\rta \max_{\overline{\Omega}} u\)) such that \(w_j+c_j\) touches \(u\) from above in \(\Omega_\rho\times[0,r]\). We claim that the touching must be at some point in \(\Omega\). This is because \(w_j\) are harmonic the touching must occur on the boundary of $\Omega_\rho \times (0,r)$, in which case we can reduce to the boundary portion \(\Omega\) due to the strict ordering on the fattened boundary \(\overline{\Omega_\rho\setminus\Omega}\times[0,r] \bigcup \overline{\Omega_\rho}\times[r(1-\rho),r]\).

Indeed we can choose $j$ sufficiently large so that \(|c_j|\ll \delta/4\), and then we have by the previous claim the strict inequality 
\[
w_j+c_j > u_j+c_j +\delta/2> u + \delta/4 \ \hbox{ in } \ \overline{\Omega_\rho\setminus\Omega}\times[0,r] \bigcup \overline{\Omega_\rho}\times[r(1-\rho),r],
\]
which contradicts the {strong subsolution property}.

\end{proof}

\subsection{Tangential sub/sup-convolution}\label{s.tnagentialsubsup}

For the comparison principle proof, we will use a procedure based on inf/sup-convolutions in the tangential variables and harmonic replacement.  This is natural for the nonlinear Neumann problem, which could also be viewed as a nonlinear fractional order PDE problem on the lower dimensional $B_1'$. 

Let us now define the tangential inf/sup-convolutions.

 \begin{definition}
     Suppose \(u,-v\in \mathrm{USC}(\overline{B_1^+})\), then the \emph{tangential sup-convolution} of \(u\) is defined as
     \[
     \Tc^\ep u(x):= \sup_{x+h\in \overline{B_1^+}(0); \,h\in \R^{d-1}} \left\{u(x+h) - \frac{1}{2\ep}|h|^2\right\}, 
     \]
     where \(x=(x_1,x')\). The tangential inf-convolution of \(v\) is defined as \(\Tc_\ep v := - \Tc^\ep(-v)\).
 \end{definition}

 \begin{remark}
     Let \(u\in \mathrm{USC}(\overline{B_1^+})\), then we may naturally extend \(u(x_1,x')=-\infty\) for \(x\not\in\overline{B_1^+}(0)\) and then \(u\in \mathrm{USC}([0,1]\times\R^{d-1})\). The advantage of this extension is that we may extend the tangential sup-convolution formula to
     \[
     \Tc^\ep u(x) = \sup_{h\in \R^{d-1}} \left\{u(x+h) - \frac{1}{2\ep}|h|^2\right\} \ \hbox{ for } \ x\in[0,1]\times\R^{d-1}.
     \]
     Similar extension can be done to \(v\).
\end{remark}

Now we will briefly establish several properties of the tangential inf/sup-convolutions that follow from or have very similar proofs to the properties of standard inf/sup convolutions which were collected in Remark \ref{sompropertyofhalfrelaxed}.

 \begin{lemma}
     For any \(\ep>0\) small and fixed \(x_1\in[0,1]\), the tangential sup-convolution \(\Tc^\ep u(x_1,x')\) is semi-convex in \(x'\in\R^{d-1}\). For every \(\ep\), \(\Tc^\ep u\) is upper semi-continuous in \([0,1]\times\R^{d-1}\).\label{somepropertoftangentialconvol}
 \end{lemma}

 \begin{proof}
     The first statement is an immediate consequence of the same result for the classical sup-convolution recalled in Remark \ref{sompropertyofhalfrelaxed}. To show the second statement, we consider a sequence of points \(z_n\rta z\in \overline{B_1^+}\). For each \(z_n\) there corresponds an \(h_n\in B_2'(0)\) such that
     \[
     \Tc^\ep u (z_n) = u(z_n+h_n)-\frac{1}{2\ep}|h_n|^2.
     \]
By compactness of \(h_n\) we may obtain after passage to a subsequence
\[
h_n\rta h_\infty,
\]
which implies that
\[
\begin{split}
    \limsup_{n\rta\infty}\Tc^\ep u (z_n) &= \limsup_{n\rta\infty}u(z_n+h_n)-\frac{1}{2\ep}|h_n|^2 \\
    &\le u (z+h_\infty) - \frac{1}{2\ep}|h_\infty|^2\\
    &\le \Tc^\ep u (z).
\end{split}
\]
\end{proof}

\begin{lemma}
    Let \(u\) be an upper semi-continuous function, \(u^\ep\) the classical sup-convolution, then 
    \be
    u^\ep\ge \Tc^\ep u \ge u,\label{comparingsuptangentialsup}
    \ee
    and in particular,
    \[
   {\limsup_{\ep\rta0^+}}^\ast \Tc^\ep u = u.
    \]
\end{lemma}
\begin{proof}
    The inequality \eqref{comparingsuptangentialsup} can be obtained by the definition directly. The half-relaxed limit can be obtained by applying the inequality \eqref{comparingsuptangentialsup} and Remark \ref{sompropertyofhalfrelaxed} (4).
\end{proof}

 \begin{lemma}
     Let \(u\in \mathrm{USC}(\overline{B_1^+})\) and \(x_0\in \overline{B_1^+}\). If a smooth function \(\phi\) touches \(\Tc^\ep u\) (strictly) from above at \(x_0\), and 
    \[
     \Tc^\ep u(x_0) = u(x_\ep) -\frac{1}{2\ep}|x_0-x_\ep|^2 \ \hbox{ for some } \ x_\ep \in x_0 +  \R^{d-1}.
     \]
     Then \(\psi(x)=\phi(x+x_0-x_\ep)+ \frac{1}{2\ep}|x_\ep - x_0|^2\) will touch \(u\) (strictly) from above at \(x_\ep\in \overline{B_1^+}\) that is close to \(x_0\), and \(\gd' \psi(x_\ep)=\gd' \phi(x_0)=\frac{1}{\ep}(x_\ep-x_0)\in \R^{d-1}\). 
  
 \end{lemma}

The proof is omitted since it is similar to the standard sup-/inf-convolution, and the details of the proof can also be in \cite{calder2018lecture}*{Proposition 8.6}.

Applying this lemma it is standard to check that the inf-convolution and sup-convolutions preserve viscosity super and subsolution properties respectively.

\begin{corollary}
    If \(u\) is a supersolution (subsolution) to \ref{supsolutiondef}, then \(\Tc_\ep u\) (\(\Tc^\ep u\)) is still a supersolution (subsolution) to \ref{supsolutiondef} (or \ref{subsolutiondef}) in \(\overline{B_{1-\gamma}^+}\) for some small \(\gamma=\gamma(\ep)>0\). Moreover, if \(u\) is a {strong subsolution}, then so is \(\Tc^\ep u\) in \(\overline{B_{1-\gamma}^+}\). The constant \(\gamma\rta 0\) as \(\ep\rta0^+\).
\end{corollary}
 
\subsection{Harmonic lift}\label{s.harmonic-lift}
In this section, we study the harmonic lift of a given bounded subharmonic function \(v\) on \(\overline{B_1^+}\). The results are standard but we want to carefully enumerate the properties of the harmonic lift since $v$ will only be upper-semicontinuous.

Using Perron's method, we can define
\begin{equation}\label{e.harmonic-replace-defn}
    w(z):=\inf\{u(z)\,; \hbox{ $u\in C\lb\overline{B_1^+}\rb$, superharmonic and } u \geq v \ \hbox{in $B_1^+$}\} \ \hbox{ for } \ z\in\overline{B_1^+}.
\end{equation}
 Note that $\max_{\overline{B_1^+}} v$ is one such superharmonic function so the infimum is well-defined. By the standard arguments in Perron's method for the Laplacian, we have \(w=w^\ast=w_\ast\) is continuous and harmonic in the interior \(B_1^+\).

\begin{lemma}
    Let \(w\) be as defined in \eref{harmonic-replace-defn}, then \(w^\ast=v\) on \(\pt B_1^+\).\label{equalwehnsupenvelope}
\end{lemma}

\begin{proof}
 Indeed, since \(v=v^\ast\), we have by the equivalent form of upper envelope
\[
v(z)=v^\ast(z)=\inf\{h(z)\,;\, h\in C\lb \overline{B_1^+}\rb,\,h\ge v \hbox{ on }\overline{B_1^+}\},\,z\in \overline{B_1^+}
\]
there would be a sequence of continuous functions \(h_{n,z}\in C\lb \overline{B_1^+}\rb\) for \(z\in \pt B_1^+\), such that \(h_{n,z}(z)\rta v^\ast(z)=v(z)\) as \(n\rta\infty\). Now we construct \(w_{n,z}\) such that
\[
\bca
\Delta w_{n,z} =0,&\hbox{ in }B_1^+\\
w_{n,z} = h_{n,z},&\hbox{ on }\pt B_1^+,
\eca
\]
and the $w_{n,z}$ are continuous up to $\partial B_1^+$ because the boundary data is continuous and the domain is outer regular.  By definition of \(w\), we always have
\[
w\le w_{n,z}\hbox{ in }\overline{B_1^+}.
\]
This shows that we have
\[
w^\ast(z) \le w_{n,z}(z),\,\forall n,
\]
which shows the inequality \(w^\ast\le v\). The other side can be obtained directly from the definition of \(w\).
\end{proof}

\begin{definition}
    Let \(v\in \hbox{USC}\lb \overline{B_1^+}\rb\) be a bounded subharmonic function, then we defined its harmonic lift to be, with $w$ from \eref{harmonic-replace-defn}, \(w^\ast\in \hbox{USC}\lb \overline{B_1^+}\rb\cap C^\infty\lb B_1^+\rb\), which is harmonic in \(B_1^+\) and \(w^\ast=v\) on \(\pt B_1^+\). The lower semicontinuous harmonic lift of a bounded superharmonic function in $B_1^+$ is defined similarly.
\end{definition}

Suppose \(u\) is a bounded subsolution to in the sense of Definition~\ref{subsolutiondef} then \(u\) is also a subharmonic function. We denote the harmonic lift of \(u\) as \(\hat{u}\). For a supersolution \(v\), in the sense of Definition~\ref{supsolutiondef}, we may do a similar procedure and obtain \(\hat{v}\). 
 
Next, we show that the harmonic lifts of sub/ supersolutions are still sub/supersolutions.
\begin{lemma}\label{l.tangential-conv+lift-soln}
     Suppose \(u,v\) are respectively bounded sub-/supersolutions to \eqref{e.grad-degen-neumann}. Then \(\hat{u},\hat{v}\) will still be sub-/supersolutions. Moreover, if \(u\) satisfies the boundary maximum principle then \(\hat{u}\) will as well.
 \end{lemma}

 \begin{proof}
     The interior PDE follows from the properties of harmonic lift established above. 
 Because of the inequalities \(\hat{u}\ge u\) and \(\hat{u}= u\) on the boundary \(\pt B_1^+\), if a test function $\varphi$ touches $\hat{u}$ from above on $B_1'$ then it also touches $u$ from above at the same point. The viscosity subsolution property of $\hat{u}$ follows immediately from this observation. The supersolution property is similar.
 
 As for the boundary maximum principle, this property depends only on the values on $B_1'$, and $\hat{u} = u$ on $B_1'$.
 \end{proof}

\subsection{Comparison principle}

Let us now prove the comparison principle in dimension \(d\ge 2\).

 \begin{proof}[Proof of Theorem \ref{comparisonprinciple}]
 It suffices to consider bounded \(u,v\) since we can replace, for some big \(N>0\), \(u\) by \(\max\{u,-N\}\) and \(v\) by \(\min\{v,N\}\). Note that \(u\) is upper semi-continuous and \(v\) is lower semi-continuous the compact set $\overline{B_1^+}$, so there is some \(N>0\) large so that \(u\le N\) and \(v\ge -N\). This combined with \(v\ge u\) on the boundary ensures that \(\min\{v,N\} \ge \max\{u,-N\}\) on the boundary for $N$ large enough.
 
 We first claim that for all \(s>0\) there is a \(\delta=\delta(u,v,s)>0\) such that
 \[
 v+s\ge u \ \hbox{ on }\bigcup_{x\in\pt B_1\cap\{x_1\ge0\}} B_\delta(x) \cap \overline{B_1^+}=: D_\delta.
 \]
 Indeed, otherwise there would be a \(s_0>0\) and a sequence \(x_j\rta x\in \pt B_1\cap\{x_1\ge0\}\) such that \(v(x_j)+s_0\le u(x_j)\) for all \(j\), then we would have
\[
u(x)-v(x) \ge \limsup_{j\rta\infty} u(x_j)-v(x_j) \ge s_0>0,
\]
which contradicts the assumption.

 For \(\ep>0\), we now study in the smaller domain \(U_\ep=B_{1-\gamma(\ep)}^+\), with both \(\ep\) and \(\gamma=\gamma(\ep)>0\) sufficiently small. Fix $s_0>0$ small and in the following we always assume $\gamma(\ep) < \delta(u,v,s_0)=: \delta_0$. On \(U_\ep\) we define
 \[
 \hat{u}^\ep := \widehat{\Tc^\ep u} \ \hbox{ and } \ \hat{v}_\ep :=\widehat{\Tc_\ep v}.
 \]
 By \lref{tangential-conv+lift-soln} these are, respectively, an upper semicontinuous strong subsolution of \eref{grad-degen-neumann} satisfying boundary maximum principle, and a lower semicontinuous supersolution of \eref{grad-degen-neumann} in $U_\ep$. We write $\hat{u}$ and $\hat{v}$ dropping the $\ep$-dependence when it is not important.  Notice that because \(\Tc^\ep u\) and \(\Tc_\ep v\) are continuous when restricted to \(B_1'\), \(\hat{u}\) and \(\hat{v}\) are also continuous up to \(B_1'\) by standard boundary barrier arguments for harmonic functions.
 
 Now we make a perturbation to a strict supersolution. Let \(\eta\gg\ep>0\) be a fixed small number, we further consider the following modified functions
\[
\overline{u}=\hat{u}^\ep=\hat{u}  \ \hbox{ and } \
 \overline{v}=\hat{v}_\ep-\eta (x_1)_+ +10\eta.
 \]
It then suffices to show that 
\be\label{intermediateproofmaximumprin}
\max_{\overline{U_\ep}}( \overline{u} - \overline{v} )= \max_{\pt{U_\ep}\setminus U_\ep'} (\overline{u} - \overline{v}). 
\ee
Indeed, if we have established this equality, then on one hand we have
\[
\begin{split}
    \max_{\overline{U_\ep}} (\overline{u} - \overline{v})& \ge \max_{\overline{U_\ep}} (\Tc^\ep u - \Tc_\ep v) -100\eta\\
    &\ge \max_{\overline{U_\ep}}  (u -  v) -100\eta,
\end{split}
\]
 and on the other hand, we have by Lemma \ref{equalwehnsupenvelope}, \ref{somepropertoftangentialconvol} and Remark \ref{sompropertyofhalfrelaxed} (5)
 \[
 \begin{split}
   \max_{\pt{U_\ep}\setminus U_\ep'} (\overline{u} - \overline{v})& = \max_{\pt{U_\ep}\setminus U_\ep'} (\Tc^\ep u - \Tc_\ep v +\eta (x_1)_+ -10\eta)\\
   & \le \max_{D_{\delta_0}} \Tc^\ep (u-v)  + 100\eta\\
     &= \max_{D_{\delta_0}}  (u-v)  + o_{\ep}(1) + 100\eta\\
   &\le s_0 + o_{\ep}(1) + 100\eta.
 \end{split}
 \]
Sending \(\ep,s_0,\eta\rta0^+\) would complete the proof.

{Let $\hat{x}$ be a point in $\overline{U}_\ep$ where the maximum in \eqref{intermediateproofmaximumprin} is achieved.} We need to show that $\hat{x}$ cannot be in \(U_\ep\sqcup U_\ep'\).  By strong maximum principle for harmonic functions \(\hat{x} \not\in U_\ep\). 

We now show that \(\hat{x} \not\in U_\ep'\). According to Remark \ref{generalizedCauNauGammau} we partition the flat boundary portion \(B_1'\) into
\[
B_{1-\gamma(\ep)}'= \Na_{\overline{v}} \sqcup \Ca_{\overline{v}} \sqcup \Gamma_{\overline{v}},
\]
with \(\Na_{\overline{v}}= \Na_{\hat{v}},\Ca_{\overline{v}}= \Ca_{\hat{v}}\) and \(\Gamma_{\overline{v}}= \Gamma_{\hat{v}}\). For convenience, we will write them as \(\Na,\Ca\) and \(\Gamma\) respectively. 

We may without loss assume that \(\hat{x}\in \Na\sqcup \Gamma\), because if \(\hat{x}\in \Ca\) then \(\hat{u}-\hat{v}=\hat{u}-K\) in a component of \(\Ca\) for some constant \(K\), and by boundary maximum principle of \(\hat{u}\) there would be another point \(\hat{x}^\ast\in \Gamma=\pt'\Ca\) such that \(\overline{u}(\hat{x}^\ast)=\overline{u}(\hat{x})\). We can then replace \(\hat{x}\) by this new \(\hat{x}^\ast\).

 Next, we observe that, by the definition of tangential sub/sup-convolutions, at the maximum point \(\hat{x}\) there exist two quadratic functions \(P_1, P_2\) on \(\R^{d-1}\) and a constant \(c\) such that \(P_1\ge\restr{\overline{v}}{B_1'}+c\ge \restr{\overline{u}}{B_1'}= \restr{\hat{u}}{B_1'} \ge P_2\) and \(P_1(\hat{x})=P_2(\hat{x})=:l\). In particular, for some \(q\in \R^{d-1}\) we can write
 \be\label{P-1andP-2}
 P_1(x')=l + q\cdot(x'-\hat{x}) + \frac{1}{2\ep}|x'-\hat{x}|^2,\, \hbox{ and }P_2(x')=l + q\cdot(x'-\hat{x}) - \frac{1}{2\ep}|x'-\hat{x}|^2.
 \ee
 For convenience, we assume \(c=0\), since it will not affect the proof. We claim that for any \(h>0\) small there exists a smooth function \(w\) that touches \(\overline{v}\) from below at \(\hat{x}\) and
 \[
 \frac{\partial w}{\partial x_1}(\hat{x}) \ge -h.
 \]
 This claim will immediately imply that \(\hat{x}\not\in \Na\sqcup\Gamma\) and lead to a contradiction because otherwise \(w^\eta=w+\eta(x_1)_+ -10\eta\) would touch \(\hat{v}\) from below at \(\hat{x}\) and \(\pt_1w^\eta(\hat{x}) \ge \eta-h>0\), violating the supersolution condition. 
 
 To prove the existence of such $w$, we let \(\ep\gg\tau>0\) be a small number and consider in the half ball \(B_\tau^+(\hat{x})\) the following functions \(w_{\tau,i}\) with \(i=1,2\):
 \be
 \bca
 \Delta w_{\tau,i} (x)=0,&x\in B_\tau^+(\hat{x}),\\
 w_{\tau,i} (x) = \hat{u}(x),&x\in \partial B_\tau^+(\hat{x})\cap\{x_1\ge0\},\\
 w_{\tau,i} (0,x') = P_i(x'),&x=(0,x')\in  B_\tau^+(\hat{x})\cap\{x_1=0\}.
 \eca\label{upperlowerextensionofpis1}
 \ee
The boundary data is potentially discontinuous on $\partial B_\tau^+(\hat{x})\cap\{x_1\ge0\}$, so we are solving \eqref{upperlowerextensionofpis1} by Perron's method as in \sref{harmonic-lift}. However, the boundary data on $B_\tau'$ is polynomial so $w_{\tau, i}$ are smooth in a neighborhood of $\hat{x}$. Now \(w_{\tau,1}\) touches \(\hat{u}\) from above and \(w_{\tau,2}\) touches \(\hat{u}\) from below at \(\hat{x}\), because of the ordering \((w_{\tau,1})_\ast\ge \hat{u} \ge w_{\tau,2}^\ast\) on the boundary \(\pt B_\tau^+\). According to the subsolution condition of \(\hat{u}\),
 \[
 \frac{\partial w_{\tau,1}}{\partial x_1} (\hat{x}) \ge 0.
 \]
 Notice that, on the other hand, \(w_{\tau,2}\) touches \(\overline{v}\) from below at \(\hat{x}\). We then just need to show that \(\partial_1 w_{\tau,2}(\hat{x})\) is sufficiently close to \(\partial_1w_{\tau,1}(\hat{x})\) when \(\tau\) is chosen small. To that end, we need to control \(\Tilde{w}_\tau=w_{\tau,1}-w_{\tau,2}\), which will satisfy the following equation,
 \be
 \bca
 \Delta \Tilde{w}_{\tau} (x)=0,&x\in B_\tau^+(\hat{x}),\\
 \Tilde{w}_{\tau} (x) = 0,&x\in \partial B_\tau^+(\hat{x})\cap\{x_1\ge0\},\\
 \Tilde{w}_{\tau} (0,x') = P_1(x')-P_2(x')=: P(x'),&x=(0,x')\in  B_\tau^+(\hat{x})\cap\{x_1=0\},
 \eca
 \ee
 where by \eqref{P-1andP-2} \(P(x')=  \frac{1}{\ep}|x'-\hat{x}|^2\). By the transformation \(w_\tau(z):= \frac{\ep}{\tau^2}\Tilde{w}_\tau(\tau z +\hat{x})\), we observe that \(w_\tau\) satisfies
 \be
 \bca
\Delta w_{\tau} (z)=0,&z\in B_1^+(0),\\
 w_{\tau} (z) = 0,&z\in \partial B_1^+(0)\cap\{z_1\ge0\},\\
 w_{\tau} (0,z') = |z'|^2,&z=(0,z')\in  B_1^+(0)\cap\{z_1=0\},
 \eca
 \ee
 which has a universal bound $C$ on its Lipschitz constant in $B_{1/2}$ and so
\[
|\partial_1 \tilde{w}_\tau(0)| = \frac{\tau}{\ep}|\partial_1 w_\tau(0)|\le C\ep^{-1}\tau.
 \] 
 which can be made arbitrarily small by choosing $\tau$ small enough depending on $\ep$.  Note that $\ep>0$ is a fixed positive number in this argument.

\end{proof}

\subsection{Minimal supersolutions are exactly the solutions satisfying boundary maximum principle}

In this section, we show that the minimal supersolution to Equation \eqref{e.grad-degen-neumann} would satisfy the strong subsolution condition, and hence the boundary maximum principle.

To construct a minimal supersolution let us first write \eqref{e.grad-degen-neumann} in the following form
\be
\begin{cases}\label{e.neumandegenewithg}
\Delta u = 0 & \hbox{in } B_1^+\\
\min\{\partial_1 u,|\gd' u|\} = 0 &\hbox{on } B_1'\\
u=g &\hbox{ on }\pt B_1\cap\{x_1\ge0\},
\end{cases}
\ee
with \(g\) an arbitrary continuous function. 
\begin{definition}
    An upper semicontinuous function \(u\) is called a viscosity subsolution to \eqref{e.neumandegenewithg} if it is a subsolution to \eqref{e.grad-degen-neumann} and \(u\le g\) on \(\pt B_1\cap\{x_1\ge0\}\). Similarly, a lower semicontinuous function \(v\) is called a viscosity supersolution to \eqref{e.neumandegenewithg} if it is a supersolution to \eqref{e.grad-degen-neumann} and \(v\ge g\) on \(\pt B_1\cap\{x_1\ge0\}\).  
\end{definition}

One can easily check that
\[
w_{\textup{sub}}(x):=-\norm{g}_\infty,
\]
is a subsolution to \eqref{e.neumandegenewithg} that satisfies boundary maximum principle on \(B_1'\). Similarly, we know that 
\[
w_{\textup{sup}}=  \norm{g}_\infty
\]
is a supersolution to \eqref{e.neumandegenewithg}.

Define the Perron's method minimal supersolution 
\be
\label{minimalsuper}
v_{\mathrm{min}}(x):= \inf\{v(x)\,;\,v\hbox{ is a supersolution to \eqref{e.neumandegenewithg} }\}.
\ee

\begin{proposition}\label{p.minimal-supersolution}
    The function \(v_{\mathrm{min}}\) is the unique viscosity solution to \eqref{e.neumandegenewithg} that satisfies the strong subsolution condition, Definition \ref{definitionofstrongsubsolution}, and equivalently the boundary maximum principle.
\end{proposition}

\begin{remark}
    By applying comparison principle, Theorem \ref{comparisonprinciple}, it can be observed immediately that all viscosity solutions to \eqref{e.neumandegenewithg} are bounded from above by the solution \({v}_{\textup{N}}\) to the mixed boundary problem 
    \be
    \bca
     \Delta v_{\textup{N}} =0 ,&\hbox{ in }B_1^+\\
     v_{\textup{N}} = g , &\hbox{ on }\pt B_1\cap\{x_1\ge0\}\\
     \pt_1 v_{\textup{N}} =0, &\hbox{ on }B_1'.
    \eca
    \ee
    and from below by the minimal supersolution \(v_{\mathrm{min}}\) to \eqref{e.neumandegenewithg}. In Example \ref{quasisolexamp1}, the solution \(w(x,y)=-\mathrm{Re}\lb(x+iy)^{3/2}\rb\) to the Signorini problem coincides with the minimal supersolution to \eqref{e.grad-degen-neumann} because \(w\) satisfies the boundary maximum principle. In Example \ref{quasisolexamp2}, the solution \(w^-(x,y)=\mathrm{Re}\lb(x+iy)^{5/2}\rb\) to the sign-reversed Signorini problem is also the minimal supersolution because it also satisfies boundary maximum principle.

\end{remark}

\begin{proof}

By the discussion above \pref{minimal-supersolution}, we have established that \(w_{\textup{sub}}\le v:= v_{\min}\le w_{\textup{sup}}\) is a well-defined bounded function. By applying a slight modification of the classical Perron's method, we know that the lower envelope \(v_\ast\) is a supersolution, and \((v_\ast)^\ast\) is a subsolution. Moreover, \((v_\ast)^\ast=v_\ast\) in \(B_1^+\) is harmonic, and we also have \((v_\ast)^\ast=v_\ast =g\) on \(\pt B_1\cap \{x_1\ge0\}\). 

It then suffices to show that \((v_\ast)^\ast\) also satisfies the boundary maximum principle. Indeed, with the boundary maximum principle we can apply the comparison principle \ref{comparisonprinciple}, which implies that \((v_\ast)^\ast\le v_\ast\), and then \(v=(v_\ast)^\ast=v_\ast\) on the whole \(\overline{B_1^+}\), which implies the continuity of \(v\) up to boundary and that \(v\) is a viscosity solution to \eqref{e.neumandegenewithg}.

We would like to show using the strong subsolution condition, which is equivalent to the boundary maximum principle according to Remark \ref{equivalencebetweenboundamxandstrongsubsolution}. Suppose for a bounded supersolution \(\tilde{v}\) to \eqref{e.neumandegenewithg}, the upper envelope \(\tilde{v}^\ast\) does not satisfy the strong subsolution condition, then there is a smooth up to boundary function \(\psi(x_1,x')\equiv \phi(x_1)\) (we may without loss assume that \(\phi''(x_1)\le 0\)) that touches \(\tilde{v}^\ast\) from above at \(z\in B_1'\), and there is a domain \(\Omega\) containing \(z\) with \(\Omega_h=\bigcup_{x\in \Omega} B_h(x) \cap \overline{B_1^+} \csubset B_1^+\sqcup B_1'\) for some small \(h>0\), so that 
\[
\psi\ge \tilde{v}^\ast + s,\,\hbox{ on } \overline{\Omega_h\setminus \Omega},
\]
for some small constant \(s>0\). To reach a contradiction, we consider the function 
\[
v_c =\bca
\min\{\tilde{v}, \psi-s/2\},&\hbox{ in }\overline{\Omega_h}\\
v,&\hbox{ elsewhere}.
\eca
\]
The proof will complete if we observe that \(v_c\) is a supersolution to \eqref{e.neumandegenewithg} but there is a point \(z_0\in \Omega\) such that \(v_c(z_0)<\tilde{v}(z_0)\). Indeed, \(v_c=v\) outside \(\Omega\) and \(\psi-s/2\) is a supersolution, and hence \(v_c\) is a supersolution to \eqref{e.neumandegenewithg}. On the other hand, by definition, there is a sequence \(z_j\in {B_1^+}\sqcup B_1'\) such that \(z_j\rta z\) and \(\tilde{v}(z_j)\rta \tilde{v}^\ast(z)\) as \(j\rta\infty\). One can check that \(v_c(z_j) < \tilde{v}(z_j)\) for a sufficiently large \(j\).

\end{proof}

\section{Flat asymptotic expansion of a singular Bernoulli problem\label{section.flatasymsingularbernoulli}}
In this section, we study the following discontinuous anisotropic model
\be
\bca
\Delta u =0, &\text{ in }\{u>0\}\setminus \overline{W},\\
u\equiv1,&\text{ on }\overline{W},\\
|\gd u|^2= Q^2(\gd u), & \text{ on } \partial \{u>0\},
\eca\label{discontinuous2}
\ee
with the anisotropy \(Q\) being a 0-homogeneous function of the form
\be\label{defineQ}
Q(e)=\bca
1,&e\ne e_1\\
2,&e=e_1,
\eca
\ee
where \(e_1,e_2,\dots,e_d\) form an orthonormal basis for \(\R^d\). More rigorously, we will assume \(u\) to be a minimal supersolution to \eqref{discontinuous2}. For the definitions and discussion of such solutions, see \cites{feldman2019free,feldman2021limit}.

We already know that when \(W\) is a convex domain, then the free domain \(\{u>0\}\) is also convex and the regularity problem of the free boundary can be easily reduced to the case discussed in \cite{changlara2017boundary}, and the optimal regularity is exactly \(C^{1,1/2}\).

In the case that \(W\) is non-convex, we would like to divide the free boundary \(\partial\{u>0\}\) into two disjoint parts:
\be
\partial\{u>0\}=:  \{|\gd u|>1\}\sqcup R=: \Lambda\sqcup R,
\ee
where ``\(|\gd u|>1\)'' is defined in the sense of viscosity. We would like to show the following facts:
\begin{itemize}
    \item[1.] \(\Lambda\) is relatively open in \(\partial\{u>0\}\);
    \item[2.] \(\Lambda\) is composed of intervals that are parallel to \(e_2\). 
\end{itemize}

Based on the second statement we would like to show that the free boundary would be \(C^{1,\alpha}\) in a neighborhood of \(\Lambda\cup \partial'\Lambda\). Following the idea of \cites{desilva2011free,changlara2017boundary}, we would like to study the asymptotic expansion of the solution \(u\) near \(\Lambda\cup\pt'\Lambda\). 

Let us begin with a more precise definition of \(\Lambda\).

\begin{definition}
    \(\Lambda\) is composed of all the points \(x\in \pt \{u>0\}\) such that there is a smooth function \(\phi\) touching \(u\) from below at \(x\) satisfying \(\gd\phi(x)\) parallel to \(e_1\) and \(|\gd\phi(x)|>1\).
\end{definition}

According to this definition, we know that any point \(x\in\Lambda\) is an inner regular point of \(\Dupo\). For the convenience of the arguments, we may without loss assume that \(x=0\) and the solution \(u\) satisfies a half-flatness condition in a unit ball \(B_1=B_1(0)\) for small \(\varepsilon>0\)
\be
u (x)\ge \gd\phi(0)\cdot x - \varepsilon=: \alpha x_1 - \varepsilon,\label{halfflat}
\ee
for some \(\alpha>1\) and we can also assume \(B_2(2e_1)\cap B_1\subset \{u>0\}\cap B_1\).

Let us now prove that \(\Dupo\) is flat near \(x=0\).

\begin{lemma}
    With the above assumptions, we show that there is a small number \(\delta>0\) such that \(B_\delta\cap\{u>0\}=B_\delta\cap\{x_1>0\}\).\label{lambdaisopenforu}
\end{lemma}

\begin{proof}
    Observe first that by \eqref{halfflat} we know that for \(\overline{x}=\frac{1}{5}e_1\)
    \[
    u\lb\overline{x}+t\rb - \frac{1}{5} \ge \frac{1}{5}(\alpha-1) -\varepsilon\gg \varepsilon,
    \]
    for \(t=(0,t_2,\cdots,t_d)^T\in\{0\}\times\R^{d-1},\,|t|<1/20\). By Harnack inequality, we may without loss assume that 
    \[
    u(x)\gg \varepsilon,\,\forall x\in B_{1/10}\lb\overline{x}\rb.
    \]
    Let \(w\) be a positive function such that it is strictly subharmonic in the annulus \(A=B_{3/4}(\overline{x})\setminus\overline{B_{1/20}(\overline{x})}\) and is \(1\) on the inner boundary and 0 on the outer boundary. We extend \(w\) to be constantly 1 in the inner disc. It is enough to take 
    \[
    w= c (|x-\overline{x}|^{-\gamma}-(3/4)^{-\gamma}), x
    \in A
    \]
    with \(c\) chosen so that the conditions above are satisfied.
    
    Now we follow De Silva's argument and compare
    \[
    u(x)\ge x_1 - \varepsilon + C_0\varepsilon (w(x)-\tau)=: v_\tau(x), \,x\in B_{1/2}?
    \]
   This is true indeed for \(\tau=1\), and we want to show that this is also true for \(\tau=0\). Let \(1\ge\tau^\ast\ge0\) be the smallest number such that the above inequality holds then by strict subharmonicity and strict inequality inside \(B_{1/20}(\overline{x})\) the touching point can only be at the boundary. But because of strict inequality \(|\gd v_{\tau^\ast}|>1\) at the touching point, the touching point can only be the origin, which is naturally excluded by choosing \(C_0<1/2\). This shows that \(\tau^\ast\le0\) and hence the inequality also holds for \(\tau=0\).

   The same argument shows that \(v_0(x+h(0)e_1)\) touches \(u\) from below exactly at the origin for some unique \( h(0) \approx \varepsilon\). Let \(\norm{h}_{L^\infty(-\delta,\delta)}\le C\varepsilon\) be the function such that for each \(|t|\ll 1/20\) we have that \(v_0(x-t+h(t)e_1)\) touches \(u\) from below exactly at the boundary. Similar as before, all the touching points would be of the form \((h(t)+c)e_1-t\) for some constant \(c\) and each fixed \(t\). 
   
   Now the function \(h(t)\) is semi-convex and Lipschitz. Moreover because of the existence of upper touching functions having gradient zero at each \(|t|<\delta\) we conclude that \(\gd h(t)=0\) in the viscosity sense and hence \(h(t)\equiv h(0)\) is a constant.

\end{proof}

\subsection{A Harnack inequality assuming both flatness and the free boundary being a function graph}
We are interested in the regularity of the free boundary near the relative boundary \(\pt'\Lambda\) of \(\Lambda\). Let \(0\in \pt'\Lambda\) and we consider a solution \(u\) to \eqref{discontinuous2} that is restricted to the unit ball \(B_1(0)\). Unlike in the case of Chang-Lara and Savin \cite{changlara2017boundary}, we don't necessarily have the inner or outer regularity of the free boundary at \(0\), and so we don't immediately obtain the differentiability of \(u\) near the origin. However, because of the definition of \(\Lambda\) and Lemma \ref{lambdaisopenforu}, we do know that if \(u\) admit differentiability at \(0\) then \(|\gd u|(0)=1\) and hence after a rescaling and a small translation, the following \emph{flatness} condition would be satisfied:
\be
(x\cdot p)_+\le u(x) \le (x\cdot p+\varepsilon)_+,\,x\in B_1(0),\,p\in \partial B_1(0).\label{fltanessforp}
\ee
For the convenience of analysis, we also assume that the free boundary is the \emph{function graph} of a continuous function \(f\) on the set \(\{x\cdot p=0\}\). 

The main difficulty of our problem is that we don't a priori know the position of the facets \(\Lambda\), but we do know the following dichotomy: either \(\pt \{u>0\} \cap B_{1/(400d)} = \Lambda_u\cap B_{1/(400d)}\) (i.e. the whole free boundary in \(B_{1/(400d)}\) is completely flat), or the following Harnack inequality is satisfied.

\begin{lemma}
    Let \(p(x):= x\cdot p \) and \(l=\frac{1}{100d}\). There exist constants \(\overline{\varepsilon}=\overline{\ep}(d)>0\) and \(l/2> \mu=\mu(d)>0\) such that if \(u\) satisfies the anisotropic boundary condition \(|\gd u|=Q(\gd u)\) with \(Q\) defined in \eqref{defineQ} on the free boundary \(\Dupo\cap B_1\) that is a function graph, \(u\) is harmonic in its positive set satisfying the \(\ep\)-flatness condition \eqref{fltanessforp} with \(0<\ep\le \overline{\ep}\), then if at \(\overline{x}=\frac{1}{5}p\)
    \be\label{subsolutioncaseinharnackcompar}
     u(\overline{x})\ge p(\overline{x})+\frac{\ep}{2},
    \ee
    either
    \begin{enumerate}
        \item[i.] the free boundary satisfies \(\pt \{u>0\} \cap B_{l/4} = \Lambda_u\cap B_{l/4}\), and so is completely flat in \(B_{l/4}\);
        \item[ii.] or there is a point \(x^\ast\in \pt \{u>0\} \cap B_{l/4}\) not contained in \(\Lambda_u\cap B_{l/4}\) and
        \be
    u\ge (p(x)+c\varepsilon)_+,\,\text{ in }\overline{B}_{\mu},\label{changlarasavinidea}
    \ee
    for some \(0<c=c(d)<1\). 
    \end{enumerate}
   If 
    \be
     u(\overline{x})\le p(\overline{x})+\frac{\ep}{2},
    \ee
    then similarly
    \be
    u\le (p(x)+(1-c)\varepsilon)_+,\,\text{ in }\overline{B}_{\mu}.
    \ee
\end{lemma}

\begin{remark}
    This Harnack inequality doesn't assume \(0\in \pt \{u>0\}\cap B_1\). It also shows that there is no degeneracy in the flatness parameter \(\Bar{\ep}\) as \(p\approx e_1\).
\end{remark}

\begin{proof}
We focus on the first case that \(u(\overline{x})\ge p(\overline{x})+\frac{\ep}{2}=1/5+\frac{\ep}{2}\), because the second case exactly falls in the case of De Silva \cite{desilva2011free}. In this case, by classical Harnack inequality, we obtain for some universal constant \(m>0\)
    \be
    u(x)-p(x)\ge m\ep,\,\forall x\in B_{1/10}\lb\overline{x}\rb.\label{harnackimplylocalinequaltiyin1over10}
    \ee
 Let \(w\) be a positive function such that it is strictly subharmonic in the annulus \(A=B_{l+1/5}(\overline{x})\setminus\overline{B}_{1/20}(\overline{x})\) for some \(l>0\) to be determined, and is \(1\) on the inner boundary and 0 on the outer boundary. We extend \(w\) to be constantly 1 in the inner disc. It is enough to take 
    \[
    w= \tilde{c} (|x-\overline{x}|^{-\gamma}-(l+1/5)^{-\gamma}), x
    \in A
    \]
    with \(\tilde{c}\) chosen so that the conditions above are satisfied. Indeed, \(w\) is strictly subharmonic in \(A\) when \(\gamma>d-2\). We will choose \(\gamma=d-1\) in the following proof.

    Now we follow De Silva's argument and compare for \(z\in B_{l/2}(0)\) and \(\tau\ge 0\)
    \be\tag{DS}\label{desilvaarg}
    u(x)\ge p(x)  + m\varepsilon (w(x-z)-\tau)=: v_{\tau,z}(x), \,x\in B_{l/2}?
    \ee
Notice that we always have the above inequality for \(\tau \ge 1\) because of the \(\ep\)-flatness assumption \eqref{fltanessforp}. Even for \(\tau = 0\), by \eqref{harnackimplylocalinequaltiyin1over10} we still have 
\be\label{boundeoutsideAplusz}
u(x) \ge v_{0,z}(x),\,z\in B_{l/2}(0),\,\hbox{ and } x \notin \overline{A+z}.
\ee
On the other hand, we claim that the portion of the barrier free boundary \(\pt\{v_{\tau,z}>0\}\setminus \{p(x)=0\}\) is contained in \(B_{r(l)}(z')\) with \(r(l)=\sqrt{3l/5 + 3l^2/4}\) and \(z'=z-p(z)p\). Indeed, the barrier free boundary portion \(\pt\{v_{\tau,z}>0\}\setminus \{p(x)=0\}\), no matter what \(z\in B_{l/2}(0)\) and \(\tau\ge0\), is contained in \(A+z\cap \{p(x)<0\}\subset B_{r(l)}(z')\). 

Let us now choose a proper \(l=l(d)>0\) so that for all \(0<\ep\le \overline{\ep}\), \(z,\tau\), the barrier free boundary portion \(\pt\{v_{\tau,z}>0\}\setminus \{p(x)=0\}\) is the graph of a convex function on \(\{p(x)=0\}\). According to implicit function theorem, for \(\ep>0\) small, the boundary \(\pt\{v_{\tau,z}>0\}\) can be denoted by a function \(\xi_1=g(\xi')\), with \((\xi_1,\xi')\) a new Euclidean coordinate system of \(\R^d\) such that \(\pt_{\xi_1}=p\cdot \gd\). It then suffices to show that, for a constant \(l=l(d)\), \(g\) is always a convex function. Indeed, we can observe that if we write \(v_{\tau,z}(\xi_1,\xi')=\xi_1-\ep\phi(\xi)\), with \(\phi(\xi)=-m(w(\xi-z)-\tau)\), we obtain
\be\label{convexgraphg1}
    (1-\ep \pt_{\xi_1}\phi) \gd_{\xi'} g = \ep \gd_{\xi'} \phi,
\ee
and hence
\be
\lb \gd_{\xi'}\rb^2 g = \ep \lb \gd_{\xi'}\rb^2 \phi + O(\ep^2).
\ee
But we know that 
\be\label{convexgraphg2}
 \frac{1}{m\tilde{c}}\lb \gd_{\xi'}\rb^2 \phi = \frac{\gamma}{|\xi-z-\overline{x}|^{\gamma+2}} \hbox{Id}_{(d-1)\times(d-1)} - \gamma(\gamma+2) \frac{(\xi'-z')\otimes (\xi'-z')}{|\xi-z-\overline{x}|^{\gamma+4}},
\ee
with the smallest eigenvalue
\[
\gamma- \gamma(\gamma+2)\frac{|\xi'-z'|^2}{|\xi'-z'|^2+|\xi_1-z_1-1/5|^2} >0,
\]
will require
\[
|\xi'-z'|^2<\frac{1}{\gamma+1} |\xi_1-z_1-1/5|^2=\frac{1}{d}|\xi_1-z_1-1/5|^2.
\]
Since \(\xi\in \pt\{v_{\tau,z}>0\}\subset\{-\ep\le p(x)\le0\}\), we know that \(|\xi_1-z_1-1/5|>1/5-l/2>0\) for \(0<l<2/5\). On the other hand, by the prior discussions on \(r(l)\), we know that on the nontrivial portion 
\(\pt\{v_{\tau,z}>0\}\setminus\{p(x)=0\}\)
\[
|\xi'|=|x'-z'|\le r(l).
\]
This means it suffices to set
\be\label{convexgraphg3}
r(l)=\sqrt{3l/5 + 3l^2/4} < \frac{1}{\sqrt{d}} (1/5-l/2),
\ee
or simply choose \(l=\frac{1}{100d}\).

By the previous paragraph, we deduce that there is at most one point \(y^\ast=y^\ast(\tau,z)\) on the nontrivial portion of the barrier boundary \(\pt\{v_{\tau,z}>0\}\setminus\{p(x)=0\}\) such that \(\gd v_{\tau,z}\) is parallel to \(e_1\) at \(y^\ast\). Let us now discuss the position of \(y^\ast\) for different \(z\) and \(\tau\). Observe that if one writes \(1=p\cdot p\), then we have
\[
\begin{split}
  v_{\tau,z}(x) &= p(x)  + m\varepsilon (w(x-z)-\tau)\\
  &= p(x-m\ep\tau p) + m \ep w(x-z)\\
  &= v_{0,z-m\ep\tau p}(x-m\ep\tau p)\\
  &=: v_{0,z-\beta p}(x-\beta p),
\end{split}
\]
where we define \(\beta = m\ep \tau\). By this formula we have 
\[
y^\ast(\tau,z) = y^\ast(0,z-\beta p) + \beta p.
\]
Furthermore, we have for \(z'=z-p(z)p\)
\[
y^\ast(\tau,z) =y^\ast(0,z-\beta p)= y^\ast(0,(p(z)-\beta)p) + z' + \beta p.
\]
Thus, we obtain that the positions of \(y^\ast(\tau,z)\) are just translations of \(y^\ast(0, \zeta p)=:y^\ast(\zeta)\). Using this one parameter family of \(y^\ast\), we now consider the following subfamily of \(v_{\tau,z}\): for each \(\tau\ge0\) such that if \(y^\ast(\tau,0)=y^\ast(-m\ep \tau )+m\ep \tau p\) exists in the nontrivial portion \(\pt\{v_{\tau,0}>0\}\setminus\{p(x)=0\}\) and the tangential magnitude \(|(y^\ast)'(\tau,0)|=|(y^\ast)'(-m\ep \tau )| \le l/4\), we choose \(z=z'(\tau)\in B_{l/2}(0)\cap\{p(x)=0\}\) so that
\[
z(\tau) + (y^\ast)'(-m\ep\tau) = (x^\ast)',
\]
which means that we can make sure that \((y^\ast)'(\tau,z(\tau))=(x^\ast)'\), and therefore the touching point can never be \(y^\ast\) by the assumptions on \(x^\ast\); for the rest cases (either \(\tau\ge0\) doesn't correspond a \(y^\ast\) or \(|(y^\ast)'|\ge l/4\)) we choose \(z(\tau)=0\). 

We are now able to complete the proof by simply considering the De Silva argument \eqref{desilvaarg} with \(v_{\tau,z}\) replaced by \(\tilde{v}_{\tau}:=v_{\tau,z(\tau)}\) and hence showing the inequality \eqref{changlarasavinidea} for some constant \(l/2>\mu>0\) (even though we wrote \(l/2\) in the De Silva argument \eqref{desilvaarg}, we eventually obtain the improvement in a smaller ball of radius \(\mu\)). Let \(\tau^\ast\ge0\) be the smallest number such that \eqref{desilvaarg} is satisfied for \(\tilde{v}_{\tau^\ast}\). Because of our choice of \(z=z(\tau)\), it suffices to consider the case that \(y^\ast(\tau^\ast,z(\tau^\ast))\) exists but \(|(y^\ast)'(\tau^\ast,z(\tau^\ast))|\ge l/4\) (in other cases, because the touching can never happen on \(y^\ast\) due to our choice of \(z(\tau)\), this is a similar touching argument of De Silva, which shows that \(\tau^\ast=0\)). In this case, \(z(\tau^\ast)=0\), and the boundary portion \(\pt\{\tilde{v}_{\tau^\ast}>0\}\setminus\{p(x)=0\}\) is the function graph of a convex function \(\xi_1=g(\xi')\) on \(B_{r(l)}(0)\cap \{p(x)=0\}\), and \(\gd_{\xi'} g(0)=0\). According to \eqref{convexgraphg1},\eqref{convexgraphg2} and \eqref{convexgraphg3}, we know that the function \(g\) is in fact \(\ep m h(d)\)-strictly-convex for some \(h(d)>0\), and hence 
\[
g(0) + {\ep m h(d)} |(y^\ast)'(\tau^\ast,0)|^2 \le g\lb(y^\ast)'(\tau^\ast,0)\rb\le0,
\]
which shows that 
\[
g(0) \le -\frac{ m h(d) l^2}{16}\ep.
\]
On the other hand, we have 
\[
g(0)+ m\ep \lb w(g(0),0) - \tau^\ast\rb=0,
\]
which shows that
\[
\tau^\ast \le w(g(0),0) - \frac{ h(d) l^2}{16}.
\]
This leads to the following inequality
\[
u(x) \ge p(x) + m\ep \lmb w(x)-w(g(0),0) + \frac{ h(d) l^2}{16}\rmb.
\]
Because the gradient of \(w\) is bounded near the origin and the point \((g(0),0)\) is \(\ep\)-close to the origin, we know that there is a constant \(\mu(d)>0\) such that for all \(x\in B_{\mu(d)}(0)\), \(|w(x)-w(g(0),0)| \le \frac{ h(d) l^2}{32}\) and
\[
u(x) \ge p(x) + m\ep \lmb  \frac{ h(d) l^2}{32} \rmb,\,\hbox{ for all }x\in B_{\mu(d)}(0),
\]
where we can simply choose \(\mu(d)=\frac{ h(d) l^2}{64d 1000^d}\).

\end{proof}

\begin{corollary}[{\textbf{Harnack Inequality}}]
    There is a universal constant \(\overline{\varepsilon}\), such that if \(u\) is a viscosity solution to \eqref{discontinuous2} and it satisfies at some point \(x_0\in \Dupo\) and for some \(p\in \partial B_1(0)\)
    \be
    (x\cdot p + a_0)_+\le u(x) \le (x\cdot p +b_0)_+,\,\text{ in }B_r(x_0),
    \ee
    with
    \[
    b_0-a_0\le \varepsilon r,\,0<\varepsilon\le \overline{\varepsilon}
    \]
    then
    \[
    (x_\cdot p + a_1)_+ \le u(x) \le (x\cdot p + b_1)_+\text{, in }B_{r\mu}(x_0),
    \]
    with
    \[
    a_0\le a_1\le b_1\le b_0,\,b_1-a_1\le (1-c)\varepsilon r,
    \]
    and \(0<c<1\) universal.
\end{corollary}

\begin{proof}
    The proof is essentially the same as De Silva.
\end{proof}

 The above corollary, by a similar argument in \cite{desilva2011free}*{Corollary 3.2}, shows the following result.

\begin{corollary}\label{preforaatheorem}
    The functions 
    \[
    w_\varepsilon=\frac{u(x)-x\cdot p}{\varepsilon}
    \]
    have a uniform H\"{o}lder modulus of continuity at \(0\) in \(B_{1}\), outside a ball of radius \(\ep/\Bar{\ep}\), i.e. for all \(x\in B_1\) with \(|x|\ge \ep/\Bar{\ep}\)
    \[
    |w_\ep(x)-w_\ep(0)| \le C |x|^\gamma.
    \]
\end{corollary}

\subsection{A compactness result}
In this section, we show Proposition \ref{flatasymptoticexpansion} by proving several lemmas.

\begin{remark}
    We will also show that for \(p\ne e_1\), the function \(w\) would correspond to the solutions to the zero Neumann boundary condition.
\end{remark}
To prove this proposition we would like to study the following blow-up families of functions
\be
w_\varepsilon (x):= \frac{u(x)-x\cdot p}{\varepsilon}.
\ee
We argue by contradiction and assume that there is a subsequence \((u_k,\varepsilon_k)\), where \(u_k\) is a solution to \eqref{discontinuous2} in \(B_1(0)\) with \(\varepsilon_k\)-flatness and \(\varepsilon_k\rta 0^+\), and we define
\be
w_{k}:= \frac{u_k-x\cdot p}{\varepsilon_k}.\label{limitbernoulli1}
\ee
By Corollary \ref{preforaatheorem} and Arzela-Ascoli theorem, we know that \(w_k\) has compactness in \(C(\overline{B_{1/2}})\), which allows us to assume that for some \(w\in C^\alpha(\overline{B_{1/2}}\cap\{x\cdot p\ge0\})\), the following uniform convergence
\be
w_{k}\rta w,\,\hbox{uniformly on }B_{1/2}.\label{limitbernoulli}
\ee
In particular, the free boundary \(\lb\partial\{u_k>0\}\rb\cap B_{1/2}\) converges in the Hausdorff sense to \(\{x\cdot p =0\} \cap B_{1/2}\).

Now we finish this section by characterizing the equation for \(w\) in the following lemmas. For notational convenience we define the following subsets of \(\{x\cdot p \ge 0\} \cap B_{1/2}\):
\[
B^+:= \{x\cdot p > 0\} \cap B_{1/2};B':= \{x\cdot p =0\} \cap B_{1/2}.
\]
\begin{lemma}
    For general \(p\in \partial B_1(0)\), the limit function \(w\) is a viscosity subsolution to the following Neumann problem
    \be
    \bca
    \Delta w =0 &\text{ in }B^+,\\
    \partial_p w \geq 0 & \text{ on }B',
    \eca\label{neumannprob}
    \ee
    where \(\partial_p=p\cdot\gd\). Moreover, the function \(w\) is harmonic in \(B^+\).
\end{lemma}

\begin{proof}
It suffices to check the condition on the flat boundary because harmonicity of \(w_k\)'s are preserved under uniform convergence in the interior. Suppose there is a smooth function \(\varphi\) touching \(w\) from above at some \(x_0\in B'\), then by standard theory there are \((c_k,x_k)\) such that \(\varphi_k=\varphi+c_k\) touches \(w_k\) from above at \(x_k\in B_{1/2}\cap\lb\{u_k>0\}\cup\partial\{u_k>0\}\rb\), and \((c_k,x_k)\rta (0^+, x_0)\) as \(k\rta\infty\). Denoting \(\phi_k=x\cdot p+\varepsilon_k\varphi_k\), it is equivalent to say that \(\phi_k\) touches \(u_k\) from above at \(x_k\) for each \(k\). By performing the transformation \(\varphi\mapsto \varphi + \eta (x\cdot p-x_0\cdot p)-C(\eta)(x\cdot p-x_0\cdot p)^2\) for suitably chosen \(\eta,C(\eta)>0\), we may assume without loss that \(x_k\in \lb\partial\{u_k>0\}\rb\cap B_{1/2}\). Now because \(|\gd u_k|\ge 1\) on the free boundary for each \(k\), we have
\[
1+2\varepsilon_k \partial_p \varphi (x_k) + o(\varepsilon_k)=|\gd\phi_k|^2\ge 1,
\]
which implies that
\[
\partial_p\varphi(x_0)\ge 0.
\]

\end{proof}

\begin{lemma}
    In the case \(p=e_1\), \(w\) would satisfy the strong subsolution condition, Definition \ref{definitionofstrongsubsolution}. That is, there are no \(C^1\) up-to-boundary function of the form \(\varphi(x_1,x')\equiv\psi(x_1)\) that touches \(w\) from above in \(\Omega_h\cap\overline{B_1^+}\) at some \(x_0\in B_1'\) and \(\varphi>w\) in \(\overline{\Omega_h\setminus\Omega} \cap \overline{B_1^+}\) where \(\Omega\) is an arbitrary open domain of \(\R^d\) containing \(x_0\) and \(\Omega_h=\cup_{y\in\Omega} B_h(y)\) for some small \(h>0\) so that \(\overline{\Omega_h}\cap \overline{B_1^+} \csubset B_1^+\cup B_1'\).\label{conditioninfinity}
\end{lemma}

\begin{proof}
    Similar as before, we have \(\phi_k=x_1+\varepsilon_k(\varphi+c_k)\) touching \(u_k=x_1+\varepsilon_k w_k\) from above at \(x_k\in\partial\{u_k>0\}\cap B_{1/2}\) that converges to \(x_0\) as \(k\rta\infty\). Because of the strict inequality \(\varphi>w\) in in \(\overline{\Omega_h\setminus\Omega} \cap \overline{B_1^+}\) and uniform convergence of \(w_k\) to \(w\), we also have \(\phi_k>u_k\) in \(\overline{\Omega_h\setminus\Omega}\cap \{u_k>0\}\). If \(|\partial_1 \phi_k|(x_k)<2\), then 
    \[
    \Tilde{u}_k:= \bca
    \min\{u_k,(\phi_k)_+(x-\eta e_1)\},&\hbox{in }\Omega_h\cap B_1^+,\\
    u_k,&\hbox{elsewhere}
    \eca
    \]
    will become a new supersolution that is strictly smaller than \(u_k\) for some small \(0<\eta\ll r_1-r_2\), which is impossible because \(u_k\) is assumed to be a minimal supersolution to \eqref{discontinuous}. Now, we have
    \[
1+2\varepsilon_k \partial_1 \varphi (x_k) + o(\varepsilon_k)=|\gd\phi_k|^2\ge 2,
\]
which implies that
\[
\psi'(0)\approx\partial_1\varphi(x_k)\ge O(1/\varepsilon_k),
\]
for any \(k\) large.
\end{proof}

\begin{lemma}
    For general \(p\), we extend \(w\) to the whole \(B_{1/2}\) evenly, and then \(w\) is subharmonic in \(B_{1/2}\).\label{subharmonic}
\end{lemma}
\begin{proof}
    It suffices to check points \(x\in B_{1/2}'\). Indeed, suppose \(\varphi\) touches \(w\) from above at \(x\in B_{1/2}'\), then we may start with locally (\(x_p=x\cdot p\), \(x=(x_p,x')\))
    \[
    \psi(x_p,x')=\frac{\varphi(x_p,x')+\varphi(-x_p,x')}{2}.
    \]
    Now we have
    \[
    \Delta \psi(x_p,0)=\Delta \varphi(x_p,0),\,\partial_p\psi(x_p,0)=0,
    \]
    and \(\psi\) also touches \(w\) from above. Let \(\psi_\varepsilon=\psi - \varepsilon x_p +C_\varepsilon\) for small \(\varepsilon>0\) and some \(C_\varepsilon>0\). By standard theory, we may choose \(C_\varepsilon\) so that \(\psi_\varepsilon\) also touches \(w\) from above at some \(x_\varepsilon\in B_{1/2}\) and \(x_\varepsilon\rta x\) as \(\varepsilon\rta0^+\). We claim that all \(x_\varepsilon\not\in B_{1/2}'\) because the function \(w\) satisfies \(\partial_p w\ge 0\) in the viscosity sense, and hence if \(x_\varepsilon\in B_{1/2}'\)
    \[
\partial_p\psi_\varepsilon(x_\varepsilon)=\partial_p\psi(x_\varepsilon)-\varepsilon\ge0,
    \]
    which violates the definition of \(\psi\). Therefore, \(x_\varepsilon\in B_{1/2}\setminus B_{1/2}'\), and so \(\Delta \varphi(x)=\Delta \psi(x)\ge0\).
\end{proof}

\begin{lemma}
    In the case that \(p=e_1\), we show that if a smooth function \(\varphi\) touches \(w\) from below at some \(x_0\in B'\) and satisfies \(|\gd'\varphi|(x_0)>0\), then we have
    \[
    \partial_1 \varphi(x_0) \le 0.
    \]

\end{lemma}

\begin{proof}
We follow a similar procedure as the proof for subsolution and obtain a converging sequence \((c_k,x_k)\rta (0^+,x_0)\) such that \(\phi_k:= x_1+\varepsilon_k(\varphi+c_k)\) touches \(u_k\) from below at \(x_k\in \lb\partial\{u_k>0\}\rb\cap B_{1/2}\). Since \(|\gd' \varphi(x_0)|>0\), for large \(k\) we also have \(|\gd' \varphi(x_k)|>0\), which implies that by the super-solution condition of \(u_k\) at \(x_k\),
\[
1\ge |\gd\phi_k|^2(x_k)=1+2\varepsilon_k\partial_1\varphi(x_k) + o(\varepsilon_k),
\]
and so we obtain \(\partial_1\varphi(x_k)\le0\) for all large \(k\). This completes the proof.

\end{proof}

\begin{lemma}
    In the case that \(p\ne e_1\), then \(w\) is harmonic inside \(B_{1/2}\).
\end{lemma}

\begin{proof}
    This is immediate by observing that for any touching function \(\varphi\) from below, there is some \(\delta>0\) such that
    \[
    \lw p+C\varepsilon\gd\varphi\rw^2-\lw(p+C\varepsilon\gd\varphi)\cdot e_1\rw^2 \ge \delta,
    \]
    independent of small \(\varepsilon>0\).
\end{proof}

\appendix\section{Classification of homogeneous solutions in 2D}\label{Appendix.classification}

In this section, we discuss the classification of homogeneous solutions of the form
\[
u(r,\theta)=r^{\kappa} m(\theta),\,r>0,\,\kappa\ge0,\,\theta\in \pt B_1 \cap \{x_1\ge 0\}
\]
to the equation \eqref{e.grad-degen-neumann} in dimension \(d=2\). Before discussing the classification, let us notice that any homogeneous solutions as described above to the problem \eqref{e.grad-degen-neumann} are also homogeneous solutions to the no-sign Signorini problem \eqref{localproblemtoe}. This is because, by the boundary condition of \(u\), we know that \(m\) (similar to the boundary condition of \(h\) in \eqref{equationforthesphericalfunctionh}) satisfies 
\[
\min\left\{-\partial_{\Vec{n}}m(\theta), \sqrt{|\gd_\tau m(\theta)|^2 + m^2(\theta)}\right\}=0 \ \hbox{ for } \ \theta\in \partial' B_1'
\]
which implies that \(m\) also satisfies
\be\label{thespherconditionofnosignsignom}
\min\left\{-\partial_{\Vec{n}}m(\theta), |m|(\theta)\right\}=0 \ \hbox{ for } \ \theta\in \partial' B_1'.
\ee
Therefore, we only have to discuss the homogeneous solutions to \eqref{localproblemtoe}.  At this point, the analysis becomes very similar to the classical Signorini case, but we present the details anyway to be complete.

In dimension \(d=2\), we call \(x_1=y\) and \(x_2=x\), and take \(\theta\) to be the standard polar coordinate, i.e. $\tan \theta = \frac{y}{x}$. We can further write
\[
u(r,\theta)=r^{\kappa} m(\theta) \ \hbox{ for } \ r>0,\,\theta\in [0,\pi].
\]
Due to \eqref{thespherconditionofnosignsignom} and \(u\) being harmonic in \(B_1^+\), we know that \(m\) would satisfy
\[
\bca
m''(\theta)+\kappa^2m(\theta)=0,\hbox{   for }\theta \in (0,\pi)&\\
\min\{m'(0),|m|(0)\}=\min\{-m'(\pi),|m|(\pi)\}=0,&
\eca
\]
The general solution to the equation can be written as 
\[
m_{\textup{general}}=a \cos(\kappa\theta) + b \sin(\kappa\theta)
\]
for some real numbers \(a,b\).
By the boundary condition, we have (we can without loss assume that \(\kappa>0\)) 
\[
\min\{b,|a|\}=0,
\]
and
\[
\min\{ a\sin(\kappa\pi)-b\cos(\kappa\pi),|a\cos(\kappa\pi)+b\sin(\kappa\pi)|\}=0.
\]
Suppose \(b>0\) then \(|a|=0\) and we can further assume that \(b=1\) after normalization. This would then give us
\[
\min\{ -\cos(\kappa\pi),|\sin(\kappa\pi)|\}=0.
\]
Since \(\cos(\kappa\pi)\) and \(\sin(\kappa\pi)\) can not be both zero at the same time, we obtain either 
\[
\cos(\kappa\pi)<0,\,\sin(\kappa\pi)=0,
\]
or
\[
\cos(\kappa\pi)=0,\,|\sin(\kappa\pi)|>0.
\]
In the first case we have \(\kappa=2k-1,\,k\in\Z_+\), and in the second case we have \(\kappa=\frac{2k-1}{2},\,k\in \Z_+\). From this we get
\be\label{classifyminthefirst}
m_\kappa=\sin(\kappa\theta),\,\kappa=2k-1,\hbox{ or }(2k-1)/2,\,k\in\Z_+.
\ee
In the case \(b=0\) and \(|a|>0\), we can normalize so that \(|a|=1\), and obtain
\[
\min\{ \pm\sin(\kappa\pi),|\cos(\kappa\pi)|\}=0.
\]
This gives in the case \(a=1\), \(\kappa = \frac{4k-3}{2}\) and in the case \(a=-1\), \(\kappa=  \frac{4k-1}{2}\), or in both cases \(\kappa=k\) for \(k\in \Z_+\). From this we get
\be\label{classifyminthesecond}
m_\kappa= \cos(\kappa \theta),\,\kappa=\frac{4k-3}{2} \hbox{ or }k,\,k\in\Z_+,
\ee
or
\be\label{classifyminthethird}
m_\kappa= -\cos(\kappa \theta),\,\kappa=\frac{4k-1}{2}\hbox{ or }k,\,k\in\Z_+,
\ee
Now combining \eqref{classifyminthefirst}, \eqref{classifyminthesecond} and \eqref{classifyminthethird}, we have that any \(\kappa\)-homogeneous solution \(u_\kappa\) to \eqref{localproblemtoe} would take one of the following form with \(k\in \Z_+\)
\begin{enumerate}
    \item For \(\kappa=2k-1\) or \((2k-1)/2\)
    \[
    u_\kappa=\hbox{Im}\lb (x_2+i|x_1|)^\kappa\rb;
    \]
    \item For \(\kappa=\frac{4k-3}{2}\) or \(k\)
    \[
     u_\kappa=\hbox{Re}\lb (x_2+i|x_1|)^\kappa\rb;
    \]
     \item For \(\kappa=\frac{4k-1}{2}\) or \(k\)
    \[
     u_\kappa=-\hbox{Re}\lb (x_2+i|x_1|)^\kappa\rb.
    \]
\end{enumerate}
Notice that the cases \(\kappa = \frac{4k-2\pm1}{2} \) with \(u_\kappa=\mp\hbox{Re}\lb (x_2+i|x_1|)^\kappa\rb\) can be identified after reflections with the case \(\kappa=(2k-1)/2\) with \(u_\kappa=\hbox{Im}\lb (x_2+i|x_1|)^\kappa\rb\).

\bibliographystyle{plainnat}
\bibliography{ref}

\end{document}